\documentclass[11pt,reqno]{amsart}

\usepackage[final
]{showkeys}

\allowdisplaybreaks

\usepackage{AKstyle}
\usepackage[toc,page]{appendix}
\numberwithin{equation}{section}

\usepackage{amsmath}
\usepackage{caption}
\usepackage[labelfont=rm]{subcaption}
\usepackage{scalerel,stackengine}
\usepackage{enumitem}
\usepackage{relsize}
\stackMath
\newcommand\reallywidehat[1]{
\savestack{\tmpbox}{\stretchto{%

\usepackage{comments}

  \scaleto{%
    \scalerel*[\widthof{\ensuremath{#1}}]{\kern-.6pt\bigwedge\kern-.6pt}%
    {\rule[-\textheight/2]{1ex}{\textheight}}
  }{\textheight}%
}{.8ex}}%
\stackon[1pt]{#1}{\tmpbox}%
}
\usepackage[pagebackref=true, colorlinks]{hyperref}

\renewcommand{\phi}{\varphi}

\newcommand{\mat}[1]{\left(\begin{matrix} #1 \end{matrix} \right)}  

\newtheorem{thmx}{{\bf Theorem}}

\hypersetup{pdffitwindow=true,linkcolor=blue,citecolor=blue,urlcolor=blue,menucolor=blue}

\usepackage{comment}

\begin{document}

\author{Jincheng Tang}
\thanks{}
\email{tangent@connect.hku.hk}
\address{Department of Mathematics, The University of Hong Kong}

\author{Xin Zhang}
\thanks{Tang and Zhang are supported by ECS grant 27307320, GRF grant 17317222 and NSFC grant 12001457 from the second author.}
\email{xz27@hku.hk}
\address{Department of Mathematics, The University of Hong Kong}

\title[Super approximation for $\text{SL}_2\times \text{SL}_2$ and $\text{ASL}_2$]
{Super approximation for $\text{SL}_2\times \text{SL}_2$ and $\text{ASL}_2$}

\begin{abstract}
 Let $S\subset \text{SL}_2(\mathbb Z)\times \text{SL}_2(\mathbb Z)$ {or $\text{SL}_2(\mathbb Z)\ltimes \mathbb Z^2$}  be finite symmetric and assume $S$ generates a group $G$ which is a Zariski-dense subgroup $\text{SL}_2(\mathbb Z)\times \text{SL}_2(\mathbb Z)$ or $\text{SL}_2(\mathbb Z)\ltimes \mathbb Z^2$.  We prove that the Cayley graphs $$\{\mathcal Cay(G(\mod q), S (\mod q))\}_{q\in \mathbb Z}$$ form a family of expanders. 
\end{abstract}
\date{\today}
\maketitle

\section{Introduction}

Let $G=\langle S\rangle$ be a subgroup of $\text{SL}_n(\mathbb Z)$ with a finite symmetric generating set $S$.  For a positive integer $q$, let $G_q=G (\mod q)$ and $\Delta_q: L^2(G_q)\rightarrow L^2(G_q)$ be defined as, for any $f\in  L^2(G_q),$  
$$\Delta_q(f)(v):= \left(\frac{1}{|S|}\sum_{s\in S} f(\pi_q(s)\cdot v)\right)-f(v).$$
Each $\Delta_q$ is a self-adjoint operator, with spectrum 
$$0=\lambda_{0,q}>\lambda_{1,q}\geq \lambda_{2,q}\cdots$$
For a set $\cA$ of positive integers, we say $G$ has \emph{super approximation} with respect to $\cA$ if there is $\varepsilon>0$ such that $\lambda_{0,q}-\lambda_{1,q}>\varepsilon$, $\forall q\in \cA$.  The existence of a uniform positive spectral gap only depends on $G$ and is independent of the choice of the finite generating set $S$.  If $\mathcal A= \mathbb Z_+$, we simply say $G$ has super approximation. \par
It has been well known that lattices in semisimple Lie groups satisfy the super approximation property \cite{Sel65, Ma73, BurgerSarnak1991, Clozel2003}, but the involved techniques (spectral method, Property $T$, automorphic forms, etc.) only works for lattices and can not deal with a general non-lattice discrete group, which is also called a thin group. A breakthrough came from Bourgain and Gamburd, who developed the so-called ``Bourgain-Gamburd expansion machine", which is an analytic-combinatorial tool that allows them to prove super approximation property for any Zariski-dense subgroup of $\text{SL}_2(\mathbb Z)$ with respect to prime moduli \cite{BG08}. A critical ingredient in the proof is Helfgott's triple product theorem \cite{He08}. Since then, there has been a series of papers extending Bourgain-Gamburd's Theorem to more general groups with respect to more general moduli  \cite{BG08a, BG09, BGS10, GV12, VA12, BV12, BGGT15, PS16, SG19, HS21}.

Regarding the most general possible linear groups having super approximation, in \cite{GV12} Salehi-Golsefidy and Varj\'u conjecture:
\begin{conj}[Question 2, \cite{GV12}] \label{SAC}Let $G<{\normalfont \text{SL}}_n(\mathbb Z)$ be finitely generated, and $\mathbb G$ be the Zariski closure of $G$. Then $G$ has the super approximation property if and only if the identity component $\mathbb G_0$ of $\mathbb G$ is perfect, i.e. $[\mathbb G_0, \mathbb G_0]=\mathbb G_0$.
\end{conj}

In the same paper \cite{GV12}, Salehi-Golsefidy-Varj\'u proved that $G$ has super approximation with respect to square free numbers if $\mathbb G_0$ is perfect.  Later Salehi-Golsefidy generalized to bounded powers of square free numbers:

\begin{thmx}(Salehi-Golsefidy)\cite{SG19}\label{Golsefidy} \textit{Let $G<{\normalfont \text{SL}}_n(\mathbb Z)$ be finitely generated, and $\mathbb G$ be the Zariski closure of $G$. Then $G$ has the super approximation property with respect to bounded powers of square free integers if and only if the identity component $\mathbb G_0$ of $\mathbb G$ is perfect, i.e. $[\mathbb G_0, \mathbb G_0]=\mathbb G_0$.}
\end{thmx} 


In \cite{SG19}, the author also claims super approximation with respect to $\{p^n\}_{n\in \mathbb Z_+}$, based on a same claim for the special case that $\mathbb G$ is semisimple in \cite{SG17}. However, there is an issue in \cite{SG17} (see Remark \ref{0857}). Once this issue is resolved, we indeed have super approximation with respect to  $\{p^n\}_{n\in\mathbb Z_+}$, since the argument for extending from semisimple to perfect groups in \cite{SG19} is valid. \par

If we require no restriction on moduli, all known results have restrictions on the Zariski closure. In \cite{BV12} Bourgain-Varj\'u proved Conjecture \ref{SAC} when the Zariski closure of $\mathbb G_0$ is $\text{SL}_d$, using an extra ingredient from homogeneous dynamics \cite{BFLM11}. Later de Saxc\'e and He extended this dynamical tool \cite{He19}, which allowed them to prove Conjecture \ref{SAC} for all $G$ with $\mathbb{Q}$-simple closure \cite{HS21}.

The full generality of Conjecture \ref{SAC}, despite abundant evidence, has still remained a technical challenge. In the paper \cite{SG20}, for the purpose of proving Conjecture \ref{SAC} in general, Salehi-Golsefidy conjectured a sum-product phenomenon over finite quotients of rings of algebraic integers. In a separate paper \cite{TZ23a}, we proved this conjecture:
\begin{thmx}[Tang-Zhang, 2023]\label{sumproduct} \textit{ Suppose $0<\alpha< 1$, $d$ is a positive integer, and $N_0\gg_{d,\alpha} 1$ is a positive integer.  Then there are $0<\varepsilon: =\varepsilon(\alpha, d)$ and positive integers $C_1=C_1(\alpha, d), C_2=C_2(\alpha, d), C_3=C_3(\alpha, d)$ such that the following statement holds: Let $K$ be any number field of extension degree $[K:\mathbb Q]$ at most $d$, and $\mathcal O$ be the ring of integers of $K$.  Suppose $\mathfrak a$ is an ideal of $\mathcal O$ such that $N(\mathfrak a):= |\mathcal O/\mathfrak a |\geq N_0$, and suppose $A\subseteq \mathcal O$ such that 
$$|\pi_{\mathfrak a}(A)|\geq|\pi_{\mathfrak a}(\mathcal O)|^{\alpha}.$$
Then there are an ideal $\mathfrak a'$ of $\mathcal O$, and $a\in\mathcal O$ such that 
\begin{align*}
&\mathfrak a^{C_1}\subseteq \mathfrak a',\\
&\pi_{\mathfrak a'}(\mathbb Za)\subset \pi_{\mathfrak a'}\left(\sum_{C_3}A^{C_2} -\sum_{C_3}A^{C_2}\right),\\
& |\pi_{\mathfrak a'}(\mathbb Za)|\geq N(\mathfrak a)^\varepsilon.
\end{align*}}
\end{thmx}
Here, $\pi_\fa \text{ (resp. } \pi_{\fa'})$ is the reduction map $\cO\rightarrow \cO/\fa \text{ (resp. }\cO\rightarrow \cO/\fa' )$, the set $A^{C_2}=\{a_1a_2\cdots a_{C_2}: a_1, \cdots, a_{C_2}\in A\}$ is the $C_2$-fold product of the set $A$, and the set $\sum_{C_3}A^{C_2}=\{b_1+b_2+\cdots+b_{C_3}: b_1,\cdots, b_{C_3}\in A^{C_2}\}$ is the $C_3$-fold sum of the set $A^{C_2}$. 

Roughly speaking, Conjecture \ref{SAC} is a quantitative statement that given a set $A\subset \mathcal O$ and an ideal $\fa\subset \mathcal O$, if $|\pi_{\fa}(A)|$ is not too small, then modulo an ideal $\fa'$ comparable to $\fa$, a sum-product set of $A$ contains a thick arithmetic progression. With this extra ingredient, we can give a new proof of Conjecture \ref{SAC} in the case that $\text{Zcl}(G)=\text{SL}_2$, a result due to Bourgain-Varj\'u, following the strategy of \cite{BG09}. Combining Theorem \ref{sumproduct} with a gluing technique (Proposition \ref{glue}) which we think is our main technical contribution, we can prove Conjecture \ref{SAC} for two representative open cases of Conjecture \ref{SAC}, which is the main theorem of this paper:

\begin{theorem}\label{main}
Let $S$ be a finite symmetric set that generates a group $G$ which is a Zariski dense subgroup of $ {\normalfont \text{SL}}_2(\mathbb Z)\times {\normalfont \text{SL}}_2(\mathbb Z)$ or $ {\normalfont \text{SL}}_2(\mathbb Z)\ltimes \mathbb Z^2$, then $G$ has super approximation with respect to all positive integers. \end{theorem}

Theorem \ref{main} confirms Conjecture \ref{SAC} for the first semisimple but non-simple closure case, and the first non-semisimple closure case. \par

One advantage of our approach using sum-product, compared to the method in \cite{BV12} and \cite{HS21} is that we can allow denominators for the group. With little modification of the proof, we have the following extension of Theorem \ref{main}:

\begin{theorem}\label{main1}
Let $S$ be a finite symmetric set in $ {\normalfont \text{SL}}_2( \mathbb Z[1/q_0])\times {\normalfont \text{SL}}_2(\mathbb Z[1/q_0])$ $( {\normalfont \text{SL}}_2(\mathbb Z[1/q_0])\ltimes(\mathbb Z[1/q_0])^2$, respectively$)$, and $S$ generates a group $G$ which is a Zariski dense in $ {\normalfont \text{SL}}_2\times {\normalfont \text{SL}}_2,  ({\normalfont \text{ASL}}_2$, respectively$)$, then $G$ has super approximation with respect to $\{q\in \mathbb Z_+, (q,q_0)=1\}$. \end{theorem}

 \par

In \cite{GS24}, Salehi-Golsefidy and Srinivas establish a joint spectral gap theorem for random walks on $\text{SL}_2(\mathbb F_p)\times \text{SL}_2(\mathbb F_p)$, expressed in terms of the spectral gaps of the projections onto the simple factors. It is an interesting question whether this result can be extended to $\text{SL}_2(\mathbb Z/q\mathbb Z)\times \text{SL}_2(\mathbb Z/q\mathbb Z)$, which, together with \cite{BV12} would imply Theorem \ref{main} in the case of $\text{SL}_2(\mathbb Z)\times \text{SL}_2(\mathbb Z)$.


\noindent {\bf Acknowledgements} \par
We thank Zeev Rudnick, Nicolas de Saxc\'e, He Weikun, Tran Chieu-Minh and the anonymous referee for many helpful corrections/suggestions on previous versions of this paper. 


\section{Notations and methodology\label{method}}
We introduce the following notations which we use throughout this paper.  \par
The unit of any multiplicatively written group  is denoted by 1.  Occasionally, if a ring structure is present,  we denote the additive unit by 0.  For given two subsets $A$ and $B$, we denote their product set by $A\cdot B=\{ab|a\in A, b\in B\}$, and their sum set by $A+B=\{a+b| a\in A, b\in B\}$.  We let $A^k=\{a_1a_2\cdots a_k: a_1, a_2, \cdots, a_k\in A\}$, and $\sum_k A=\{a_1+a_2+\cdots+ a_k: a_1, a_2, \cdots, a_k\in A\}$.

 \par
 
 If $f$ and $g$ are two complex valued functions on a discrete group $G$, we denote by $f*g$ their convolution
$$f*g(x)=\sum_{y\in G}f(y)g(y^{-1}x). $$ \
We write $f^{(l)}$ for the $k$-fold convolution of $f$ with itself.  \par

For a prime $p$ and an integer $q$, we define $v_p(q)=n$ if $p^n| q$ but $p^{n+1}\nmid q$. Similarly, for a prime ideal $\cP$ and a general ideal $\fa$ of the ring of integers $\cO$ of a number field $K$, we define $v_{\cP}(\fa)=n$ if $\cP^n\supset \fa$ but $\cP^{n+1}\not\supset \fa$. For notational convenience by writing $v_{\cP}(a)$ for $a\in \mathcal \cO$ we mean $v_{\cP}((a))$.

The valuation function is naturally extended to matrices: for a matrix $M\in \text{Mat}_2 (K)$, and a prime ideal $\mathcal P$ in $\mathcal O$, $v_{\mathcal P}(M)=n$ if the minimum of $v_{\mathcal P}$-valuations of all coefficients of $M$ is $n$. 

The notion ``exact division'' is extremely helpful for the presentation in this paper. Given a natural number $q$, for each prime $p|q$, we write $p^n\Vert q$ if $v_p(q)=n$. Similarly, for two integers $q_1$ and $q_2$, we write $q_1\Vert q_2$ if for every $p|q_1$, we have $v_p(q_1)=v_p(q_2)$, and we call $q_1$ an exact divisor of $q_2$.  For two ideals $\mathcal Q_1, \mathcal Q_2$ of a ring of algebraic integers $\mathcal O$, we write $\mathcal Q_1\Vert \mathcal Q_2$ if for all prime ideals $\mathcal P\supset \mathcal Q_1$, we have $v_\cP (\cQ_1)=  v_{\cP} (Q_2)$.  \par 

For an ideal $\mathcal Q \subset \mathcal O$, we say $\mathfrak q\in\mathcal O$ is a \emph{uniformizer} of $\mathcal Q$ if $\frak q\mathcal O/\mathcal Q^2= \mathcal Q/\mathcal Q^2$.

For $q=\prod_{i\in I} p_i^{n_i}\in \mathbb Z_+$ and $\alpha$ a positive real number, we let $q^{\{\alpha\}}=\prod_i p_i^{[n_i\alpha]}$, where $[n_i\alpha]$ is the integer part of $n_i\alpha$. For any non-empty subset $U\subset I$, we write $q^U=\prod_{i\in U}p_i^{n_i}$, which is an exact divisor of $q$. \par
%

Let $\pi_q: \mathbb Z\rightarrow \mathbb Z/q\mathbb Z $ be the residue map, which induces residue maps in various other contexts, and we denote them by $\pi_q$ as well. 

Let $\Lambda=\text{SL}_2(\mathbb Z)\times \text{SL}_2(\mathbb Z)$ or $\text{SL}_2(\mathbb Z)\ltimes \mathbb Z^2$, $\Lambda_q= \Lambda( \mod q)$ and $\Lambda (q)$ be the kernel of the residue map $\pi_q: \Lambda\rightarrow \Lambda_q$. Throughout this paper, we fix $\Gamma=\text{SL}_2(\mathbb Z)$ and $V=\frak{sl}_2(\mathbb Z)=\{\gamma\in\text{Mat}_2(\mathbb Z): \text{tr}(\gamma)=0\}$. \par
\par
Sometimes we need to reduce the two factors of $\Gamma\times \Gamma $ by two different moduli $q_1, q_2$, and we denote this residue map by $\pi_{q_1, q_2}$. We also use $\pi_{q_1, q_2}$ to denote the reduction $\Gamma \ltimes \mathbb Z^2\rightarrow \Gamma_{q_1}\ltimes (\mathbb Z/q_2\mathbb Z)^2$ for $q_2|q_1$.  \par

We use the standard asymptotic notations $O, \Omega, \Theta$ to describe growth of functions. For two positive functions $f, g$ and a set of parameters, we write $f=O_X(g)$ if there is a positive function $C_1(X)$ of  the set of parameters such that $f\leq C_1(X) g$; we write $f=\Omega_X(g)$ if there is a positive function $C_2(X)$ such that $f\geq C_2(X)g$; we write $f=\Theta_X(g)$ if $f=O_X(g)$ and $f=\Omega_X(g)$. \par

 For convenience we also adopt Bourgain's notations: we write $f(q)<q^{c+}$ to mean $f(q)<q^{c+\varepsilon}$ for arbitrarily small $\varepsilon$ when $q$ large.  Similarly,  $f(q)>q^{c-}$ means $f(q)>q^{c-\varepsilon}$ for arbitrarily small $\varepsilon$ when $q$ large. \par

Let $\chi_S$ be the normalized uniform counting measure supported on $S$, i.e., for $A\subset  G$,  $\chi_S(A)=\frac{|A\cap S|}{|S|}$.  Let $\pi_q[\chi_S]$ be the pushforward of $\chi_S$ under the residue map $\pi_q$.  Let $T_q$ be the convolution operator by $\pi_q[\chi_S]$, i.e., For $f\in L^2(G_q)$, 
\begin{align}\label{1457}T_q(f)= \pi_q[\chi_S]* f.
\end{align}
  Then $T_q$ is a self adjoint operator on $L^2(G_q)$ with an invariant subspace $l_0^2(G_q)$ consisting of functions with average 0.  Denote the set of eigenvalues of $T_q$ on $l_0^2(G_q)$ by $E_q$. Then following the argument in \cite{BV12} (Proof of Theorem 1, Page 156-158), Theorem \ref{main} can be derived from the following proposition:

\begin{proposition}\label{2132}
Let $S$ be a finite symmetric set in $\Lambda= {\normalfont \text{SL}}_2(\mathbb Z)\times {\normalfont \text{SL}}_2(\mathbb Z)$ or $ {\normalfont \text{SL}}_2(\mathbb Z)\ltimes \mathbb Z^2$, and assumle that it generates a group $G$ which is Zariski-dense in ${\normalfont \text{SL}_2\times \text{SL}_2}$ or ${\normalfont \text{ASL}_2}$.  Then for any $\varepsilon>0$ there is $0<\delta<1$ depending on $\varepsilon$ such that the following holds.  If $q\in\mathbb Z_+$ sufficiently large , $A\subset \Lambda$ symmetric, and some integer $l$ satisfy 
\begin{align}\label{1131}
\chi_S^{(l)}(A)>q^{-\delta}, \hspace{5mm}  l>\delta^{-1}\log q  \hspace{5mm} \text{ and }\hspace{5mm} |\pi_q(A)|< |\Lambda_q|^{1-\varepsilon},
\end{align}
then 
\begin{align} \label{2320}
|\pi_q(A\cdot A\cdot A)|>|\pi_q(A)|^{1+\delta}. 
\end{align}
\end{proposition}

We recall how to deduce Theorem \ref{main1} from Proposition \ref{2132} using the argument from \cite{BV12}, Page 156-158. It follows from Alon and Milman \cite{AM85} that $G$ has super approximation with respect to $\mathbb Z_+$ if and only if there is some constant $c<1$ independent of $q$ such that $\lambda<c$ for all eigenvalues $\lambda \in E_q$.  Assume $S$ freely generates a group $G$; if not, applying Tits alternative \cite {Tits1972} or Proposition 20, \cite{GV12}, which generalizes Tits alternative from semisimple groups to perfect groups, we can find a finite symmetric set $S'$ which freely generates a subgroup $G'<G$, with $G'$ Zariski dense in $\Lambda$. Theorem \ref{main1} for $S'$ then implies Theorem \ref{main1} for $S$. 

Let $\lambda$ be an eigenvalue of $T_q$ and let $\mu$ be a corresponding eigenfunction. Consider the right regular representation of $\Lambda_q$ on $L^2(\Lambda_q)$. Let $\rho$ be the irreducible representation that contains $\mu$. Assume $\rho$ is faithful, i.e., $\rho$ is not induced from a subrepresentation of the right regular representation of $\Lambda_{q'}$ for some $q'<q$; otherwise, replace $q$ by $q'$.  According to Proposition 19, \cite{SG19}, if $\lambda$ is outside a finite set $\Sigma$, then $\rho$ has multiplicity at least $|\Lambda_q|^{\delta_0}$ for some positive $\delta_0$, and thus $\lambda$ has multiplicity at least $|\Lambda_q|^{\delta_0}$. We can now bound $\lambda^{2l}$ by computing the trace of $T^{2l}$ in the standard basis:
\begin{align}\label{0849}
\lambda^{2l}\leq |\Lambda_q|^{-\delta_0}\text{Tr}(T^{2l})= |\Lambda_q|^{1-\delta_0}\Vert \pi_q[\chi_S^{(l)}] \Vert_2^2.
\end{align}
The Kesten's bound implies there exists $l_0> c_0\log q$ for some small $c_0>0$, such that 
$$\Vert \pi_q[\chi_S^{(l_0)}] \Vert_2 <|\Lambda_q|^{-\varepsilon_0}$$
for some $\varepsilon_0>0$.

If $$\Vert \pi_q[\chi_S^{(l_0)}] \Vert_2 \geq |\Lambda_q|^{-\frac{1}{2}+\frac{\delta_0}{4}},$$ then Proposition \ref{2132} implies $$\Vert \pi_q[\chi_S^{(2l_0)}] \Vert_2<\Vert \pi_q[\chi_S^{(l_0)}] \Vert_2^{1-\delta} $$ for some $\delta>0$, with the help of a non-commutative analog of Balog-Szameredi-Gowers theorem due to Bourgain-Gamburd. See \cite{BV12} for details. 
Keep iterating the convolution until we obtain some $l<C_0\log q$ such that 
\begin{align}\label{0900}
\Vert \pi_q[\chi_S^{(l)}] \Vert_2 < |\Lambda_q|^{-\frac{1}{2}+\frac{\delta_0}{4}}.
\end{align}
Applying \eqref{0900} to \eqref{0849} leads to $\lambda < e^{-\frac{\delta_0}{C_0}}$ as desired.

\subsection{Sketch of proof for Proposition \ref{2132}}
We explain the case $\Lambda= \text{SL}_2(\mathbb Z)\times \text{SL}_2(\mathbb Z)$. The case $\Lambda=\text{SL}_2(\mathbb Z)\ltimes \mathbb Z^2$ follows in a similar fashion. Assuming \eqref{2320} does not hold for a sufficiently small $\delta$, we will show a bounded product of $A$ has size at least $|\Lambda_q|^{1-\varepsilon/2}$, which will force \eqref{2320} to hold after all. \par
Let $\mathbb P_1, \mathbb P_2$ be the projections of $\Lambda$ to its first and second components. If all prime divisors of $q$ have large exponents, we can follow closely the method in \cite{BG08a} and \cite{BG09} and apply Theorem \ref{sumproduct} to show that over one simple factor, say $\mathbb P_1 (\Lambda)$, there is a not-too-small exact divisor $q_*$ of $q$ (i.e., $\frac{\log q_*}{\log q}=\Omega_S(1)$), and some constant $C$ such that $\pi_{q_*}\circ \mathbb P_1(A^C)$ contains a large congruence subgroup of $\Gamma_{q_*}$ (see Proposition \ref{1631}). In this process, compared to \cite{BG09}, we apply Theorem \ref{Golsefidy} to take a shortcut to deduce some results on non-concentration of random walks at the Archimedean place. However, for the purpose of proving Proposition \ref{2132}, we need $q_*$ to be very large (i.e., $\frac{\log q_*}{\log q}\approx 1$). This leads to the following question: 
\begin{question}\label{814}
Suppose for two exact divisors $q_1=q^{I_1}, q_2=q^{I_2}$ of $q=\prod_{i\in I}p_i^{n_i}$, $I_1, I_2\subset I$, we have a symmetric set $B\subset \Gamma_{q_1}\times \Gamma_{q_2}$ such that $\mathbb P_1(B), \mathbb P_2(B)$ are large in $\Gamma_{q_1}, \Gamma_{q_2}$, can we boundedly generate a large subset of $\Gamma_{q_1}\times \Gamma_{q_2}$ by $B$? 
\end{question}
\begin{remark}\label{0857}
If $q_*$ is not large, we need to consider the case that $q_1, q_2$ are coprime for Question \ref{814} to grow $q_*$ to a larger divisor, under the reduction of which a large congruence subgroup can be boundedly generated by $B$. In addition, we also need to consider the case that $q_1, q_2$ are non-coprime when passing from one simple factor to another. In the paper \cite{SG17}, the author claims super approximation for a general group with semisimple closure with respect to $\{p^n\}_{n\in\mathbb Z_+}$. There is an issue, which already exists in the case $\text{SL}_2(\mathbb Z/p^n \mathbb Z)\times \text{SL}_2(\mathbb Z/p^n\mathbb Z)$. In the proof, The author arrives at the situation of Question \ref{814} for $q_1=q_2=p^n$. It appears the author's argument is that $\mathbb P_1(B), \mathbb P_2(B)$ being large implies $B$ is large, for which simple counter examples can be found (e.g. diagonal groups). This issue can be resolved if we also allow some extra help from the set $A$ given by Proposition \ref{2132}. This is good enough for proving Proposition \ref{2132}. \end{remark}
We develop a gluing tool (Proposition \ref{glue}) to address a more general situation of Question \ref{814}, which in our  view is a main technical contribution of this paper. We give a (more detailed) sketch of solution to Question \ref{814} here, which contains all essential ideas for Proposition \ref{glue}. Related to our application, we assume $n_i$ is large for each $i\in I_2$ (say, $n_i>\theta^{-2}$ for a parameter $\theta$ introduced below), and assume $\mathbb P_2(B)\subset \Gamma(\tilde q_2)/\Gamma(q_2)$, where $\tilde q_2$ is the square free part of $q_2$, so that $\Gamma(\tilde q_2)/\Gamma(q_2)$ is a product of $p$-groups, and the loss of size from $\Gamma_{q_2}$ to $\Gamma(\tilde q_2)/\Gamma(q_2)$ is insignificant. We do not have any such requirement for $q_1$, e.g. $q_1$ can be square free. 

 Without hiding key features, for simplicity we assume $\mathbb P_1(B)=\Gamma_{q_1}$ and $\mathbb P_2(B)=\Gamma(\tilde q_2)/\Gamma({q_2})$, where $\tilde q_2$ is the square free part of $q_2$. Our goal is to show that a bounded product of $B\cup A$ covers a large subset of $\Gamma_{q_1}\times \Gamma_{q_2}$.\par
 Consider a section map $$\psi: \Gamma_{q_1}\rightarrow B,$$ i.e., $\mathbb P_1\circ \psi$ is the identity map.  Take a small parameter $0<\theta<1$, e.g. $\theta=10^{-12}$. For each $i\in I_2$ (i.e. $p_i^{n_i}\Vert q_2$), we consider the map $\psi_i= \pi_{p_i^{[n_i\theta]}}\circ \mathbb P_2\circ \psi$. According to a dichotomy (Proposition \ref{BZq_coro6.9}), there are two scenarios: 
\begin{enumerate}
\item There exists $\mathcal G_j\subset \Gamma_{q_1}\times \Gamma_{q_1}, |\mathcal G_j|>10^{-4}|\Gamma_{q_1}|^2$, such that $\psi_j(xy)\neq \psi_j(x)\psi_j(y)$, $\forall (x,y)\in \mathcal G_j$. 
 \item  There exists $S_j\subset \Gamma_{q_1}, |S_j|\geq \frac{99}{100}|\Gamma_{q_1}|$, such that over $S_j$, $\psi_j$ agrees with a homomorphism $h_j: \Gamma_{q_1}\rightarrow \Gamma_{p_i^{[n_i\theta]}}$. 
 \end{enumerate}
 
Split $I_2=J_1\sqcup J_2$, where $J_1$ is the collection of indices when Case (1) happens, and $J_2$ is the complement. We further split $J_2= J_{21}\sqcup J_{22}$ where $J_{21}$ is the set of indices such that $h_j$ is trivial at the \emph{half} level, i.e. $\pi_{p_j^{[\frac{1}{2}{n_j\theta}]}}\circ h_j=1$, and $J_{22}$ is the complement. \par
\noindent {\bf Case 1}: $q^{J_1}\geq (q_2)^{\frac{1}{2}}$. A probability argument shows that there is a large subset $\mathcal G \subset \Gamma_{q_1}\times \Gamma_{q_1}$ (in fact $|\mathcal G|>q_2^{0-}|\Gamma_{q_1}|^2 $),  and $U_0\subset J_1$ with $q^{U_0}\geq (q^{J_1})^{\frac{1}{2}10^{-4}}$, such that $\psi_j(g_1)\psi_j(g_2)\neq \psi_j(g_1g_2), \forall (g_1,g_2)\in \mathcal G, \forall j\in U_0$. Take any $(g_1, g_2)\in \mathcal G$. Let $w=\psi(g_1)\psi(g_2)(\psi(g_1g_2))^{-1}$, and $w_B=\{[[w, b_1], b_2]: b_1, b_2\in B \}$. From Lemma \ref{2247} and Lemma \ref{1407}, we have $\mathbb P_1 ((w_B)^{N_1})=1$ and $\pi_{q^{U_0}}\circ \mathbb P_2 ((w_B)^{N_1})$ contains a large congruence subgroup of $\Gamma_{q^{U_0}}$, for some $N_1=O(\log \frac{1}{\theta})$. Therefore, $\pi_{q_1, q^{U_0}} (B\cdot (w_B)^{N_1})$ covers a large subset of $\Gamma_{q_1}\times \Gamma_{q^{U_0}}$. 

\noindent {\bf Case 2.1}: $q^{J_2}\geq (q_2)^{\frac{1}{4}}$. Similar to Case 1, we can find a large subset $W_1\subset \Gamma_{q_1}$, and $U_1\subset J_2$ with $q^{U_1}\geq (q^{J_2})^{\frac{99}{200}}$, such that $\pi_{p_j^{[\frac{1}{2}{n_j\theta}]}}(g)\circ \psi_j=1, \forall g\in W_1, \forall j \in U_1$. Then applying Lemma \ref{1407}, for some $N_2=O(\log\frac{1}{\theta})$, $\mathbb P_1 ([W_1, W_1]^{N_2})$ covers a large congruence subgroup of $\Gamma_{q_1}$, while $\pi_{p_j^{2[\frac{1}{2}n_j\theta]}}\circ \mathbb P_2 ([W_1, W_1]^{N_2})=1, \forall j\in U_1$, noticing that we have increased the exponent of $p_j$ by a factor of 2 after taking commutator. Keep iterating until we obtain a set $B_1\subset [W_1, W_1]^{O((\log \frac{1}{\theta})^2)}$ such that $\mathbb P_1(B_2)$ remains large in $\Gamma_{q_1}$, but $\mathbb P_2(B_1)=1(\mod q^{U_1})$. Then, $\pi_{q_1, q^{U_1}} (B\cdot B_1)$ covers a very large subset of $\Gamma_{q_1}\times \Gamma_{q^{U_1}}$.

\noindent {\bf Case 2.2}:  $q^{J_{21}}\geq (q_2)^{\frac{1}{4}}$. In this case, we can find a large set $W_2\subset \Gamma_{q_1}$ and $U_2\subset J_{21}$, such that over $W_2$, $\pi_{p_j^{[n_j\theta]}}\circ \mathbb P_2\circ \psi =h_j$, and $\pi_{p_j^{[\frac{1}{2}n_j\theta]}}\circ h_j\neq 1$, $\forall j\in U_2$. The set $W_2$ boundedly generates a large subgroup $G$ of $\Gamma_{q_1}$, assuming for simplicity that $G=\Gamma_{q_1}$. Then we can construct a section map 
$$\bar{\psi}: G \rightarrow W_2^{O(1)},$$ such that $\mathbb P_1\circ \bar\psi$ is the identity map, and $ \pi_{p_j^{[n_j\theta]}}\circ \mathbb P_2 \circ \bar \psi=h_j, \forall j\in U_2$.  
Utilizing that each $h_j, j\in U_2 $ is a homomorphism and nontrivial at half-level (which implies $U_2\subset I_1$), one can construct an element $g\in G$ such that the set $\{\bar \psi(g^n): n\in \mathbb N\}$ satisfies 
 \begin{align}\label{2071}
\nonumber & \pi_{q^{I_1-U_2}}\circ \mathbb P_1 (\bar \psi (g^n))=1 \\
 & \bar \psi (g^n)\equiv (1+ n Q_1 X_1,n Q_2 X_2) (\mod Q_1^2, Q_3) 
 \end{align}
where $X_1, X_2\in \text{Mat}_2(\mathbb Z)$ traceless and primitive, $Q_1, Q_2, Q_3$ are three divisors of $q^{U_2}$ such that $Q_1^2| q^{U_2}, Q_2| Q_3|Q_2^2 | q^{U_2}$, and $$v_{p_j}(Q_1), v_{p_j}(Q_2), v_{p_j}(Q_3), v_{p_j}(Q_3)-v_{p_j}(Q_2)=\Theta (n_j\theta), \hspace{0.5cm} \forall j\in U_2.$$
 Then \eqref{2071} gives a one-parameter group, whose projection to each simple factor is ``thick'' over the modulus $q^{U_2}$. We can then use random walks (Proposition \ref{decay2}) to find $g_1, g_2, g_3, g_4, g_5\in A$ to conjugate $X=(X_1, X_2)$ to different directions so that $$\text{gcd}(\text{Det}(X, g_1 Xg_1^{-1},g_2 Xg_2^{-1},g_3 Xg_3^{-1},g_4 Xg_4^{-1},g_5 Xg_5^{-1} ), q^{U_2}) $$ is small. 
Multiplying  $\{\bar \psi(g^n): n\in \mathbb N\}$ with all of its $g_i$-conjugates together, $1\leq i\leq 5$, and taking a further product set yield a set $B_2\subset \{B\cup A\}^{O(\log \frac{1}{\theta})}$ such that 
$\pi_{q^{I_1-U_2}}\circ \mathbb P_1 (B_2)=1$ and $\pi_{q^{U_2}, q^{U_2}} (B_2)$ is large in $\Gamma_{q^{U_2}}\times \Gamma_{ q^{U_2}} $.  Therefore, $\pi_{q_1, q^{U_2}}(B\cdot B_2)$ is a large subset in $\Gamma_{q_1}\times \Gamma_{q^{U_2}}$.  \par

In each of the above case, we manage to find a set from a bounded (depending on $\theta$) product of $B\cup A$ whose reduction by $(q_1, q^{U_i}), 0\leq i\leq 2$ is large, compared to the set $B$ whose reduction is known to be large at $(q_1, 1)$. It is noted that the progress of modulus increase, measured by $\frac{\log q^{U_i}}{\log q_2}$ is bounded below by an absolute positive constant. One can then iterate the above gluing process boundedly many times to improve the modulus of the second component to a very large exact divisor of $q_2$, thus giving an affirmative answer to Question \ref{814} (with extra help from $A$!).  \par

Regarding the proof of Proposition \ref{2132}, write $q=q_sq_l$ where all exponents of prime factors of $q_s$ are small and all exponents of prime factors of $q_l$ are large. If $q_s$ is not too small, we can first apply Theorem \ref{Golsefidy} to show $\pi_{q_s} (A)$ is big in $\Gamma_{q_s}\times \Gamma_{q_s}$, and iteratively apply Proposition \ref{1631} (which helps to create large sets under smaller moduli) and Proposition \ref{glue} (which helps to glu) to improve the moduli, until we find  a bounded product of $A$ covers a large congruence subgroup of $\Lambda_q$ of size at least $|\Lambda_q|^{1-\frac{\varepsilon}{2}}$. This forces \eqref{2320} to hold for a sufficiently small choice of $\delta$. If $q_s$ is too small, we can just work with $q_l$ and directly start the iterative process. 

\begin{remark} The discussions of Case 1, Case 2.1, Case 2.2 correspond to the analysis in \S \ref{0223}, \S \ref{0224} and \S \ref{0225} for the proof of Proposition \ref{glue}.
 {\bf Caution}: Some  notations appearing above, such as $\mathbb P_1, \mathbb P_2, Q_1, Q_2, Q_3, W_1, W_2, G$ have different meanings in \S \ref{gluing}. 
\end{remark}

\begin{remark} One can start reading from Section \ref{bg} and reference back to more or less standard lemmas in Sections \ref{cb} and \ref{rw} when necessary. If one is only concerned with a proof for Theorem \ref{main}, a simpler argument for Proposition \ref{1631} can be obtained using Theorem 1 of \cite{BV12} compared to the one given in this paper. See the argument for the $q_s>q^{\frac{\varepsilon}{2}}$ case in Section \ref{semisimple} using spectral gap. However, Theorem 1 in \cite{BV12} itself relies on a technically intricate dynamical tool from \cite{BFLM11}. We include a detailed proof of Proposition \ref{1631} utilising Theorem \ref{sumproduct} for two reasons: first, together Proposition \ref{glue} it gives an alternative proof for the $d=2$ case of Theorem 1 \cite{BV12} without relying on \cite{BFLM11}; second, this proof adapts easily to groups with denominators (Theorem \ref{main1}). 
\end{remark}


\section{Preliminaries on combinatorics\label{cb}}


The first ingredient is a generalization of Corollary 6.9 in \cite{Bou08} which proved the case $G_1, G_2$ abelian and $\text{gcd}(|G_1|, |G_2|)=1$.    
\begin{proposition}
	\label{BZq_coro6.9}
	Let $G_1, G_2$ be two finite groups and let $\psi: G_1 \rightarrow G_2$ be some map. Then for $0<\varepsilon<\frac{1}{1600}$ we have either
	\begin{equation}
		\label{BZq_6.10}
		\left|\left\{(x, y) \in G_1 \times G_1 \mid \psi(xy)=\psi(x)\psi(y)\right\}\right|<\left(1-\varepsilon\right)\left|G_1\right|^2, 
	\end{equation}
or there exists a subset $S\subset G_1$ with $|S|>(1-\sqrt{\varepsilon})|G_1|$ and a group homomorphism $f: G_1 \rightarrow G_2$ such that
	\begin{equation}
		\label{BZq_6.11}
		\psi|_S=f|_S.
	\end{equation}
	
\end{proposition}

%

Proposition \ref{BZq_coro6.9} relies on the following Lemma \ref{0955} and Theorem \ref{BZq_lemma6.5}.  

\begin{lemma}
	\label{0955}
	Let $A$ be a finite subset of a group $Z$ and $\mathcal{G} \subset A \times A$, $0<\varepsilon<1 / 4$, such that
	$$
	|\mathcal{G}|>(1-\varepsilon)|A|^2
	$$
	Then there exists a subset $A^{\prime}$ of $A$ satisfying
	$$
	\left|A^{\prime}\right|>(1-\sqrt{\varepsilon})|A|
	$$
	and
	$$
	\left|A^{\prime}A^{\prime}\right|<\frac{|A\stackrel{\mathcal{G}}{\cdot}A|^4}{(1-\sqrt{\varepsilon})(1-2 \sqrt{\varepsilon})^2|A|^3},
	$$
	where $A\stackrel{\mathcal{G}}{\cdot}A:=\{ab\mid a,b\in A, (a,b)\in \mathcal{G}\}$.
\end{lemma}

\begin{proof}[Proof of Lemma \ref{0955}] The proof was stated in \cite{Bou08} for an abelian group $Z$. The same proof works for a general group as well.
\end{proof}
\begin{theorem}[Noncommutative Freiman-Kneser theorem for small doubling]
	\label{BZq_lemma6.5}
	Let $0<\varepsilon \leq 1$, and let $S \subset G$ be a finite non-empty subset of a multiplicative group $G$ such that $|A \cdot S| \leq(2-\varepsilon)|S|$ for some finite set $A$ of cardinality $|A|$ at least $|S|$, where $A \cdot S:=\{a s: a \in A, s \in S\}$ is the product set of $A$ and $S$. Then there exists a finite subgroup $H$ of $G$ with cardinality $|H| \leq C(\varepsilon)|S|$, such that $S$ is covered by at most $C^{\prime}(\varepsilon)$ right-cosets $H \cdot x$ of $H$ for some $C(\varepsilon), C^{\prime}(\varepsilon)\leq {2}/{\varepsilon}-1$.
\end{theorem}
\begin{remark}
Theorem \ref{BZq_lemma6.5} is due to Hamidoune.  See the article ``Hamidoune's Freiman-Kneser theorem for nonabelian groups'' in Terence Tao's blog which gives a concise proof of Theorem \ref{BZq_lemma6.5}. 
\end{remark}

\begin{proof}[Proof of Proposition $\ref{BZq_coro6.9}$] We follow the method of Bourgain in his proof of Corollary 6.9 in \cite{Bou08}. Suppose $(\ref{BZq_6.10})$ fails, so that
$$
\mathcal{G}=\left\{(x, y) \in G_1 \times G_1 \mid \psi(xy)=\psi(x)\psi(y)\right\}
$$
satisfies
$$
|\mathcal{G}| \geq\left(1-\varepsilon\right)\left|G_1\right|^2 .
$$
Denote
$$
A=\left\{(x, \psi(x)) \mid x \in G_1\right\} \subset G_1 \times G_2,
$$
and  $$\mathcal{G}'=\{((x, \psi (x) ) ,(y, \psi(y))): (x,y)\in \mathcal G \}.$$ Then $|\mathcal G'|=|\mathcal G|\geq (1-\varepsilon)|G_1|^2=(1-\varepsilon)|A|^2$. Applying Lemma $\ref{0955}$ with $Z=G_1\times G_2$ to the set $A$, we obtain a subset $A^{\prime} \subset A$ satisfying
\begin{equation}
	\label{BZq_6.13}
	\left|A^{\prime}\right|>(1-\sqrt{\varepsilon})\left|A\right|= (1-\sqrt{\varepsilon})\left|G_1\right|
\end{equation}
and

\begin{equation}
	\label{BZq_6.14}
	\left|A^{\prime}A^{\prime}\right|<\frac{\left|G_1\right|}{\left(1-\sqrt{\varepsilon}\right)\left(1-2\sqrt{\varepsilon}\right)^2}<\left(1+10\sqrt{\varepsilon}\right)\left|A^{\prime}\right|<\frac{5}{4}\left|A^{\prime}\right|.
\end{equation}
Next, apply Theorem $\ref{BZq_lemma6.5}$ to $A^{\prime} \subset G_1 \times G_2$. There is a subgroup $H$ of $G_1 \times G_2$ and $\left(x_1, x_2\right) \in G_1 \times G_2$ such that
\begin{align}
	A^{\prime} &\subset\left(x_1, x_2\right)\cdot H \label{BZq_6.15},\\
	|H|& <\frac{5}{3}\left|A^{\prime}\right| \label{BZq_6.16}.
\end{align}
Let $\mathbb P_1$ be the projection map $G_1\times G_2\rightarrow G_1$. Let $H_1=\mathbb P_1\left(H\right)$. We have
$$
\left|H_1\right| \geq\left|\mathbb P_1\left(A^{\prime}\right)\right|=\left|A^{\prime}\right|\stackrel{(\ref*{BZq_6.13})}{>}\left(1-\sqrt{\varepsilon}\right)\left|G_1\right|>\frac{1}{2}\left|G_1\right|,
$$
implying that $H_1=G_1$. Then for any $x\in G_1,$ there exists $f(x)\in G_2$ such that $(x, f(x))\in H$. Assume there exists $(y, z_1), (y, z_2)\in H$ with $z_1\neq z_2$. Then for any $x\in G_1$, we have $(x, f(x)z_1z_2^{-1})\in H$ with $(x, f(x)z_1z_2^{-1})\neq (x, f(x))$. So $|H|\geq 2|G_1|\geq 2|A^\prime|$. Contradiction. Hence the choice of $f(x)$ is unique for all $x$. Since $H$ is a subgroup, we get $(1, 1)\in H$ so $ f(1)=1$. Also we see $(y_1, f(y_1))(y_2^{-1}, f(y_2^{-1}))=(y_1y_2^{-1}, f(y_1)f(y_2^{-1}))=(y_1y_2^{-1}, f(y_1y_2^{-1}))$ by the uniqueness of choice for $f(y_1y_2^{-1})$, so $f$ is a group homomorphism. Therefore, over $A'$ we have  \par
 \begin{align}\label{1736}
 (x,\psi(x))=(x_1, x_2)\cdot (x_1^{-1}x, f(x_1^{-1}x)) \Rightarrow \psi(x)= x_2f(x_1^{-1}x).
 \end{align}

Since $$\left|\mathbb P_1\left(A^{\prime}\right)\right|=|A'|>\left(1-\sqrt{\varepsilon}\right)\left|G_1\right|,$$
we deduce 
$$\left|\mathcal{G} \cap\left(\mathbb P_1\left(A^{\prime}\right) \times \mathbb P_1\left(A^{\prime}\right)\right)\right|\geq|\mathcal{G}|+|\mathbb P_1\left(A^{\prime}\right) \times \mathbb P_1\left(A^{\prime}\right)|-|G_1|^2>\left(1-2\sqrt{\varepsilon}\right)\left|G_1\right|^2$$
by inclusion-exclusion. Hence $$\left|\mathbb P_1\left(A^{\prime}\right)\stackrel{\mathcal{G}}{\cdot}\mathbb P_1\left(A^{\prime}\right)\right|\geq \frac{\left|\mathcal{G} \cap\left(\mathbb P_1\left(A^{\prime}\right) \times \mathbb P_1\left(A^{\prime}\right)\right)\right|}{|G_1|}> (1-2\sqrt{\varepsilon})\left|G_1\right|.$$ 
Therefore, $\mathbb P_1\left(A^{\prime}\right) \cap\left[\mathbb P_1\left(A^{\prime}\right) \stackrel{\mathcal{G}}{\cdot} \mathbb P_1\left(A^{\prime}\right)\right] \neq \emptyset$. For $\left(y_1, y_2\right) \in \mathcal{G} \cap\left(\mathbb P_1\left(A^{\prime}\right) \times \mathbb P_1\left(A^{\prime}\right)\right)$ with $y_1y_2 \in \mathbb P_1\left(A^{\prime}\right)$, we get
$$
x_2f(x_1^{-1}y_1y_2)\stackrel{\eqref{1736}}{=}\psi\left(y_1y_2\right)=\psi\left(y_1\right)\psi\left(y_2\right)\stackrel{\eqref{1736}}{=}x_2f(x_1^{-1}y_1)x_2f(x_1^{-1}y_2) .
$$
$$
\Rightarrow f(x_1)=x_2 \Rightarrow (x_1, x_2)\in H \Rightarrow A^\prime \subset H.
$$
We finish the proof of Proposition \ref{BZq_coro6.9} by taking $S=\mathbb P_1\left(A^{\prime}\right)$.
\end{proof}
We also need the following bounded generation result over congruence quotients of $\text{SL}_2(\mathbb Z)\times \text{SL}_2(\mathbb Z)$.  
\begin{proposition}
	\label{0643}
	For any $0<\delta<\frac{1}{25}$, the following holds: Let $A\subset \mathrm{SL}_2(\mathbb{Z}/q_1\mathbb{Z})\times \mathrm{SL}_2(\Z/q_2\mathbb{Z})$ be symmetric and $|A|>(q_1q_2)^{3-\delta}$.
	Then there exists $q_1'|q_1, q_2'|q_2, q_1'q_2'<(q_1q_2)^{80\delta}$ such that 
	\begin{align}
		&A^{5760}\supset \mathrm{SL}_2(q_1'\mathbb Z/q_1\mathbb Z)\times\mathrm{SL}_2(q_2'\mathbb Z/q_2\mathbb Z)\nonumber
		.
	\end{align}
\end{proposition}
 In \cite{He08} Helfgott proved that given $A\subset \text{SL}_2(\mathbb F_p)$ with $|A| > p^{\frac{8}{3}}$, then $A\cdot A\cdot A= {\normalfont \text{SL}}_2(\mathbb F_p)$ for a sufficiently large prime $p$.  Helfgott's result can be obtained in an elegant way by a representation theoretical approach due to Gowers \cite{Go08}, but Gower's approach does not seem to have an easy generalization for a general composite modulus.  We follow Helfgott's approach for the proof of Proposition \ref{0643}.

We start with a few lemmas. 
\begin{lemma}\label{1159} Let $0<\gamma<\frac{1}{4}$ and let $A, B\in \mathbb Z/q\mathbb Z$ with $|A|, |B|>q^{1-\gamma}$, then there exists $q'|q, q'<q^{\frac{12\gamma}{5}}$ such that
$$q'\mathbb Z/q\mathbb Z\subset \sum_{24}(AB-AB).$$
\end{lemma}
\begin{proof} This is a special case of Corollary A.7 in \cite{TZ23a}. 
\end{proof}
We need a two dimensional analog of Lemma \ref{1159}.  
\begin{lemma}\label{1241} Let $0<\delta<\frac{1}{20}$.  Let $q_1, q_2\in \mathbb Z$ and $A, B$ be subsets of $\mathbb Z/q_1\mathbb Z\times \mathbb Z/q_2\mathbb Z$ such that $|A|, |B|> (q_1q_2)^{1-\delta}$. Then there exists $q_1'|q_1, q_2'|q_2, q_1'q_2'<(q_1q_2)^{10\delta}$, such that 
\begin{align}\label{1232}
\sum_{96}(AB-AB)\supset q_1'\mathbb Z/q_1\mathbb Z\times q_2'\mathbb Z/q_2\mathbb Z.
\end{align}
\end{lemma}
\begin{proof}
Let $\mathbb P_i, i=1, 2$ be the projection from $\mathbb Z/q_1\mathbb Z\times \mathbb Z/q_2\mathbb Z$ to the $i$-th component. Without loss of generality we assume $q_1=q_2^\alpha$ for some $0<\alpha\leq 1$.  Since $|A|>(q_1q_2)^{1-\delta}$, there is $x_0\in A$ such that 
$$|\{x\in A: \mathbb P_1(x)=x_0\}|> (q_1q_2)^{1-\delta}/q_1> q_2^{1-\delta-\alpha\delta},$$
so there is $A'\subset A-A$, $|A'|> q_2^{1-\delta-\alpha\delta}$, $\mathbb P_1(A')=\{0\}$. \par
Similarly, there is $B'\subset B-B$, $|B'|> (q_1q_2)^{1-\delta}/q_1> q_2^{1-\delta-\delta\alpha}$, $\mathbb P_1(B')=0$.  Applying Lemma \ref{1159}, we obtain $q_2'|q_2, q_2'<q_2^{\frac{12(\delta+\alpha\delta)}{5}}$ such that 
\begin{align}\label{1218}\sum_{48} (AB-AB)\supset \sum_{24}(A'B'-A'B')\supset \{0\} \times q_2'\mathbb Z/q_2\mathbb Z. 
\end{align} \par
If $\alpha < 5 \delta$, then one can take $q_1'=q_1$ and we have $q_1'q_2'< q_2^{10\delta}$, and \eqref{1218} gives Lemma \ref{1241}. \par
If $\alpha>5\delta$, then there exists $A''\subset A-A, B''\subset B-B$, such that $|A''|, |B''|> q_1^{1-\delta-\frac{\delta}{\alpha}}, \mathbb P_2 (A'')=\mathbb P_2 (B'')=\{0\}$.  The exponent $1-\delta-\delta/\alpha$ exceeds $3/4$, so applying  Lemma \ref{1159}, we obtain $q_1'|q_1$, $q_1'<q_1^{\frac{12(\delta+\frac{\delta}{\alpha})}{5}}$, such that 
\begin{align}\label{1219}\sum_{48} (AB-AB)\supset \sum_{24}(A''B''-A''B'')\supset q_1'\mathbb Z/q_1\mathbb Z \times \{0\}. 
\end{align}
Adding \eqref{1218} and \eqref{1219}, we obtain Lemma \ref{1241}, with 
$$q_1'q_2'<q_1^{\frac{12(\delta+\frac{\delta}{\alpha})}{5}}q_2^{\frac{12(\delta+\alpha\delta)}{5}}=(q_1q_2)^{\frac{24\delta}{5}}.$$
\end{proof}

\begin{proof}[Proof of Proposition \ref{0643}]  Since $|A|>(q_1q_2)^{3-\delta}$, by the pigeon hole principle, there exists $\vec v\in (\mathbb Z/q_1\mathbb Z)^2, \vec w\in (\mathbb Z/q_2\mathbb Z)^2$, such that the cardinality of the set 
$$\{(\gamma_1, \gamma_2): \gamma_1 \text{ has lower row }\vec{v}, \gamma_2 \text{ has lower row }\vec{w}\}$$
exceeds $(q_1q_2)^{1-\delta}$.  This implies the cardinality of the set 
$$A_1:=A\cdot A^{-1}\cap \left\{ \left(\mat{1&m\\0&1}, \mat{1&n\\0&1}\right): m\in \mathbb Z/q_1\mathbb Z, n\in \mathbb Z/q_2\mathbb Z \right\}$$
exceeds $(q_1q_2)^{1-\delta}$. \par
Similarly, if we let 
$$A_2:=A\cdot A^{-1}\cap \left\{ \left(\mat{1&0\\m&1}, \mat{1&0\\n&1}\right): m\in \mathbb Z/q_1\mathbb Z, n\in \mathbb Z/q_2\mathbb Z \right\}.$$
Then, $|A_2|>(q_1q_2)^{1-\delta}$. \par
Next, we define an equivalence relation $\sim$ on $\text{SL}_2(\mathbb Z/q_1\mathbb Z)\times \text{SL}_2(\mathbb Z/q_2\mathbb Z)$ as: 
$(\gamma_1, \gamma_2)\sim (\gamma_1', \gamma_2')$ if and only if the second rows of $\gamma_1$, $\gamma_1'$ are the same up to a scaler multiple in $(\mathbb Z/q_1\mathbb Z)^*$, and the second rows of $\gamma_2$, $\gamma_2'$ are the same up to a scaler multiple in $(\mathbb Z/q_2\mathbb Z)^*$.  There are at most $(q_1q_2)^{1+}$ many such classes.  By pigeon hole, there exists one class which contains at least $(q_1q_2)^{2-\delta-}$ many elements from $A$.  This implies the set 
$$H_0=A\cdot A^{-1}\cap \left\{ \left(\mat{\lambda_1&x\\0&\lambda_1^{-1}}, \mat{\lambda_2&y\\0&\lambda_2^{-1}}\right) : \lambda_1\in (\mathbb {Z}/q_1\mathbb Z)^*, x\in \mathbb {Z}/q_1\mathbb Z, \lambda_2\in (\mathbb {Z}/q_2\mathbb Z)^*,  y\in \mathbb {Z}/q_2\mathbb Z\right\}$$
has cardinality exceeding $ (q_1q_2)^{2-\delta-}$.
By pigeon hole again, there exists $x_0\in \mathbb {Z}/q_1\mathbb Z, y_0\in \mathbb {Z}/q_2\mathbb Z$ such that 
$$H=A\cdot A^{-1}\cap \left\{ \left(\mat{\lambda_1&x_0\\0&\lambda_1^{-1}}, \mat{\lambda_2&y_0\\0&\lambda_2^{-1}}\right) : \lambda_1\in (\mathbb {Z}/q_1\mathbb Z)^*, \lambda_2\in (\mathbb {Z}/q_2\mathbb Z)^* \right\}$$
has cardinality $> (q_1q_2)^{1-\delta-}$.\par
We have the following elementary computation 
\begin{align}\label{0845}
\mat{\lambda&x_0\\0&\lambda^{-1}}\mat{1&x\\0&1} \mat{\lambda&x_0\\0&\lambda^{-1}}^{-1}=\mat{1&\lambda^2 x\\0&1}.
\end{align}

Using \eqref{0845} and applying Lemma \ref{1241} to the sets $$\left\{(\lambda_1^2, \lambda_2^{2}):  \left(\mat{\lambda_1&x_0\\0&\lambda_1^{-1}}, \mat{\lambda_2&y_0\\0&\lambda_2^{-1}}\right) \in H   \right\}$$
and $$\left\{(m, n):\left(\mat{1&m\\0&1}, \mat{1&n\\0&1}\right)\in A_1 \right\}$$ with exponent $1-2\delta$, we obtain $Q_1|q_1, Q_2|q_2$, $Q_1Q_2<(q_1q_2)^{20\delta}$ such that 
$$A^{1152}\supset \left\{ \left( \mat{1&Q_1\mathbb Z/q_1\mathbb Z\\0&1}, \mat{1&Q_2\mathbb Z/q_2\mathbb Z\\0&1 } \right)   \right\}$$
Similarly, we can obtain $Q_1'|q_1, Q_2'|q_2$, $Q_1'Q_2'<(q_1q_2)^{20\delta}$ such that 
$$A^{1152}\supset \left\{ \left( \mat{1&0\\Q_1'\mathbb Z/q_1\mathbb Z&1}, \mat{1&0\\Q_2'\mathbb Z/q_2\mathbb Z&1 } \right)   \right\}$$
Let $Q_1^{*}=\text{lcm}(Q_1, Q_1')$, $Q_2^{*}=\text{lcm}(Q_2, Q_2')$. 
It is an elementary exercise to check that for $m\leq {n}/{2}$, any element of the group 
$$\left\{\mat{a&b\\c&d}\in \text{SL}_2(\mathbb Z/p^n\mathbb Z): a, d\equiv 1\ (\mod p^{2m }), b, c\equiv 0\ (\mod p^{2m})\right\}$$ 
can be written as $a_1b_1a_2b_2a_3$, where $a_1,a_2,a_3\in \mat{1& p^m \mathbb Z/p^n\mathbb Z\\0&1}$ and $b_1,b_2\in \mat{1&0\\ p^m \mathbb Z/p^n\mathbb Z&1}$. 
From this it follows that if we let $q_1'= \text{gcd}((Q_1^*)^2, q_1)$, $q_2'= \text{gcd}((Q_2^*)^2, q_2)$, then $q_1'q_2'\leq q^{80\delta}$, and $$A^{5760}\supset \mathrm{SL}_2(q_1'\mathbb Z/q_1\mathbb Z)\times \mathrm{SL}_2(q_2'\mathbb Z/q_2\mathbb Z).$$ \end{proof}

Following the same method, we can prove
\begin{prop}
\label{1433}
	For any $0<\delta<\frac{1}{10}$, the following holds: Let $A\subset \mathrm{SL}_2(\mathbb{Z}/q\mathbb{Z})$ be symmetric and $|A|>q^{3-\delta}$.
	Then there exists $q'|q, q'<q^{20\delta}$ such that 
	\begin{align}
		&A^{1440}\supset \Gamma(q')/\Gamma(q).\nonumber
	\end{align}
\end{prop}
We need Proposition \ref{1433} to prove the following bounded generation result for $\text{ASL}_2$.
\begin{proposition}\label{22311}For any $0<\delta<10^{-4}$, the following holds: Let $q_2|q_1$ and let $A\subset \mathrm{SL}_2(\mathbb{Z}/q_1\mathbb{Z})\ltimes (\Z/q_2\mathbb{Z})^2$ be symmetric and $|A|>(q_1^3q_2^2)^{1-\delta}$.
	Then there exists $q_1'|q_1, q_2'|q_2, q_1'<q_1^{50\delta}, q_2'<q_1^{55\delta}$, such that 
	\begin{align}
		&A^{14404}\supset \Gamma(q_1')/\Gamma(q_1)\ltimes (q_2'\mathbb Z/q_2\mathbb Z)^2.\nonumber
		\end{align}
\end{proposition}

\begin{proof} 
Define an equivalence relation on $A$: \[
(\gamma_1, v_1)\sim (\gamma_2, v_2) \text{ if and only if }\gamma_1^{-1}v_1 = \gamma_2^{-1}v_2.
\]
Since $|A|>(q_1^3q_2^2)^{1-\delta}$ and the number of equivalence classes is $q_2^2$, there is one class $B$ such that $|B|>q_1^{3(1-\delta)}q_2^{-2\delta}>q_1^{3-5\delta}$. Take any $b\in B$, then $Bb^{-1}\subset \{(\gamma,0): \gamma\in \text{SL}_2(\mathbb Z/q_1\mathbb Z)\}$ and $|Bb^{-1}|>q_1^{3-5\delta}$. Applying Proposition \ref{1433}, we have
\begin{align}\label{12211}
A^{2880}\supset (Bb^{-1})^{1440}\supset \Gamma(q_1')/\Gamma(q_1)\ltimes\{0\}
\end{align}
for some $q_1'<q_1^{50\delta}$.  \par
Since $|A|>(q_1^3q_2^2)^{1-\delta}$, by the pigeon hole principle, there exists $\gamma_0\in \mathrm{SL}_2(\mathbb Z/q_1\mathbb Z)$ such that $B'=\{(\gamma_0, w), w\in \mathbb Z/q_2\mathbb Z\}\cap A$ has at least $\frac{|A|}{q_1^3}=q_1^{-3\delta}q_2^{2-2\delta}$ elements. Take any $b'\in B'$. It follows that there exists $(1, w)\in B'b'^{-1}$, $w=q'w_0$, where $q'|q_2, q'\leq q_1^{3\delta}q_2^{2\delta}$ and $w_0$ is a primitive vector in $(\mathbb Z/q_2\mathbb Z)^2$. \par
It is an elementary exercise to check that for any primitive vector $v\in (\mathbb Z/p^n\mathbb Z)^2$, any $m\leq n$,
\begin{align}\label{12151}
\Gamma(p^m)/\Gamma(p^n)v-\Gamma(p^m)/\Gamma(p^n)v\supset (p^m\mathbb Z/p^n\mathbb Z)^2. 
\end{align}
Let $R\subset A^{5762}$ be the collection of conjugates of $(1,w)$ by $\Gamma(q_1')/\Gamma(q)\ltimes\{0\}$. Using \eqref{12151}, we obtain 
\begin{align}\label{12212}
A^{11524}\supset RR^{-1}\supset (1, \text{gcd}(q_1'q', q_2)\mathbb Z/q_2\mathbb Z)
\end{align}
Multiplying the RHS of \eqref{12211} and the RHS of \eqref{12212} we obtain Proposition \ref{22311} with $q_2'=\text{gcd}(q_1'q', q_2)$. 

\end{proof}

Let $u, v\in\mathbb Q$ and $f_{u, v}: \mathbb{Z}^6\rightarrow \mathbb{Z}^2$ be given by $$f_{u, v}(A, B, C, D, E, F)=\left(\mat{A&B\\C&D}-1\right)\mat{u\\v}-\mat{E\\F}.$$ It is straightforward that  
\[
H_{u, v}:=\left\{\left(\mat{A&B\\C&D},\mat{E\\F}\right)\in {\normalfont \text{ASL}}_2(\mathbb Z) \bigg|f_{u, v}(A,B,C,D,E,F)=0\right\}.
\]
is a subgroup of $\text{SL}_2(\mathbb Z)\ltimes \mathbb Z^2$ as $H_{u,v}$ is the stabilizer of $(u,v)^t$. 

The following proposition shows if a section map from (a large subgroup of) $\text{SL}_2(\mathbb Z/p^n\mathbb Z)$ to $\text{SL}_2(\mathbb Z/p^n\mathbb Z)\ltimes (\mathbb Z/p^n\mathbb Z)^2$ is a homomorphism, then the image must be close to $H_{u,v}$ for some $u, v$.  \par

\begin{proposition}\label{1744}
Let $n\geq 5$, and let $H$ be a subgroup of $\Gamma(p^m)/\Gamma(p^n) \ltimes (\mathbb Z/p^{n-m}\mathbb Z)^2$ such that $\mathbb{P}_0(H)=\Gamma(p^m)/\Gamma(p^n)$ and $|H|=|\mathbb P_0(H)|$, where we require $1\leq m\leq \frac{n}{4}$ if $p\neq 2$, and $2\leq m\leq  \frac{n}{4}$ if $p\neq 2$. Then there exist some $u, v\in \frac{1}{p^m}\Z$ such that
    \[
    H\cap \left( \Gamma(p^{2m})/\Gamma(p^n) \ltimes (\mathbb Z/p^{n-m}\mathbb Z)^2 \right) (\mod p^{[  \frac{n}{4}]+\epsilon(p) })\subset H_{u, v}(\mod p^{[  \frac{n}{4} ]+\epsilon(p) }),    \] 
where $\epsilon(p)=0$ if $p\neq 3$ and $\epsilon(p)=-1$ if $p=3$. 

\end{proposition}
\begin{proof} Proposition \ref{1744} follows from the proofs of a series of claims.  \par
\textbf{Claim 1:} Suppose $(g ,w)\in H$. If $p^k|g-1$, then $p^{k-m}| w$. 

\textit{Proof of Claim 1.} We prove by backward induction on $k$. When $k=n$, we must have $w=0$ by the assumptions on $H$. Now assuming Claim 1 is true for $k+1$ for some $2\leq k<n$ if $p\neq 2$, or $3\leq k<n$ if $p=2$, we show that Claim 1 also holds for $k$.  Write $g=1+p^kM$. Since for $t\in \mathbb{N}$, $g^t\equiv 1+tp^kM \pmod{p^{\min\{2k, n\}}}$, we have
\[
(g ,w)^p=(g^p, (1+g+g^2+\cdots+g^{p-1})w)\]
with $g^p\in \Gamma(p^{k+1})/\Gamma(p^n)$ and 

$$(1+g+g^2+\cdots+g^{p-1})w \equiv  p(1+\frac{(p-1)}{2}p^k M)w (\mod{p^{\min\{2k, n-m\}}}).$$

By the induction hypothesis, 
\[
p(1+\frac{p-1}{2}p^kM)w\in (p^{k-m+1}\mathbb{Z}/p^{n-m}\mathbb{Z})^2.
\]
Since $1+\frac{p-1}{2}p^kM$ is non-singular $\mod p$, we get $$pw\in (p^{k-m+1}\mathbb{Z}/p^{n-m}\mathbb{Z})^2\Rightarrow w\in (p^{k-m}\mathbb{Z}/p^{n-m}\mathbb{Z})^2.$$ This proves Claim 1.

 \textbf{Claim 2:} If $$A_1=(g_1,w_1)=\left(\mat{1&p^k\\0&1}, \mat{p^{k-m}a_1\\p^{k-m}b_1}\right)\in H$$ and $$A_2=(g_2,w_2)=\left(\mat{1&0\\p^k&1}, \mat{p^{k-m}a_2\\p^{k-m}b_2}\right)\in H,$$ with $m\leq k\leq \frac{n}{4}$, then $a_2, b_1\equiv 0\pmod{p^{k}}$ if $p\neq 3$, and $a_2, b_1\equiv 0\pmod{p^{k-1}}$ if $p=3$.

\textit{Proof of Claim 2.} We compute $$A_3=A_1A_2A_1^{-1}A_2^{-1}\equiv \left(\mat{1+p^{2k}&0\\0&1-p^{2k}},p^{2k-m}\mat{b_2\\-a_1}\right) (\mod p^{3k}, p^{3k-m})$$
and 
    \[
 A_4=A_3A_1A_3^{-1}A_1^{-1}\equiv \left(\mat{1&2p^{3k}\\0&1}, p^{3k-m}\mat{2a_1\\-b_1}\right)(\mod p^{4k}, p^{4k-m}).
    \]
    On the other hand, 
   \begin{align*}
   A_5=A_1^{2p^{2k}} \equiv\left(\mat{1&2p^{3k}\\0&1}, p^{3k-m}\mat{2a_1\\2b_1}\right)(\mod p^{4k}, p^{4k-m}).
   \end{align*}
 Applying Claim 1 to $A_4A_5^{-1}$, we obtain $3b_1\equiv 0(\mod p^k)$. Similarly, $3a_2\equiv 0(\mod p^k)$. Claim 2 thus follows. 
   
\textbf{Claim 3:} Let $A_6=\left(\mat{1&p^m\\0&1}, \mat{ u_1\\  v_1} \right)\in H$. Then $v_1\equiv 0\pmod{p^{[  \frac{n}{4} ]}}$ if $p\neq 3$, and $v_1\equiv 0\pmod{p^{[  \frac{n}{4} ]-1}}$ if $p=3$.

\textit{Proof of Claim 3.} consider
    \[
   A_6^{p^{t}}=\left(\mat{1&p^{m+t}\\0&1}, \mat{\cdots \\p^{t}v_1}\right)\in H.
    \]
    If taking $t=[  \frac{n}{4} ]-m$, by Claim 2 we get $v_1\equiv 0 \pmod{p^{ [  \frac{n}{4} ]}}$ if $p\neq 3$, and $v_1\equiv 0 \pmod{p^{ [  \frac{n}{4} ]-1}}$ if $p=3$.   This finishes the proof of Claim 3. \par
    By the same consideration, we obtain
    
    \textbf{Claim 4:} Let $\left(\mat{1&0\\p^m&1}, \mat{u_2\\  v_2} \right)\in H$, then $u_2\equiv 0 \pmod{p^{ [  \frac{n}{4} ]}}$ if $p\neq 3$, and $u_2\equiv 0 \pmod{p^{ [  \frac{n}{4} ]-1}}$ if $p=3$.

Proposition \ref{1744} thus follows by taking $u=\frac{1}{p^m}v_2$ and $v=\frac{1}{p^m}u_1$, since 
\[
\left(\mat{1&p^m\\0&1},\mat{ u_1\\ \cdot}\right), \left(\mat{1&0\\p^m&1},\mat{\cdot \\ v_2}\right)\in H_{u, v} \begin{cases} (\mod p^{[ \frac{n}{4}]}) & \text{if } p\neq 3 \\ (\mod p^{[ \frac{n}{4}]-1}) & \text{if } p= 3 \end{cases},
\]
and that $\mat{1&p^m\\0&1}$ and $\mat{1&0\\ p^m&1}$ generate $\Gamma(p^{2m})/\Gamma(p^n)$.
\end{proof}

Finally, we record some elementary combinatorics on $\text{SL}_2$. 
\begin{lemma}\label{1521} Let $H_1, H_2 \subset \normalfont\text{SL}_2(\mathbb Z)$ and let $p$ be a prime.  Suppose $1+p^{m_1}V (\mod p^{m_2})\subset H_1 (\mod p^{m_2})$ and $1+p^{n_1}V (\mod p^{n_2})\subset H_2 (\mod p^{n_2})$ with $1\leq m_1\leq m_2\leq 2m_1$ and $1\leq n_1\leq n_2\leq 2n_1$.  Then 
\begin{equation}
\label{152302}
1+ 2p^{m_1+n_1}V \subset [H_1, H_2]^3 (\mod p^{\min\{m_2+n_2-1, 2m_1+n_1, m_1+2n_1\}}).
\end{equation}
\end{lemma}

To prove Lemma \ref{1521}, we first need

\begin{lemma}
    \label{152301}
    Let $p$ be a prime. Let $A, B\subset \mathbb{Z}, m_1, m_2\in\mathbb{N}$ such that 
    \begin{align*}
        &\pi_{p^{m_1}}(A)\supset \mathbb{Z}/p^{m_1}\mathbb{Z} \\
        &\pi_{p^{m_2}}(B)\supset \mathbb{Z}/p^{m_2}\mathbb{Z}.
    \end{align*}
    Then 
    \begin{equation*}
         \pi_{p^{m_1+m_2-1}}\left(\sum_{2}AB\right)\supset \mathbb{Z}/p^{m_1+m_2-1}\mathbb{Z}.
    \end{equation*}
\end{lemma}
\begin{proof}[Proof of Lemma \ref{152301}]
Take $a\in A$ such that $v_p(a)=m_1-1$, then 
\begin{align}\label{0438}
aB(\mod p^{m_1+m_2-1}\mathbb Z)= p^{m_1-1}\mathbb Z/  p^{m_1+m_2-1}\mathbb Z.
\end{align}
Take $b\in B$ such that $v_p(b)=0$, then 
\begin{align}\label{0439}
Ab (\mod p^{m_1-1}\mathbb Z)= \mathbb Z/p^{m_1-1}\mathbb Z.
\end{align}
\eqref{0438} and $\eqref{0439}$ then implies $aB+Bb(\mod p^{m_1+m_2-1})= \mathbb Z/p^{m_1+m_2-1}\mathbb Z$.    
\end{proof}
To proceed, we also need the following well known identity (see Lemma 6.1, \cite{BG09}).  
\begin{lemma} \label {1946} Let $m,m'\in \mathbb Z_+$, and let $x, y\in \normalfont\text{SL}_2 (\mathbb Z)$, $x\equiv 1 (\mod p^m), y\equiv 1 (\mod p^{m'})$. Then 
$$xyx^{-1}y^{-1}\equiv 1 (\mod p^{m+m'}),$$
and 
$$xyx^{-1}y^{-1}\equiv 1+xy-yx (\mod p ^{m+m'+\min \{m, m'\}})$$
\end{lemma}

\begin{proof}[Proof of Lemma \ref{1521}]
     Take $g_1\in H_1$ and $g_2\in H_2$. Write 
    \begin{align*}
        g_1=1+p^{m_1}\begin{pmatrix}
          a_1 & b_1\\c_1 & -a_1
      \end{pmatrix}+p^{2m_1}M_1,\\
        g_2=1+p^{m_2}\begin{pmatrix}
          a_2 & b_2\\c_2 & -a_2
      \end{pmatrix}+p^{2m_2}M_2,
    \end{align*}
    with $a_1,b_1,c_1\in[0, p^{m_1}-1]$, $a_2,b_2,c_2\in[0, p^{m_2}-1]$, and $M_1, M_2\in \normalfont\text{Mat}_2(\mathbb Z)$.

    By Lemma \ref{1946}, we get
    \begin{align}
        g_1g_2g_1^{-1}g_2^{-1}& \equiv 1+[g_1, g_2]\equiv 1+p^{m_1+n_1}\bigg[\begin{pmatrix}
          a_1 & b_1\\c_1 & -a_1
      \end{pmatrix}, \begin{pmatrix}
          a_2 & b_2\\c_2 & -a_2
      \end{pmatrix}\bigg] \nonumber \\
      & = 1+p^{m_1+n_1}\begin{pmatrix}
          b_1c_2-b_2c_1 & 2(a_1b_2-b_1a_2)\\
          2(c_1a_2-a_1c_2) & -(b_1c_2-b_2c_1)\\
      \end{pmatrix} \pmod{p^{m_1+n_1+\min\{m_1,n_2\}}}. \nonumber
    \end{align}
    
    If we take $a_1, a_2=0$ and $b_1, c_1, b_2, c_2$ arbitrary, then we can apply Lemma \ref{152301} to obtain
    \begin{align}
    \label{1523.1eq1}
        [H_1,H_2]\supset 1+p^{m_1+n_1}\mathbb{Z}\begin{pmatrix}
            1 & 0\\
            0& -1
        \end{pmatrix} \pmod{p^{\min\{m_2+n_2-1, 2m_1+n_1, m_1+2n_1\}}}.
    \end{align} 
    
    Similarly, by setting $b_1=0$ and $b_2=0$, we can get
    \begin{align}
        [H_1,H_2] & \supset 1+2p^{m_1+n_1}\mathbb{Z}\begin{pmatrix}
            0 & 0\\
            1& 0
        \end{pmatrix} \pmod{p^{\min\{m_2+n_2-1, 2m_1+n_1, m_1+2n_1\}}} \label{1523.1eq2},
      \end{align}
      and by setting $c_1, c_2=0$, we get
      \begin{align}  
        [H_1,H_2] & \supset 1+2p^{m_1+n_1}\mathbb{Z}\begin{pmatrix}
            0 & 1\\
            0& 0
        \end{pmatrix}  \pmod{p^{\min\{m_2+n_2-1, 2m_1+n_1, m_1+2n_1\}}} )  \label{1523.1eq3}
    \end{align}
Then (\ref{152302}) follows from (\ref{1523.1eq1}), (\ref{1523.1eq2}) and (\ref{1523.1eq3}).
\end{proof}

\begin{lemma} \label{1407} Let $p$ be a prime, $a,b\in \mathbb N$ and $b\geq 2 a\geq 0$.  Suppose $H\subset \normalfont\text{SL}_2(\mathbb Z)$ and $H(\mod p^b)=\Gamma (p^a)(\mod p^b)$. 
\begin{enumerate}
\item If $a\neq 0$, then if $p\neq 2$,
\begin{align} \label{1200}
[H,H]^{O(\log \frac{b}{a})} (\mod p^b)=\Gamma(p^{2a})/\Gamma (p^b).
\end{align}
If $p=2$, 
 \begin{align}
 [H,H]^{O(\log \frac{b}{a})} (\mod 2^b)= \Gamma(2^{2a+1})/\Gamma (2^b).
 \end{align}
\item If $a=0$, then 
\begin{align}\label{1214}
[H,H]^{O(\log b)}(\mod p^b)= \Gamma_{p^b}.
\end{align}
\end{enumerate}

\end{lemma} 
\begin{proof}
Assume $p\neq 2$. \par
If $a\neq 0$, let $s_0$ be the largest integer such that $1.5^{s_0}<\frac{b}{a}$. Apply Lemma \ref{1521} to $H_1=H_2=H$ with $m_1=n_1=[a\cdot 1.5^{j}]$, $m_2=n_2=\min\{2m_1, b\}$, $0\leq j\leq s_0$. Multiplying the implied sets together yields \eqref{1200}.  \par
If $a=0$, Then we have 
\begin{align}\label{1210}
[H,H]^3(\mod p)=\Gamma_p. 
\end{align}
See Theorem 1.1 of \cite{Sh09}. Let $H_0\subset [H,H]^3$ be the implied set of representatives of $\Gamma_p$. It follows from Lemma 4 of \cite{BV12} that $H_0(\mod p^2)$ is not a group, so there is $h_1, h_2, h_3\in H_0$ such that $h_1 h_2 h_3^{-1}\equiv 1(\mod p)$ but $h_1 h_2 h_3^{-1}\not\equiv 1(\mod p^2)$. From here it is easy to deduce that 
\begin{align}\label{11082}
[H_0, H_0]^{O(1)}(\mod p^2)=\Gamma_{p^2}.
\end{align}  \par If $b\geq 2$, we also know 
 \begin{align}\label{1213}
[H,H]^{O(\log b)}(\mod p^b)\supset \Gamma (p^2)/ \Gamma({p^b})
\end{align}
by the $a=1$ case of \eqref{1200}.  Combining \eqref{11082} and \eqref{1213} gives \eqref{1214}. \par
The proof for the $p=2$ case is analogous.

\end{proof}

\begin{lemma}\label{2247}
Let $p$ be a prime and $b\geq 2a\geq 1$. Suppose $\gamma_0\in \normalfont \text{SL}_2(\mathbb Z)$ with $v_p(\gamma_0-1)=a$, and $H\subset \normalfont \text{SL}_2(\mathbb Z)$ with $H\supset 1+p^b V (\mod p^{2b})$. Then $$ [[\gamma_0, H], H]^2     \supset 1+4p^{a+2b} V(\mod p^{a+3b}).$$
\end{lemma}
\begin{proof}  Assume $p$ is odd.  Write $\gamma_0 \equiv 1+p^a X_0 (\mod p^{2a})$ for some traceless $X_0\in \text{Mat}_2(\mathbb Z)$ that is also primitive mod $p$.

 It is an elementary exercise to check that it is always possible to choose $h_1, h_2\in H$, $h_1\equiv 1+ p^{b} Y_1 (\mod p^{2b}), h_2\equiv 1+ p^{b}Y_2(\mod p^{2b})$, $Y_1, Y_2$ traceless and primitive mod $p$, so that  $X_1= X_0Y_1 -Y_1 X_0$ and $X_2=  X_0Y_2 -Y_2 X_0$ are primitive and are not colinear mod $p$. \par
Write $\gamma_1=\gamma_0h_1\gamma_0^{-1}h_1^{-1}$ and $\gamma_2=\gamma_0h_2\gamma_0^{-1}h_2^{-1}$. From Lemma \ref{1946},  
\begin{align}
\gamma_1 \equiv 1+ \gamma_0 h_1- h_1\gamma_0 \equiv 1+ p^{a+b} X_1 (\mod p^{2a+b}) \\
\gamma_2 \equiv 1+ \gamma_0 h_2- h_2\gamma_0 \equiv 1+ p^{a+b} X_2 (\mod p^{2a+b})
\end{align}

Take $h_3, h_4\in H$ and write $h_3\equiv 1+p^b Y_3(\mod p^{2b})$ and  $h_4\equiv 1+p^b Y_4(\mod p^{2b})$. \par 
Applying Lemma \ref{1946} again, 
\begin{align}\label{22501}
\gamma_1h_3\gamma_1^{-1}h_3^{-1}\cdot \gamma_2h_4\gamma_2^{-1}h_4^{-1}  \equiv 1+ p^{a+2b}([X_1, Y_3]+[X_2, Y_4])(\mod p^{a+3b}).
\end{align}

Since $X_1, X_2$ are primitive and not colinear mod $p$, it is another elementary exercise to check that 
\begin{align}
\label{2251}
\{[X_1, Y_3]+[X_2, Y_4]: Y_3, Y_4\in V\}(\mod p^b) = V(\mod p^b),
\end{align}
(Hint: it suffices to assume $X_1=\mat{0&1\\0&0}$ and $X_2=\mat{0&0\\1&0}$, by identifying $V$ with $\mathbb Z^3$ and use a matrix in $\text{SL}_3(\mathbb Z)$ to change variables.)      \eqref{22501} and \eqref{2251} then imply Lemma \ref{2247}. \par

The $p=2$ case follows from the same argument. The extra factor 4 is due to the fact that $2V\subset [V, V]\neq V$. 

 \end{proof}

\begin{lemma} \label{2144} Let $p_1, p_2$ be two primes. Suppose $f$ is a homomorphism from $\Gamma(p_1^{m_1})/\Gamma(p_1^{n_1})$ to $\Gamma(p_2^{m_2})/\Gamma(p_2^{n_2})$ for some $0\leq m_1\leq n_1, 1\leq m_2\leq n_2$, and for some $\xi \in \Gamma(p_1^{m_1})/\Gamma(p_1^{n_1})$, $f(\xi)\neq 1$. Then $p_1=p_2$ and $v_{p_1} (\xi- 1) \leq m_1+n_2-m_2$.  
\end{lemma} 
\begin{proof}
We first observe that given $0\leq m_1 < n_1, 1\leq m_2<n_2$ and two primes $p_1, p_2$, if there is a nontrivial homomorphism $f$ from $\Gamma(p_1^{m_1})/\Gamma(p_1^{n_1})$ to $\Gamma(p_2^{m_2})/\Gamma(p_1^{n_2})$, this will force $p_1=p_2$.  To see this, if $m_1\neq 0$, then all elements in $\Gamma(p_1^{m_1})/\Gamma(p_1^{n_1})$ have orders powers of $p_1$, and all elements in $\Gamma(p_2^{m_2})/\Gamma(p_2^{n_2})$ have orders powers of $p_2$. Since $f$ is nontrivial, we have $p_1=p_2$.  \par
If $m_1=0$, and let us suppose $p_1\neq p_2$, then $f$ will factor through $\Gamma(p_1)/\Gamma(p_1^{n_1})$, which induces a homomorphism $f': \Gamma_{p_1}\rightarrow \Gamma(p_2^{m_2})/\Gamma(p_2^{n_2})$.  Since  $\Gamma_{p_1}$ is almost simple, $f'$ must either be trivial, or injective, or factor through the center $Z$. $f'$ can not be injective or factor through the center because $\mat{1&1\\0&1}\in \Gamma_{p_1}$ has order $p_1$, which forces $p_1| p_2$ if $f$ is injective or factor through $Z$, a contradiction.  Hence, $f'$ as well as $f$ must be trivial. Therefore, if $f$ is nontrivial, we must have $p_1=p_2$.   

Suppose $p_1=p_2=p$ and  $f(\xi)=y\not\equiv 1(\mod p^{n_2})$. Since the order of all elements in $\Gamma(p^{m_2})/\Gamma(p^{n_2})$ is bounded by $p^{n_2-m_2}$, it follows that $\xi$ can not be a $p^{n_2-m_2}$-power in $\Gamma(p^{m_1})/\Gamma(p^{n_1})$, and this will force $v_p(\xi) \leq m_1+n_2-m_2$.  
\end{proof}

\section{Preliminaries on Random Walks\label{rw}}

Recall that $S$ is a finite symmetric set in $\normalfont{\text{SL}}_2(\mathbb Z)\times \normalfont{\text{SL}}_2(\mathbb Z)$ or $\normalfont{\text{SL}}_2(\mathbb Z)\ltimes \mathbb Z^2$ such that $\langle S\rangle$ is Zariski-dense,  and $\chi_S$ is the uniform probability measure supported on $S$. In this section we record some quantitative statements on non-concentration of iterated convolutions of $\chi_S$ in certain subvarieties. Proposition \ref{decay1} and Proposition \ref{decay3} concern linear subvarieties, and Proposition \ref{decay2} is on certain quadratic subvarieties. \par

\begin{proposition} \label{decay1}Let $S$ be a finite symmetric set on $\normalfont{\text{SL}}_2(\mathbb Z)\times \normalfont{\text{SL}}_2(\mathbb Z)$ such that $\langle S\rangle$ is Zariski-dense in $\normalfont\text{SL}_2\times \normalfont\text{SL}_2$.  Let $\chi_S$ be the uniform probability measure supported on $S$.  
{There are constants $0<c_S, c<1$ depending only on $S$}
such that for any sufficiently large $Q\in \mathbb Z_+$, for any $l>{c_S}{\log Q}$, and any $n\in\mathbb Z$, we have 
$$
\chi_S^{(l)}\left(\left\{g\in {\normalfont\text{SL}}_2(\mathbb Z)\times {\normalfont\text{SL}}_2(\mathbb Z)| L(g)\equiv n (\mod Q)  \right\}\right)< Q^{-c},$$
where  $L$ is a general primitive linear form on $\text{SL}_2(\mathbb Z)\times \text{SL}_2(\mathbb Z)$, i.e., 
$$L\left( \mat{x_0&x_1\\x_2&x_3}, \mat{x_4&x_5\\x_6&x_7}    \right) = X_0 x_0+X_1 x_1+ X_2x_2 + X_3 x_3+ X_4x_4+X_5x_5+X_6x_6+X_7x_7, $$
with $\normalfont\text{gcd}(X_0, X_1, X_2, X_3, X_4, X_5, X_6, X_7)=1$. 

\end{proposition}

\begin{proposition}\label{decay2} Let $S$ be given as in Proposition \ref{decay1}. There are constants $0<c_S, c<1$ depending only on $S$ such that the following holds.  Suppose $\xi=(\xi_1, \xi_2),\eta=(\eta_1, \eta_2)\in \normalfont\text{Mat}_2(\mathbb Z)\times \normalfont\text{Mat}_2(\mathbb Z)$ satisfy
\begin{align}
&\normalfont\text{Tr}(\xi_1)=\normalfont\text{Tr}(\xi_2)=\normalfont\text{Tr} (\eta_1)=\normalfont\text{Tr}(\eta_2)=0, \nonumber
\end{align}
Then for any sufficiently large $Q\in\mathbb Z_+$ with either $\normalfont\text{gcd}(Q, \text{entries of }\xi_1 \text{ and } \eta_1) = 1$ or $\normalfont\text{gcd}(Q, \text{entries of }\xi_2 \text{ and } \eta_2) = 1$, for any $l>c_S\log Q$ and for any $n\in\mathbb Z$, 
$$\chi_S^{(l)}(\{ (g_1, g_2)\in {\normalfont\text{SL}}_2(\mathbb Z)\times {\normalfont\text{SL}}_2(\mathbb Z) | \normalfont\text{Tr} (g_1\xi_1 g_1^{-1} \eta_1)+\normalfont\text{Tr} (g_2\xi_2 g_2^{-1} \eta_2) \equiv n(\mod {Q})\})<Q^{-c}.$$
\end{proposition} 

{
In the $\mathrm{ASL}_2$ case, we need \par

\begin{proposition} \label{decay3}Let $S$ be a finite symmetric set on $\mathrm{SL}_2(\mathbb{Z})\ltimes \Z^2$ such that $\langle S\rangle$ is Zariski-dense.  Let $\chi_S$ be the uniform probability measure supported on $S$.  
There are constants $0<c_S, c<1$ depending only on $S$
such that for any sufficiently large $Q\in \mathbb Z_+$, any $l>{c_S}{\log Q}$ and any $n\in\mathbb Z$, we have 
$$\chi_S^{(l)}\left(\left\{g\in \mathrm{SL}_2(\mathbb{Z})\ltimes \Z^2 | L(g)\equiv n (\mod Q)  \right\}\right)< Q^{-c},$$
where $L$ is a primitive linear form on $\mathrm{SL}_2(\mathbb{Z})\ltimes \Z^2$.
\end{proposition}
}

\begin{remark}
In Proposition \ref{decay1} and Proposition \ref{decay3}, it is important that the implied constant $c$ is independent of $L$ and $n$. 
\end{remark}

We focus on proving Proposition \ref{decay1}. Proposition \ref{decay2} and Proposition \ref{decay3} follow in a similar way. Our method is similar to the proof of Proposition 4.1 in \cite{BG09}. \par
We first show that Proposition \ref{decay1} follows from 

\begin{lemma}\label{linear} There are constants $0<c_S, c_0<1$ depending only on $S$ 
such that for any $Q\in \mathbb Z_+, n\in \mathbb Z$ and for any $1\ll_S l\leq c_S \log Q$, we have
	\begin{equation}
		\label{BG_II4.4.2}
		\chi_S^{(l)}(\{ g\in \normalfont\text{SL}_2(\mathbb Z)\times \normalfont\text{SL}_2(\mathbb Z)| L(g)\equiv n (\mod Q) \})< e^{-c_0l}.  
	\end{equation}

\end{lemma}

\begin{proof}[Proof of Proposition \ref{decay1} assuming Lemma \ref{linear}] Let $c_S, c_0$ be the constants given by Lemma \ref{linear}. Let $l_0= [c_S\log Q]$ and write $\chi_S^{(l)}=\chi_S^{l_0}*\chi_S^{(l-l_0)}$.  For any $g'$ in the support of $\chi_S^{l-l_0}$, let $L_{g'} (g)= L(gg')$.  Clearly, $L_{g'}$ is also primitive, so Lemma \ref{linear} is applicable to $L_{g'}$. Therefore, 
\begin{align*}
&\chi_S^{(l)}\left(\left\{g\in {\normalfont\text{SL}}_2(\mathbb Z)\times {\normalfont\text{SL}}_2(\mathbb Z)| L(g)\equiv n (\mod Q)  \right\}\right) \\
=&\sum_{g'\in \Lambda} \chi_S^{(l_0)}\left(\left\{g\in \text{SL}_2(\mathbb Z)\times \text{SL}_2(\mathbb Z)| L_{g'}(g)\equiv n (\mod Q)  \right\}\right) \chi_S^{(l-l_0)}(g')\\
<&e^{-c_0l_0}<Q^{-c_0c_S}
\end{align*}
Proposition \ref{decay1} is thus proved with $c=c_0c_S$.
\end{proof}

Now we focus on proving Lemma \ref{linear}.  We begin with

\begin{lemma}[Non-concentration at Archimedean place]   
	
	\label{BG_IIProp4.2.1}
	There is a constant $c_0>0$ depending only on $S$ such that for any $l\gg_{S }1$, any $n\in\mathbb Z$, we have
	\begin{align}
		&\chi_S^{(l)}\left(\left\{g \in \mathrm{SL}_2(\mathbb{Z})\times \mathrm{SL}_2(\mathbb{Z}): L(g)=n \right\}\right)\nonumber<e^{-c_0 l} .
	\end{align}
\end{lemma}
\begin{proof}By Theorem \ref{Golsefidy}, there exists an absolute constant $0<\lambda<1$ which is the upper bound for all eigenvalues of the family of the operators $T_p: l_0^2(\Lambda)\rightarrow l_0^2(\Lambda)$ defined at \eqref{1457}.
So
\begin{equation}
	\|\chi_S^{(l)}-\dfrac{1}{|\pi_p(\G)|}\boldsymbol{1}_{\pi_p(\G)}\|_2\leq \lambda^{l}\|\chi_S\|_2={\lambda^l}{|S|^{-\frac{1}{2}}}
\end{equation}
for all prime $p$. Since $|\pi_p(\G)|\sim p^6$, if $l>\frac{6\log p}{\log (1/\lambda)}$,  
\begin{equation}
	\chi_S^{(l)}(g)<\dfrac{2}{|\pi_p(\G)|}
\end{equation}
for any $g\in \pi_p(\G)$.  Therefore,
\begin{align*}
	&\chi_S^{(l)}\{g\in \mathrm{SL}_2(\mathbb{Z})\times {\normalfont\text{SL}}_2(\Z):L(g)=n\}\\
	&\leq \chi_S^{(l)}\{g\in \mathrm{SL}_2(\mathbb{Z})\times {\normalfont\text{SL}}_2(\Z):L(g)\equiv n(\mod{p})\}\\
	&< \frac{4}{p},
\end{align*}
by counting points in subvarieties of codimension $1$ determined by primitive linear equations.  So we are done by picking any prime $p\in [e^{\frac{l}{8}\log \frac{1}{\lambda}}, e^{\frac{l}{7}\log \frac{1}{\lambda}}]$, which exists by the Prime Number Theorem when $l$ is sufficiently large. This finishes the proof of Lemma $\ref{BG_IIProp4.2.1}$.
\end{proof}

We follow the method from Bourgain and Gamburd \cite{BG09} to prove Lemma \ref{linear}.  The proof requires Effective B\'ezout Identity \cite{BY91}.  For readers' convenience, we record it here: 
 \begin{thmx}[The Effective B\'ezout Identity]
	\label{EB}
	\textit{Let $\mathcal P_1, \ldots, \mathcal P_N \in \mathbb{Z}[z_1,\dots , z_n]$ without common zeros in $\mathbb{C}^n$ with $n\geq 2, \operatorname{deg} \mathcal P_j \leqslant d, d \geqslant 3$, $h\left(\mathcal P_j\right) \leqslant h$. Here $h(\mathcal P)$ is the logarithmic of the modulus of the coefficient of $\mathcal P$ of largest absolute value. Then there is an integer $D \in \mathbb{Z}_{+}$, polynomials $\mathcal Q_1, \ldots, \mathcal Q_N \in \mathbb{Z}[z_1,\dots , z_n]$ such that
	$$
	\begin{gathered}
		\mathcal P_1 \mathcal Q_1+\ldots+\mathcal P_N \mathcal Q_N=D,\\
		\operatorname{deg} \mathcal Q_j\leq n(2n+1)d^n,\\
		h\left(\mathcal Q_j\right) \leqslant \frak  X(n) d^{8 n+3}(h+\log N+d \log d), \\
		\log D \leqslant \frak X(n) d^{8 n+3}(h+\log N+d \log d),
	\end{gathered}
	$$
	where $\frak X(n)$ is an effective constant which only depends on $n$ and can be computed explicitly. }
\end{thmx}

\begin{proof}[Proof of Lemma \ref{linear}]
Write $Q=\prod_{i\in I}p_i^{n_i}$.  We divide our proof into two cases $n=0$ and $n\neq 0$.\par
\noindent{\bf Case 1: $n= 0$.}  By the primitivity of $L$,  for each $p|Q$, at least one element in 
$$  \mathcal{A}:=\{X_0, X_1, X_2, X_3, X_4, X_5, X_6, X_7\}$$
must be invertible $\mod \ p$. For each $t\in \mathcal{A}$, let 
$$Q_t=\prod_{\substack {p_i^{n_i} \Vert Q \\ \text{gcd} (p_i, t)=1 }}p_i^{n_i}. $$
Since $\prod_{t\in \mathcal A}Q_t\geq Q$, there exists $t\in \mathcal{A}$ such that
\begin{equation*}
Q':=Q_t\geq Q^{\frac{1}{8}}.
\end{equation*}
Assume $t=X_0$ without loss of generality. Then $Q'\Vert Q$ and  $(Q', X_0)=1$. 

For a matrix $g\in \normalfont\text{Mat}_2(\mathbb C)$, let $\|g\|=\max\{ \text{absolute values of coefficients of }g\}$, and let $$\|(g_1, g_2)\|=\max\{ \|g_1\|, \|g_2\| \}.$$

Let $N_1$ be an upper bound of $ \|g\|$ for all $g \in \operatorname{supp} [ \chi_S]$. \par Define
$$
\mathcal{G}=\left\{ g \in  \text{supp } [\chi_S^{(l)}]  \mid L(g) \equiv 0 \ (\mod {Q'})\right\}.
$$
To show $(\ref{BG_II4.4.2})$, it suffices to show
\begin{align}
\label{1011}\chi_S^{(l)}(\mathcal{G})<e^{-c l},
\end{align}
for some $c>0$. \par
For each $\g=\left(\begin{pmatrix}
	\gamma_0 & \gamma_1\\
	\gamma_2 & \gamma_3
\end{pmatrix}, \begin{pmatrix}
	\gamma_4 & \gamma_5\\
	\gamma_6 & \gamma_7
\end{pmatrix}\right)\in \mathcal{G}$, we introduce a linear polynomial $$f_\g(x_1,x_2,x_3,x_4,x_5,x_6,x_7) \in \mathbb{Q}[x_1,x_2,x_3,x_4,x_5,x_6,x_7]$$ as follows:
\begin{equation}
	f_\g(x_1, \cdots, x_7)=\gamma_0+\sum_{i=1}^7 \gamma_ix_i \label{linear_equation}
\end{equation}
Then we get 
\begin{equation}
	f_\g\left(X_1\overline{X_0} , X_2\overline{X_0}, X_3\overline{X_0}, X_4\overline{X_0} , X_5\overline{X_0}, X_6\overline{X_0}, X_7\overline{X_0}\right) \equiv 0\ (\mod {Q'})  \label{BG_II4.8.1}
\end{equation}
\text{ for all }$\g\in \mathcal{G}.$
Here $\overline{X_0}$ is the multiplicative inverse of $X_0$ mod $Q'$.  \par
Also, by the definition of $\mathcal{G}$, the coefficients of $f_\g$, namely the entries of $\gamma$, are bounded by $(2N_1)^l$. Hence,
$$
h\left(f_g\right)<2l\log N_1,
$$
{\bf Claim:} There is a common zero $(\tilde{x}_1, \cdots, \tilde{x}_7) \in \mathbb{C}^{7}-\{\vec{\bf 0}\}$ to the following system of equations:
\begin{align}
	f_\g({x}_1, \cdots, {x}_7) =0,  \quad   \g \in \mathcal{G} \label{BG_II4.11.1}.
\end{align}
Hence $\mathcal{G}$ is contained in some proper subvariety of $\mathrm{SL}_2(\mathbb{Z})\times \mathrm{SL}_2(\mathbb{Z})$. Note that the size of $\mathcal G$ is bounded by $|S|^l$.

Assume the claim fails to hold. We invoke Theorem $ \ref{EB}$ with $n=7, d=3, h=2l\log N_1$ and $N\leq |S|^l$. It follows that there is an integer $M \in \mathbb{Z}_{+}$ and polynomials $\phi_\gamma$ of degree at most $b=7\times 15 \times 3^{7}$ satisfying
\begin{equation}
	\label{BG_II4.12}\sum_{\gamma\in \mathcal G} f_\gamma \varphi_\gamma=M
\end{equation}
with
\begin{equation}
	\label{BG_II4.13}
	0<\log M, h\left(\varphi_\gamma\right)< \frak X(7) 3^{59}((\log N_1 +\log (2|S|))l +3 \log 3)<N_2 l
\end{equation}
where $N_2= 3^{61}\mathfrak X(7)\log (2N_1 |S|)$.  \par
Now we take $(y_1,y_2,y_3,y_4,y_5,y_6,y_7)\in\mathbb Z^7$ such that $$(y_1,y_2,y_3,y_4,y_5,y_6,y_7) \equiv(X_1\overline{X_0} , X_2\overline{X_0}, X_3\overline{X_0}, X_4\overline{X_0} , X_5\overline{X_0}, X_6\overline{X_0}, X_7\overline{X_0}) (\mod Q').$$ It follows from \eqref{BG_II4.12} and \eqref{BG_II4.8.1} that
$$
0= \sum_{\gamma\in\mathcal G}f_\gamma (y_1,\cdots, y_7) \phi_\gamma(y_1,\cdots, y_7)-M\equiv -M (\mod Q').
$$

Therefore, 
since $M\neq 0$, 
by $(\ref{BG_II4.13})$ we deduce $\frac{1}{8}\log Q\leq \log Q'\leq \log M<N_2 l$, which contradicts the restriction $l< c_S \log Q$ if we take $c_S\leq \frac{1}{8N_2}$. This proves the claim. \par
Since the linear system \eqref{BG_II4.11.1} admits a solution and the coefficients of $f_\gamma$ are all integral, it must admit a rational solution $(\tilde{x}_1, \tilde{x}_2, \tilde{x}_3, \tilde{x}_4, \tilde{x}_5, \tilde{x}_6, \tilde{x}_7)$. In other words,  
\begin{align}\label{0246}
\gamma_0+ \sum_{i=1}^7 \gamma_i \tilde{x}_i=0, \forall \gamma\in \mathcal G
\end{align}

If at least one element $\gamma\in \mathcal G$ has $\gamma_0\neq0$, then $(\tilde{x}_1, \tilde{x}_2, \tilde{x}_3, \tilde{x}_4, \tilde{x}_5, \tilde{x}_6, \tilde{x}_7)\neq \vec{\bf{0}}$. Rescaling to get rid of the common denominator in \eqref{0246}, we obtain $(v_0,v_1,v_2,v_3,v_4,v_5,v_6,v_7)\in \mathbb Z^8-\vec{\bf 0}$ such that $\text{gcd}(v_0,v_1,v_2,v_3,v_4,v_5,v_6,v_7)=1$, and for all $\gamma\in \mathcal G$, we have 
$$\sum_{i=0}^7 \gamma_i v_i=0.$$
Apply Lemma $\ref{BG_IIProp4.2.1}$ to the linear form $\tilde L$ determined by the constants $v_1,v_2,v_0,v_3,v_4,v_5,v_6,v_7$: For $1\ll_S  l<c_S \log Q$,
\begin{align*}
	\chi_S^{(l)}(\mathcal{G}) \leq\ &\chi_S^{(l)}\left(\left\{\g \in \mathrm{SL}_2(\mathbb{Z})\times \mathrm{SL}_2(\mathbb{Z})\mid \tilde{L}(\g)=0 \right\}\right)\\
 <\ & e^{-c_0 l} .
\end{align*}
where $c_0$ is the constant given in Lemma \ref{BG_IIProp4.2.1}. \par
If all $\gamma\in \mathcal G$ have $\gamma_0=0$, then $\mathcal G$ lies in a linear subvariety, and \eqref{BG_II4.4.2} is given by Lemma \ref{BG_IIProp4.2.1}. We have thus fully proved the case of Lemma \ref{linear} when $n=0$. \par
 \medskip
\noindent{\bf Case 2: $n\neq 0$.} Write $Q_1=(Q, n)$, $n=Q_1n_1$ and $Q=Q_1Q_2$. We clearly have $(n_1, Q_2)=1$. We analyse the cases $Q_1$ small and $Q_1$ large separately. \par
\noindent{Case 2.1: $Q_1< Q_2^{\frac{1}{3N_3}}$}, where the constant $N_{3}$ is given at \eqref{1535}. Consider the following set: 
$$
\mathcal{G}'=\left\{ g \in  \text{supp } [\chi_S^{(l)}]  \mid L(g)-Q_1n_1 \equiv 0 \ (\mod {Q_2})\right\}.
$$

For each $\gamma=\left(\mat{\gamma_0&\gamma_1\\\gamma_2&\gamma_3}, \mat{\gamma_4&\gamma_5\\\gamma_6&\gamma_7}\right)\in \mathcal G'$, define 
$$f_\gamma (x_0, x_1, x_2, x_3, x_4, x_5, x_6, x_7)= \sum_{i=0}^7 \gamma_i x_i- Q_1.$$
If $\{f_\gamma, \gamma\in \mathcal G'\}$ has no common zero, then there exist integral polynomials $\phi_\gamma$ and $M\in\mathbb Z_+$ such that 
\begin{align}\label{11401}
\sum_{\gamma\in\mathcal G' }f_\gamma \phi_\gamma=M.
\end{align}
with
\begin{equation}
	\label{BG_II4.14}
	0<\log M, h\left(\varphi_\gamma\right)< \frak X(8) 3^{67}((\log N_1 +\log (2|S|))l +\log Q_1+3 \log 3)<N_3 (l+\log Q_1)
\end{equation}
where 
\begin{align}\label{1535}
N_3=3^{69}\frak X(8)\log (2|S|N_1)
\end{align}
 \par
Let $(z_0, \cdots, z_7)\in\mathbb Z^8$ such that $(z_0, \cdots, z_7) \equiv (\gamma_0\overline{n_1}, \cdots, \gamma_7\overline{n_1})(\mod Q_2)$, where $\overline{n_1}$ is the multiplicative inverse of $n_1(\mod Q_2)$, then \eqref{11401} implies 
$$\sum_{\gamma\in\mathcal G'} f_\gamma (z_0, \cdots, z_8)\phi_\gamma(z_0,\cdots,z_8)\equiv M(\mod Q_2),$$
so that
\begin{align}\label{331}
M\geq Q_2>Q^{\frac{3}{4}},
\end{align}
which contradicts \eqref{BG_II4.14} if we take $c_S\leq \frac{1}{3N_3}\log Q $. Hence $\{f_\gamma, \gamma\in \mathcal G'\}$ have a common zero $(v_1,\cdots,v_8)\in\mathbb Z^8-\{\vec{\bf 0}\}$. Applying Lemma \ref{BG_IIProp4.2.1} then yields 
\begin{align}
\label{0434}\chi_S^{(l)}(\mathcal{G}')<e^{-c l},
\end{align}
for some $c>0$.  \par
\noindent {Case 2.2: $Q_1\geq Q_2^{\frac{1}{3N_3}}$.}  In this case the equation $L(\gamma)-n\equiv 0(\mod Q)$ implies $L(\gamma)\equiv 0(\mod Q_1)$. Thus we reduce this case to Case 1. 
\end{proof}

\section{Bounded generation for simple factors\label{bg}}

Let $\mathbb P_0$ be the projection from $\text{SL}_2(\mathbb Z)\ltimes \mathbb Z^2$ to $\text{SL}_2(\mathbb Z)$, and let $\mathbb P_1$ ($\mathbb P_2$, resp.) be the projection map from $\text{SL}_2(\mathbb Z)\times \text{SL}_2(\mathbb Z)$ to its first factor (second factor, resp.).  The goal of this section is to prove the following bounded generation result for simple factors of $\Lambda (\mod q)$, where $q$ is a product of large powers of primes: \par

\begin{proposition} \label{1631} Let $S$ be a finite symmetric set in $\Lambda$ which generates a Zariski dense subgroup. There exist positive functions $\rho=\rho(\delta)=O_S(\delta), C=C(\delta)\in\mathbb Z_+$ defined for any sufficiently small $\delta>0$, such that the following holds: For any sufficiently large $q=\prod_{i\in I}p_i^{n_i}, n_i\geq \frac{3}{\delta}$, and for any $l\in\mathbb Z_+$ and $A\subset \Lambda$ that satisfy 
$$\chi_S^{(l)}(A)>q^{-\delta}, \hspace{5mm}  l>\delta^{-1}\log q  $$
and 
\begin{align}\label{20241}
|\pi_q(A\cdot A\cdot A)|\leq |\pi_q(A)|^{1+\delta}, 
\end{align}
we have 
\begin{equation}
 \Gamma (q_*^{\{\rho\}})/\Gamma(q_*)\subset\begin{cases} & (\mathbb P_iA)^{C}(\mod q_*), i=1,2,  \hspace{3mm} \text{if } \hspace{2mm} \Lambda= {\normalfont\text{SL}}_2(\mathbb Z)\times {\normalfont\text{SL}}_2(\mathbb Z), \\ & (\mathbb P_0 A)^{C} (\mod q_*), \hspace{16mm} \text{if }\hspace{2mm} \Lambda={\normalfont\text{SL}}_2(\mathbb Z)\ltimes \mathbb Z^2.
 \end{cases}
 \end{equation}
where $q_*\Vert q$ and $q_*\geq q^{c}$ for some $c>0$ depending only on $S$.
\end{proposition}

We focus on the case $\Lambda=\text{SL}_2(\mathbb Z)\times \text{SL}_2(\mathbb Z)$ and $i=1$. The proofs for the remaining two cases are virtually identical. Our proof follows the framework in \cite{BG09}, which are outlined below. These statements are understood to hold under the reduction of appropriate moduli. 
\begin{enumerate}[leftmargin=*, label=(\alph*)]
\item Use non-concentration of random walks (Proposition \ref{decay1}) to find a large set of matrices with distinct traces from $\mathbb P_1(A^2)$ (Lemma \ref{0608}).  
\item Assuming \eqref{20241} holds, use the matrices obtained from (a) to create a large set of commutative matrices (Lemma \ref{1059}).
\item Diagonalize the matrices from (b) over a quadratic field to generate a large set of diagonal matrices that satisfy multiple density and proximity-to-identity constraints (Lemma \ref{0804}). 
\item Apply the sum-product theorem (Theorem \ref{sumproduct}) to create a thick line in $\normalfont\mathfrak {sl}_2$ (Lemma \ref{1236}). 
\item Use random walks (Proposition \ref{decay2}) to move the line from (d) to different directions by conjugation, thus to create a thick segment in $\normalfont\mathfrak {sl}_2$ (Lemma \ref{segment1}). 
\item Iteratively take commutators of the segment obtained from (e) to finish the proof of Proposition \ref{1631}.

\end{enumerate}

\begin{remark}The method for finding commutative matrices can be traced back to \cite{He08}. 

\end{remark}

The presentation in this section involves several constants. For readers' convenience we record them here. \par
 Let $c_1, c_2$ be the implied constants for $c$ in Propositions \ref{decay1}, \ref{decay2}. We take $c_1, c_2<1$. \par 
 Let $C_1, C_2, C_3, \varepsilon=\varepsilon_0$ be the implied constants in Theorem \ref{sumproduct} for $\alpha=\frac{c_1}{6}$. \par
 
Let $\rho_1, \rho_2$ be two small quantities linearly depending on the (only) variable $\delta$:  
\begin{align}\label{0108}
\rho_1=L_1\delta, \hspace{1cm} \rho_2=L_2\rho_1=L_1L_2\delta.
\end{align}
Here $L_1$ and $L_2$ are two big constants up to several adjustments in the proof process, and we assume at each adjustment, they satisfy not only the current requirement but also all preceding ones. While it is important $L_1$ and $L_2$ only depend on $S$, we take $L_1$ and $L_2$ so large that any positive linear combination of $\delta$ and $\rho_1$ appearing in our presentation is dominated by the $\rho_1$ term, and any positive linear combination of $\rho_1$ and $\rho_2$ is dominated by the $\rho_2$ term. So $\delta$ should be considered significantly smaller than $\rho_1$, which in turn is significantly smaller than $\rho_2$.  The main parameter at \eqref{0108} is $\rho_2$; the parameter $\rho_1$ plays a correction role. \par
Unless otherwise specified, the Landau notations $\Theta, \Omega, O$ in this section describe functions of $\delta$ and the implied constants depend only on $S$.  \par

We also need to introduce two small quantities $w_1, w_2 =\Theta(\rho_2)$ given in Lemma \ref{0804}. So by our convention above, $w_1$, $w_2$ is much larger than $\rho_1$.

%

%
%
%
%
%
\subsection{A large set of traces\label{Helfgott}} Let $$A_0=(\mathbb P_1 (A)\cdot \mathbb P_1(A))\cap \Gamma (\tilde q) \subset \mathbb P_1(A^2)\subset \text{SL}_2(\mathbb Z),$$ where $\tilde q=\prod_{p|q}p$ is the square-free part of $q$. \par 
Let $\nu$ be the push forward of $\chi_S$ under the projection map $\mathbb P_1$. \par 
For each $x\in \Gamma_{\tilde q}$, let 
$$B_x= \{y\in \mathbb P_1(A): y\equiv x(\mod \tilde q)\}$$

 Then, 

\begin{align} \label{2023}\nonumber\nu^{(2l)}(A_0)&= \sum_{x\in \Gamma_{\tilde q}}\nu^{(l)}(B_x)\nu^{(l)}(B_{x^{-1}}) \stackrel{\text{symmetry of $A$}}{=}\sum_{x\in \Gamma_{\tilde q}}\nu^{(l)}(B_x)^2\\ & \stackrel{\text{Cauchy-Schwarz}}{\geq} \frac{\nu^{(l)}(\mathbb P_1(A))^2}{|\Gamma_{\tilde q}|} > q^{-2\delta}(\tilde q)^{-3}\geq q^{-3\delta},
\end{align} 
where in the last inequality we have used that the exponent for each prime divisor of $q$ exceeds $\frac{3}{\delta}$. \par

The goal of this section is to prove 

\begin{lemma}\label{0608}
There is a set $A_1\subset A_0$, an exact divisor $q_1\Vert q$ with $q_1>q^{\frac{1}{2}}$, such that 
\begin{align}\label{06081}
&|A_1(\mod q_1^{\{\rho_2\}})|>q_1^{\frac{\rho_2c_1}{2}}, \\
&\label{0618} \forall \gamma_1, \gamma_2\in A_1, \gamma_1\neq\gamma_2\Rightarrow \normalfont\text{tr}(\gamma_1)\not\equiv\normalfont\text{tr}(\gamma_2)(\mod q_1^{\{\rho_2\}}) \\
&\label{0619} \forall \gamma=\mat{\gamma_{11}&\gamma_{12}\\\gamma_{21}&\gamma_{22}}\in A_1, \normalfont\text{tr}(\gamma)\not\equiv \pm 2, \gamma_{21}\not\equiv 0(\mod p^{[4n\rho_1]}),   \forall p^n\Vert q_1,
\end{align}
where $c_1$ is the implied constant $c$ from Proposition \ref{decay1}.
\end{lemma}
\begin{remark}
Lemma \ref{0608} gives us a large set of matrices with distinct traces not close to $\pm 2$ at all non-Archimedean places $p$ where $v_p(q_1)>0$ . We also need conditions \eqref{0618} and \eqref{0619} because they ensure that we can diagonalize any element from $A_1$ by a matrix with controlled denominator. 
\end{remark}

\begin{proof}[Proof of Lemma \ref{0608}]

For each $Q|q$, let
\begin{align*}
&\mathcal E_1(Q)=  \{\gamma\in \text{SL}_2(\mathbb Z):  \text{tr}(\gamma)\equiv \pm 2 \equiv 0(\mod Q))\}, \\
&\mathcal E_2(Q)=  \left\{\gamma\in \text{SL}_2(\mathbb Z): \gamma_{21} \equiv 0(\mod Q))\right\}.
\end{align*}
Proposition \ref{decay1} gives
$$\nu^{(2l)}(\mathcal E_1(Q)), \nu^{(2l)}(\mathcal E_2(Q))<Q^{-c_1},$$
for sufficiently large $Q$.
Let 
$$A_1'=A_0-\bigcup_{\substack {Q|q\\Q\geq q^{\rho_1}}}(\mathcal E_1(Q)\cup \mathcal E_2(Q)).$$
Since the number of divisors of $q$ is bounded by $q^{0+}$, it follows that 
$$\nu^{(2l)}(A_1')> q^{-3\delta-},$$
if we let 
\begin{align}\label{1704}
L_1>\frac{1}{c_1}.
\end{align}

For each $\gamma\in A_1'$, let $$q_\gamma^{'}=\prod_{\substack{p^n\Vert q \\ \text{tr}(\gamma)\not \equiv \pm 2 (\mod p^{[4n\rho_1]}) }}p^n,$$
and 
$$q_\gamma^{''}=\prod_{\substack{p^n\Vert q \\ \gamma_{21}\not \equiv  0 (\mod p^{[4n\rho_1]}) }}p^n.$$

Since $\text{gcd}(\text{tr}(\gamma)\pm 2, q)<q^{\rho_1}$ and $\text{gcd}(\gamma_{21}, q)<q^{\rho_1}$, we must have $q_\gamma', q_\gamma''>q^{\frac{3}{4}}$. \par

Let $q_\gamma=\text{gcd}(q_\gamma', q_\gamma'')$. Then $q_\gamma> q^{\frac{1}{2}}$, and for any $p^n\Vert q_\gamma$, 
\begin{align}
\label{2059}
& \text{tr}(\gamma)\not\equiv \pm 2(\mod p^{[4\rho_1 n]}),\\
\label{2058}
&\gamma_{21}\not\equiv 0(\mod p^{[4\rho_1n]}).
\end{align}

Since again the number of divisors of $q$ (which bounds the number of choices for $q_\gamma$) is bounded by $q^{0+}$, there exists $q_1\Vert q, q_1>q^{\frac{1}{2}}$  and a set $A_1''\subset A_1'$, with

\begin{align}\label{0723}\nu^{(2l)}(A_1'')>q^{-3\delta-}>q_1^{-6\delta-}, 
\end{align}
such that for any $\gamma\in A_1''$, any $p^n \Vert q_1$, $\eqref{2059}$ and $\eqref{2058}$ are satisfied. \par

Let $$ W(n) := \{\gamma\in  \Gamma: \text{tr}(\gamma)= n(\mod q_1^{\{\rho_2\}})  \}.$$

By Proposition \ref{decay1}, for any $n\in \mathbb Z$, 
\begin{align} \label{1628}
\nu^{(2l)}(W(n))< q_1^{-\rho_2c_1}.
\end{align}

If we take 
\begin{align}\label{11211}
L_2>\frac{12}{c_1},
\end{align}
then \eqref{0723} implies $$\nu^{(2l)}(A_1'')>q_1^{-\frac{\rho_2c_1}{2}},$$ which together with \eqref{1628} implies the existence a set $A_1\subset A_1''$ satisfying all requirements in 
Lemma \ref{0608}. 

\end{proof}

\subsection{A large set of commutative matrices} Our goal in this section is to prove 
\begin{lemma} \label{1059}
Let $A_1$ be as given in Lemma \ref{0608}. There is an element $\alpha_0\in A_1$ and a set $A_2\subset \mathbb P_1(A)^2$ such that 
\begin{enumerate}
\item $\forall \gamma_1,\gamma_2\in A_2$, $\gamma_1\neq \gamma_2\Rightarrow \gamma_1\neq \gamma_2(\mod q_1^{\{\rho_2\}})$,

\item $|A_2|> q_1^{\frac{\rho_2c_1}{3}}$,

\item  $A_2$ commutes with $\alpha_0$ {\text{mod }}$q_1^{\{\rho_2\}}$.

\end{enumerate}
\end{lemma}
\begin{proof}
Denote $$Q_0=q_1^{\{\rho_2\}}.$$

First, it follows from Proposition \ref{decay1} that for any $\gamma_0\in \Gamma_q$, 
\begin{align}\label{0238}
\nu^{(2l)}\{\gamma \in \Gamma: \gamma\equiv \gamma_0(\mod Q_0)\}< Q_0^{-c_1}.
\end{align}
Therefore,  
\begin{align}\label{0431}
|\pi_{Q_0}(\mathbb P_1(A))|>Q_0^{c_1}q^{-3\delta}>Q_0^{\frac{c_1}{2}},
\end{align}
recalling \eqref{11211}. \par
Next, we use our assumption \eqref{20241} to get a control over the growth of $\pi_{Q_0}(\mathbb P_1(A))$. By Lemma 2.2 from \cite{He08}, \eqref{20241} implies for any $l\geq 3$,
\begin{align}\label{0929}
&|\pi_q(A^l )|\leq \left( \frac{|\pi_q(A\cdot A\cdot A)|}{|\pi_q(A)|} \right)^{l-2}|\pi_q(A)|\\\label{0931}<&|\pi_q(A)|^{1+\delta(l-2)}.
\end{align}
This implies 

\begin{align}\label{1733}
|\pi_{Q_0}(\mathbb P_1(A)\cdot \mathbb P_1(A)\cdot \mathbb P_1(A))|< |\pi_{Q_0}(\mathbb P_1(A))|^{1+\frac{c_1}{36}}.
\end{align}
To see this, suppose not, then 
\begin{align}\nonumber \label{01071}
|\pi_{q}(A^4)|\geq& |\pi_{Q_0}(\mathbb P_1(A)\cdot \mathbb P_1(A)\cdot \mathbb P_1(A))|\cdot   \max_{\gamma_0\in \Gamma_{Q_0}}| \{(\gamma_1,\gamma_2)\in \pi_q(A):  \gamma_1\equiv \gamma_0(\mod Q_0)\} | \\
\nonumber\geq& |\pi_{Q_0}(\mathbb P_1(A))|^{1+\frac{c_1}{36}}\cdot   \max_{\gamma_0\in \Gamma_{Q_0}}| \{(\gamma_1,\gamma_2)\in \pi_q(A):  \gamma_1\equiv \gamma_0(\mod Q_0)\} | \\
\nonumber\geq &  |\pi_{Q_0}(\mathbb P_1(A))|^{\frac{c_1}{36}}\cdot |\pi_q(A)| \\
\stackrel{\eqref{0431}}{>}& |\pi_q(A)|Q_0^{\frac{c_1^2}{72}} > |\pi_q(A)|q^{\frac{\rho_2c_1^2}{144}},
\end{align}
where for the first inequality, we have used that if $X\subset \pi_q(A^3)$ is a set of representatives for $\pi_{Q_0}\circ\mathbb P_1(A^3)$, and $Y\subset \{(\gamma_1,\gamma_2)\in \pi_q(A):  \gamma_1\equiv \gamma_0(\mod Q_0)\}$, then $|X\cdot Y|=|X|\cdot |Y|$, and for the third inequality, we have fibered $\pi_q(A)$ over $\pi_{Q_0}\circ \mathbb P_1(A)$, so that $|\pi_q(A)|$ is upper bounded by the product of the number of fibers and the fiber of the maximal size. \eqref{01071} contradicts \eqref{0929} for $l=4$, if we take $L_2$ sufficiently large.


For each $\alpha\in A_1$, let $C_\alpha=\{\gamma \alpha\gamma^{-1}: \gamma\in \mathbb P_1(A)\}$. Since distinct elements in $A_1$ have distinct eigenvalues (mod $Q_0$), the sets $\{\pi_{Q_0}(C_\alpha)\}_{\alpha\in A_1}$ are mutually disjoint.  Therefore, 
\begin{align}
\label{0149}\sum_{\alpha\in A_1}|\pi_{Q_0}(C_\alpha)|=|\cup_{\alpha\in A_1}\pi_{Q_0}(C_\alpha)|< |\pi_{Q_0}(\mathbb P_1(A)^4)|<|\pi_{Q_0}(\mathbb P_1(A))|^{1+\frac{c_1}{18}}.
\end{align}
Thus, recalling \eqref{06081}, \eqref{0149} implies that there exists some $\alpha_0=\mat{\alpha_{11}&\alpha_{12}\\\alpha_{21}&\alpha_{22}}\in  A_1$,  
\begin{align}\label{1024}
|\pi_{Q_0}(C_{\alpha_0})|\leq |\pi_{Q_0}(\mathbb P_1(A))|^{1+\frac{c_1}{18} }Q_0^{-c_1/2}< |\pi_{Q_0}(\mathbb P_1(A))| Q_0^{-\frac{c_1}{3}}.
\end{align}

It follows from \eqref{1024} that there exists $x_0\in \mathbb P_1(A)$ and 

\begin{align}\label{1428}
A_2':= \{x\in \mathbb P_1(A): x\alpha_0x^{-1}= x_0\alpha_0x_0^{-1}(\mod Q_0)\},
\end{align}
such that 
\begin{align}|\pi_{Q_0}(A_2')|>Q_0^{\frac{c_1}{3}}. \label{0851}
\end{align} 
Therefore, $A_2''= x_0^{-1}\cdot A_2'$ commutes with $\alpha_0$ mod $Q_0$ and 
\begin{align}\label{1206}
\displaystyle |\pi_{Q_0}(A_{2}'')|>Q_0^{\frac{c_1}{3}},
\end{align}
Choose $A_2\subset A_2''$ to be a set of representatives for $\pi_{Q_0}(A_{2}'')$, which gives Lemma \ref{1059}.
\end{proof}

\subsection{Applying sum-product}

To proceed, we first diagonalize $\alpha_0$ in a quadratic extension of $\mathbb Q$. Let $$\lambda_1=\frac{\text{tr}(\alpha_0)+\sqrt{\text{tr}(\alpha_0)^2-4}}{2}, \hspace{5mm} \lambda_2=\frac{\text{tr}(\alpha_0)-\sqrt{\text{tr}(\alpha_0)^2-4}}{2}$$ be the two eigenvalues of $\alpha_0$.  Let $K=\mathbb Q[\lambda_1]$, and $\mathcal O= \mathcal O_K$ be the ring of integers of $K$. Let $M=\mat{\lambda_1-\alpha_{22}& \lambda_2-\alpha_{22}\\ \alpha_{21}& \alpha_{21}}$.  Then 
\begin{align}\label{1011}
M^{-1}\alpha_0M^{}=\mat{\lambda_1&0\\0&\lambda_2}.
\end{align}
 Since $\alpha_0$ satisfies the condition \eqref{0619} for $\gamma$, \begin{align}\label{0609}\lambda_1\not\equiv \lambda_2 (\mod p^{[4\rho_1n]}),\hspace{1cm} \forall p^n\Vert q_1.\end{align}   \par

Let $C_1, C_2, C_3, \varepsilon=\varepsilon_0 $ be the implied constants of Theorem \ref{sumproduct} for $\alpha=\frac{c_1}{6}$.

The first goal of this section is to prove 
\begin{lemma} \label{0804}There are scales $0<w_1, w_2<1$ with $w_1, w_2=\Theta (\rho_2)$, $w_2 \leq w_1$, and a set $H\subset M^{-1}\cdot A_2^2 \cdot M$, and an ideal $\mathcal H\supset  \left(q_1^{\{\frac{w_2}{C_1}\}}\right) $ in $\mathcal O$, such that 
\begin{enumerate}[label=\normalfont(\roman*)]
\item $H\equiv 1(\mod q_1^{\{w_1\}})$,
\item $H$ is diagonal mod $q_1^{\{w_1+w_2\}}$,
\item{ \normalfont(lower bound for dimension)} Elements in $H$ are distinct mod $q_1^{\{w_1+\frac{w_2}{C_1}\}}$ and 
$$\left\vert\pi_{q_1^{\{w_1+\frac{w_2}{C_1}\}}}(H)\right\vert >q_1^{\frac{w_2}{C_1}\cdot\frac{c_1}{5}-}. $$
\item {\normalfont (common divisor)} For each $h\in H$, writing $h\equiv \mat{a_h&0\\0&a_h^{-1}} (\mod q_1^{\{w_1+w_2\}})$ for some $a_h \in \mathcal O$, then $\langle { a_h-1}, q_1^{\{w_1+\frac{w_2}{C_1}\}}\rangle= q_1^{\{w_1\}}\mathcal H $.
\end{enumerate}
\end{lemma}

\begin{remark}
The reason that we choose the scales $w_1, w_2$ in the above way is that subsequently we will construct a set $\mathscr H\subset \mathcal O$ from $H$ and apply Theorem \ref{sumproduct} to $A=\mathscr H$, $\fa = (q_1^{\{\frac{w_2}{C_1}\}})$ and density $\alpha=\frac{c_1}{6}$. The conclusion of Theorem \ref{sumproduct} then ensures that the obtained ideal $\fa'  \supset (q_1^{\{w_2\}})$, so that we can use Property (ii) from Lemma \ref{0804} in the computation of some matrix products for the proof of Lemma \ref{1123}. Condition (iv) in Lemma \ref{0804} is not directly used in applying Theorem \ref{sumproduct} but is employed in the proof of Lemma \ref{1123}, where an element $\bar{h}\in  H$ satisfying certain valuation estimates \eqref{1652} is required (in fact this estimate holds for any element in $H$).
\end{remark}

\begin{proof}[Proof of Lemma \ref{0804}]

The denominator of $M^{-1}$ is controlled $\text{Det}(M)=\alpha_{21}\sqrt{\text{tr}(\alpha_0)^2-4}$, from which together with \eqref{0619},  we have
\begin{align} \label{1235}
p^{[8\rho_1 n]}\nmid \text{Det}(M), \hspace{1cm} \forall p^n\Vert q_1.
\end{align}   
Therefore, from Lemma \ref{1059}, since $A_2$ commutes with $\alpha_0$ mod $q_1^{\{\rho_2\}}$, we have $$H_0:= M^{-1} A_2 M$$ commutes with $M^{-1} \alpha_0 M$ mod $q_1^{\{\rho_2-8\rho_1\}}$. Moreover, because of the condition \eqref{0609}, we have, for any $\gamma \in A_2$, we can find $a_\gamma\in \mathcal O$ such that 
\begin{align} \label{05071}
M^{-1}\gamma M \equiv \mat{a_\gamma&0\\0&a_\gamma^{-1}} (\mod q_1^{\{\rho_2-12\rho_1\}}).
\end{align}
 
 We first give an estimate for $|H_0|$. The conditions (1), (2) of Lemma \ref{1059} implies that 
 \begin{align}\label{1111}
 |A_2(\mod q_1^{ \{\rho_2- 20\rho_1\}} )|> q_1^{\frac{\rho_2c_1}{4}},
 \end{align}
 by taking $L_2$ sufficiently large. \par
 If $\gamma_1 \not\equiv \gamma_2(\mod q_1^{ \{\rho_2- 20\rho_1\}} )$, then $ M^{-1}\gamma_1 M \not\equiv M^{-1}\gamma_2M (\mod q_1^{ \{\rho_2- 12\rho_1\}} )$. \eqref{1111} thus implies:
 \begin{align}\label{11321}
 | H_0(\mod q_1^{\{\rho_2-12\rho_1\}})|>q_1^{\frac{\rho_2c_1}{4}}.
 \end{align}
 
We should think of the $\rho_1$ terms in \eqref{05071}, \eqref{1111}, \eqref{11321} negligible compared to the $\rho_2$ terms, and this applies to any similar expressions. 
 
 \medskip

 The density condition \eqref{11321} implies that at east one of the following two events occurs:
{\bf Event 1}: There is a scale $w_1\in [\frac{c_1}{80}\rho_2,\frac{1}{2}(\rho_2-12\rho_1)]$, and $h_0\in H_0$ such that the set
 \begin{align}\label{1021}
 H_1: \{h\in H_0: h\equiv h_0(\mod q_1^{\{ w_1\}})\}
 \end{align}
 satisfies 
 \begin{align}\label{0352}
 |H_1 (\mod q_1^{\{w_1(1+C_1^{-1})\}})|> q_1^{\frac{w_1 }{C_1}\cdot \frac{c_1}{5} },
 \end{align}
 In this case, we set $w_2=w_1$. \par
 \noindent {\bf Event 2}: There is a scale $w_1\in [\frac{1}{2}(\rho_2-12\rho_1), (1-\frac{c_1}{80})\rho_2]$ and $h_0\in H_0$ such that the set
 \begin{align}\label{1021}
 H_1: \{h\in H_0: h\equiv h_0(\mod q_1^{\{ w_1\}})\}
 \end{align}
  satisfies 
 \begin{align}\label{0352}
 |H_1 (\mod q_1^{\{w_1+\frac{\rho_2-12\rho_1-w_1}{C_1}\}})|> q_1^{\frac{\rho_2-12\rho_1- w_1 }{C_1}\cdot \frac{c_1}{5} },
 \end{align}
In this case, we set $w_2=\rho_2-12\rho_1-w_1$. \par
 
 Indeed, suppose no such $w_1$ and $h_0$ exist. Set $t_0= \frac{c_1}{80} \rho_2, $ and $t_s=t_0(1+C_1^{-1})^s$ for $1\leq s\leq s_1$, where $s_1$ is the largest integer such that  
 $t_{s_1}\leq \frac{1}{2}(\rho_2-12\rho_1)$. 
 We continue to iteratively define $t_{s_1+1}= t_{s_1}+\frac{\rho_2-12\rho_1-t_{s_1}}{C_1}$, $t_{s_1+2}= t_{s_1+1}+\frac{\rho_2-12\rho_1-t_{s_1+1}}{C_1}\cdots $, until we reach a first integer $s_2$ such that $t_{s_2}> (1-\frac{c_1}{80})\rho_2$. 
 Then 
 \begin{align*}
 | H_0(\mod q_1^{\{\rho_2-12\rho_1\}})| &\leq |\pi_{q_1^{\{t_{s_2}\}}}(H_0)|\cdot q_1^{\frac{\rho_2c_1}{40}}\\
& \leq  |\pi_{q_1^{\{t_{s_2-1}\}}}(H_0)|\cdot q_1^{({t_{s_2}-t_{s_2-1} })\cdot \frac{c_1}{5}} \cdot q_1^{\frac{\rho_2c_1}{40}} \\
& \leq  |\pi_{q_1^{\{t_{s_2-2}\}}}(H_0)|\cdot q_1^{({t_{s_2-1}-t_{s_2-2} })\cdot \frac{c_1}{5}} \cdot q_1^{({t_{s_2}-t_{s_2-1} )}\cdot \frac{c_1}{5}}\cdot  q_1^{\frac{\rho_2c_1}{40}} \\
&   \cdots\\  &\leq  |\pi_{q_1^{\{t_{0}\}}}(H_0)|\cdot q_1^{{ t_{s_2}}\cdot \frac{c_1}{5}} \cdot q_1^{\frac{\rho_2c_1}{40}} \\
& \leq q_1^{\frac{\rho_2c_1}{40}}\cdot q_1^{{ t_{s_2}}\cdot \frac{c_1}{5}} \cdot q_1^{\frac{\rho_2c_1}{40}} \leq q_1^{\frac{\rho_2c_1}{4}},
 \end{align*}
 contradicting \eqref{11321}. So such $h_0$ and $w_1$ have to exist. \par
 Let $H_2\subset h_0^{-1}H_1$ be a set of representatives of $h_0^{-1}H_1 (\mod q_1^{\{w_1+\frac{w_2}{C_1}\}})$. Then $H_2$ satisfies the conditions (i), (ii), (iii) for Lemma \ref{0804}.
  
For each $h\in h_0^{-1}H_2$, take some $a_h\in \mathcal O$ so that 
 $$h\equiv \mat{a_h&0\\0&a_h^{-1}}(\mod q_1^{\{w_1+w_2\}} ).$$ 
 
 From \eqref{1021}, for each $h\in H_2$, we have $a_{h}\equiv 1(\mod q_1^{\{w_1\}})$. Let $\mathcal H_h$ be an ideal of $\mathcal O$ such that $q_1^{\{w_1\}} \mathcal H_h =\langle a_h-1, q_1^{\{w_1+\frac{w_2}{C_1}\}}\rangle$. Since the number of divisors of $\left(q_1^{\left\{\frac{w_2}{C_1}\right\}}\right)$ is $O(q_1^{0+})$, there is $\mathcal H$ such that 
 $$H:= \{h\in H_2: \mathcal H_h=\mathcal H\}$$
satisfies 
\begin{align}\label{1409}
|H(\mod q_1^{\{w_1+\frac{w_2}{C_1}\}})|> q_1^{\frac{w_2}{C_1}\cdot\frac{c_1}{5}-}.
\end{align}
Then $H$ satisfies all requirements in Lemma \ref{0804}.

\end{proof}
 
For each $h\in H$, there is $\lambda_h\in \mathcal O$ such that 
 \begin{align}\label{1052}
 a_h-a_{h}^{-1}\equiv q_1^{\{w_1\}} \lambda_h (\mod q_1^{\{w_1+w_2\}}).
 \end{align}
 
Let $$\mathscr H:=\{ \lambda_h: h\in H_2\}.$$

From \eqref{1409} and the fact that for an arbitrary ideal $\fa\subset \mathcal O$, the map $$(\mathcal O/\fa  )^{*}\rightarrow \mathcal O/\fa: x\rightarrow x-x^{-1}$$
is $O(|\mathcal O/\fa|^{0+})$ to one, we have 
\begin{align} 
|\mathscr H(\mod q_1^{\{\frac{w_2}{C_1}\}})|>q_1^{\frac{w_2}{C_1}\cdot \frac{c_1}{6}}.
\end{align}

For a construction in the next section, we apply Theorem \ref{sumproduct} to the set $\mathscr H$, the density $\alpha=\frac{c_1}{6}$ and the ideal $\fa=(q_1^{\{\frac{w_2}{C_1}\}})$, which gives us an ideal $\mathcal  Q\subset \fa^{C_1} = (q_1^{\{w_2\}})$, $\xi\in\mathcal O$, and some $\varepsilon_0>0$ depending only on $c_1$, such that 
\begin{align}  
&\pi_{\mathcal Q}(\mathbb Z\xi)\subset \pi_{\mathcal Q} \left(\sum_{C_3}\mathscr H^{C_2}-\sum_{C_3}\mathscr H^{C_2}\right),\\
\label{0835}&|\pi_{\mathcal Q}(\mathbb Z\xi)|> q_1^{{w_2\varepsilon_0}}. 
\end{align}
\begin{remark} \label{0922}
In fact, following the proof of Theorem \ref{sumproduct}, the ideal $\mathcal Q$ satisfies the property that different prime ideals $\mathcal P| \mathcal Q$ live above different natural primes $p$. We will assume this property of $\mathcal Q$ in the following discussion. 
\end{remark}

\subsection{Constructing a line }

Let $q_2$ be the product of all $p^n\Vert q_1$ where the projection of the arithmetic progression \eqref{0835} localized at $p$ is large, i.e. write $\mathcal Q_p= \langle \mathcal Q, p^n \rangle$, and let 
\begin{align}\label{2228}
q_2:= \prod_{\substack{p^n\Vert q_1\\ |\pi_{\mathcal Q_p} (\mathbb Z\xi) |> p^{nw_2\cdot\frac{\varepsilon_0}{2}}}} p^n .
\end{align}

The density condition \eqref{0835} implies 
\begin{align}\label{0854}
q_2>q_1^{{\varepsilon_0}/{4}}.
\end{align}
 Indeed, suppose not, then

$$|\pi_{\mathcal Q}(\mathbb Z\xi)|= \prod_{p^n\Vert q_2}  |\pi_{\mathcal Q_p}(\mathbb Z\xi)| \cdot \prod_{p^n\Vert \frac{q_1}{q_2}} |\pi_{\mathcal Q_p}(\mathbb Z\xi)|\leq q_2^2 (q_1/q_2)^{\frac{w_2\varepsilon_0}{2}} \leq q_1^{w_2 \varepsilon_0},$$
contradicting \eqref{0835}.  \par

%
%
For each $p|q_2$, let $\mathcal P_p | \mathcal Q$ be the unique prime ideal that lives above $p$ (see Remark \ref{0922}), and $\mathcal Q_p$ be the $\mathcal P_p$ power that exactly divides $\mathcal Q$. In our case, the extension degree is at most 2, so $v_{\mathcal P_p}(p)=1$ or 2. \par

Our main goal in this section is the following lemma, which says that a bounded product of $A_0$ produces a ``thick'' line in $\mathfrak {sl}_2$ at an appropriate level: 

\begin{lemma} \label{1236} There is a positive integer $K_1$ depending only on $S$, an exact divisor $q_4\Vert q_2$, $q_4\geq q_2^{1/4}$, rational integers $Q_1 | Q_2|Q_1^2| q_4$, a traceless matrix $X=\mat{a&b\\c&-a}\in {\normalfont \text{Mat}}_2(\mathbb Z)$ coprime to $q_4$, such that 
$$1+\mathbb Z Q_1 X \subset A_0^{K_1}(\mod Q_2),$$
where for each $p^n\Vert q_4$, we have 
\begin{align}
&v_p(Q_2)=\Theta(\rho_2 n) \\
& v_{p}(Q_2)-v_p(Q_1)=\Theta(\rho_2 n).
\end{align}

\end{lemma}

To prove Lemma \ref{1236}, we plan to utilise the set of commutative matrices $H$ obtained from Lemma \ref{0804}, and an element $\zeta$ given in Lemma \ref{1123} below, where $\zeta$ is very non-commutative with $H$, yet lies in certain proximity to the identity. This is quantitatively described in terms of certain valuation estimate \eqref{1107} \eqref{1108}. Then we can apply Lemma \ref{1947} iteratively, starting with $x\in H$ and $y= \zeta$, to reveal a sum-product structure \eqref{1133}, for which we can apply Theorem \ref{sumproduct} to obtain a thick one-parameter group from a bounded product of $A_0$, under the reduction of a modulus from the quadratic extension $K$ (Lemma \ref{1244}). Then it is not a difficult job to turn the modulus into a real one since our source $A_0$ is real.  

\begin{lemma} \label{1123} There is an exact divisor $q_4\Vert q_2$ with $q_4\geq q_2^{\frac{1}{4}}$, an ideal $\mathscr Q_1\supset (q_4)$, an element $\zeta \in M^{-1} A_0 H A_0^{-1} H^{-1} M$, such that 
$$\zeta\equiv 1+ \mathfrak q_1 \mat{X&Y\\Z&-X}(\mod \mathscr Q_1^2)$$
where  $\mat{X&Y\\Z&-X}$ is a primitive matrix in $\normalfont\text{Mat}_2(\mathcal O)$, and $\frak q_1 $ is an element in $\mathcal O$ that uniformizes $\mathscr Q_1$.  For each $p^n\Vert q_4$, we have the following estimates for valuations: 
\begin{align}\label{1107}
 0.99 v_{\mathcal P_p }(p^{[nw_1]}) \leq v_{\mathcal P_p}(\mathscr Q_1) \leq 1.01 v_{\mathcal P_p }(p^{[nw_1]})
\end{align}
where $w_1$ is given as in Lemma \ref{0804} and 
\begin{align}\label{1108}
v_{\mathcal P_p }(Z) =O(n\rho_1).
\end{align}

\end{lemma}  

\begin{proof}

We work with the linear form 
\begin{align}\label{1121}
 \mathcal L\left(\mat{a&b\\c&d} \right)= \alpha_{21}a+(\alpha_{22}-\alpha_{11})c-\alpha_{21}d,
 \end{align}
 which gives the 2-1 entry of $\alpha_0 \mat{a&b\\c&d}-\mat{a&b\\c&d}\alpha_0$. If \eqref{1121} is non-zero, this certainly implies $\zeta$ and $\mat{a&b\\c&d}$ are not commutative.

 Let $$\mathcal E_3(Q)=\{\gamma\in \text{SL}_2(\mathbb Z): \mathcal L(\gamma)\equiv 0(\mod Q)\}.$$

By Proposition \ref{decay1}, for any sufficiently large $Q|q_2$ , 

\begin{align}\label{1134}
\nu^{(2l)}(\mathcal E_3(Q))<  Q^{-c_1},
\end{align}
Define
$$A_3= A_0-\cup_{\substack{Q|q_2\\Q>q_2^{\rho_1}}}  \mathcal E_3(Q).$$
By \eqref{2023}, \eqref{0854} and \eqref{1134},  
$$\nu^{(2l)}(A_3)> q^{-3\delta}- q_2^{-\rho_1c_1-}>  q^{-3\delta}- q_1^{-\frac{\varepsilon_0\rho_1c_1}{4}-}>q^{-3\delta}- q^{-\frac{\varepsilon_0\rho_1c_1}{8}-}> q^{-4\delta},$$
if taking $L_1$ sufficiently large. 

Then using the same argument for constructing $A_1$ in Lemma \ref{0608}, by taking $L_1$ sufficiently large, we produce some $q_3\Vert q_2$, $q_3>q_2^{\frac{1}{2}}$, $\zeta_1 \in A_3$ such that for any $p^n\Vert q_3$, 
\begin{align} \nonumber
& \mathcal L({\zeta_1})\neq 0(\mod p^{[2\rho_1n]}),
\end{align}
which of course implies for any $p^n\Vert q_3$,

\begin{align}\label{10301}
& \mathcal L({\zeta_1})\neq 0(\mod \mathcal P_p^{ \{O(\rho_1n) \} }).
\end{align}

The matrix $\zeta_1$ itself is very non-commutative with $\alpha_0$ in the sense of \eqref{10301}. However for our purpose, we need to find a non-commutative element that is also sufficiently close to identity (at the scale $w_1$) as required by lemma \ref{1123}, and then conjugate it by $M$. \par
The modification is as follows:  First recall the diagonalization of $\alpha_0$ at \eqref{1011} and that $\lambda_1-\lambda_2= \sqrt{\text{tr}(\alpha_0)^2-4}$.  For any $p^n\Vert q_3$, since $v_{\mathcal P_p}(\text{tr}(\alpha_0)^2-4)=O(\rho_1 n)$, we have 
\begin{align}\label{1026}
\lambda_1\not\equiv \lambda_2(\mod \mathcal P_p^{\{O(\rho_1 n)\}}).
\end{align}
From \eqref{10301} and \eqref{1026}, we have $$\zeta_2=M^{-1}\zeta_1 M= \mat{a_{11}&a_{12}\\a_{21}&a_{22}}$$ with $a_{12}$ or $a_{21}\not\equiv 0 (\mod \mathcal P_p^{\{O(\rho_1 n)\}})$. Therefore, there exists an exact divisor $q_4 \Vert q_3, q_4\geq q_3^{\frac{1}{2}}$ such that one of $a_{12}$ and $a_{21}$, let us say $a_{21}$ without loss of generality, satisfies:  $\forall p^n\Vert q_4$, 
\begin{align}\label{1032}
a_{21}\not\equiv 0(\mod \mathcal P_p^{\{O(n\rho_1)\}}).
\end{align}
Take any element $\bar{h}\in H $ as given by Lemma \ref{0804} with $\bar{h}\equiv \mat{a_{\bar h} & 0\\0& (a_{\bar h})^{-1}} (\mod q_4^{\{w_1+w_2\}}) $. A direct computation gives
\begin{align}
\label{2101}
\zeta= [\zeta_2, \bar{h}]=\zeta_2 \bar {h} \zeta_2^{-1}{\bar {h}}^{-1} \equiv 1+\mat{a_{12}a_{21}(1-(a_{\bar h})^{-2})& a_{11}a_{12}(1-(a_{\bar h})^2)\\ a_{21}a_{22}(1-(a_{\bar h})^{-2})& a_{12}a_{21}(1- (a_{\bar h})^2) } (\mod q_4^{\{w_1+w_2-O(\rho_1)\}}),
\end{align}
where again we have a loss of an insignificant $O(\rho_1)$ term in the exponent of $q_4$ because of the denominator of $M$. \par

We have the following estimate for  $v_{\mathcal P_p}(\bar{h}-1)$:
\begin{align}\label{1652}
v_{\mathcal P_p}(q_4^{\{w_1\}}) \leq v_{\mathcal P_p}(\bar{h}-1)< v_{\mathcal P_p}(q_4^{\{w_1+\frac{w_2}{C_1}\}}).
\end{align}\label{1051}
The second inequality of \eqref{1652} in particular holds, because otherwise, for any $h\in H$,
\begin{align}\label{10521}
v_{\mathcal P_p}(\lambda_h)=v_{\mathcal P_p}(\lambda_{\bar h})\geq v_{\mathcal P_p}(q_4^{\{\frac{w_2}{C_1}\}})=v_{\mathcal P_p}(q_1^{\{\frac{w_2}{C_1}\}}),
\end{align}
 recalling \eqref{1052} and Property (iv) in Lemma \ref{0804}. Following the proof of Theorem \ref{sumproduct} in \cite{TZ23a}, \eqref{10521} would imply the implied arithmetic progression $\mathbb Z\xi$ has no contribution at the modulus $\mathcal Q_p$, i.e., $\mathbb Z\xi (\mod \mathcal Q_p)=\{0\}$. This contradicts the thickness bound of the projection of the arithmetic progression $\mathbb Z\xi$ at $\mathcal Q_p$ in the definition of $q_2$ at \eqref{2228}. So \eqref{1652} holds. \par

Let $\mathscr Q_1$ be the ideal of $\mathcal O$ generated by $q_4$ and all entries of $\xi-1$. 

Write
$$\zeta\equiv 1+ \mathfrak q_1 \mat{X&Y\\Z&-X}(\mod \mathscr Q_1^2),$$
where $\frak q_1 $ is an element in $\mathcal O$ that uniformizes $\mathscr Q_1$, and $\mat{X&Y\\Z&-X}$ is a primitive matrix in $\text{Mat}_2(\mathcal O)$.  From \eqref{2101} and \eqref{1032}, we have for any $p^n\Vert q_4$,
\begin{align}\label{0926}
v_{\mathcal P_p} (\mathscr Q_1) =v_{\mathcal P_p}(\mathfrak q_1)= v_{\mathcal P_p} (\bar{h}-1)\pm O(\rho_1 n)
\end{align}
and 
\begin{align}\label{0927}
v_{\mathcal P_p} (Z) =  O(\rho_1 n).
\end{align}

Requiring $C_1>100$ and taking $L_2$ sufficiently large then give \eqref{1107} in Lemma \ref{1123}.

\end{proof}
Recall Lemma \ref{1946}, which has an obvious generalization to general ideals of a general ring of integers $\mathcal O$ by multiplicativity: \par
\begin{lemma} \label {1947} Let $\mathcal I_1, \mathcal I_2$ be two ideals of $\mathcal O$, and let $x, y\in \normalfont\text{SL}_2 (\mathcal O)$, $x\equiv 1 (\mod \mathcal I_1), y\equiv 1 (\mod \mathcal I_2)$. Then 
$$xyx^{-1}y^{-1}\equiv 1 (\mod \mathcal I_1\mathcal I_2),$$
and 
$$xyx^{-1}y^{-1}\equiv 1+xy-yx (\mod \mathcal I_1\mathcal I_2\langle \mathcal I_1, \mathcal I_2 \rangle).$$
\end{lemma}

To apply Lemma \ref{1947}, we take $x=\zeta$ given by Lemma \ref{1123} with $\mathcal I_1=\mathscr Q_1$, and $y=h_1\in H$ given in Lemma \ref{0804} with $\mathcal I_2=\mathscr Q_2 :=(q_4^{\{w_1\}})$. \par 
Write $\mathscr Q_3=(q_4^{\{w_2-O(\rho_1)\}})$, where the implied constant for $O(\rho_1)$ is taken so that $\mathscr Q_3 $ divides $ \langle \mathscr Q_1, \mathscr Q_2  \rangle $, noticing that for each $p^n\Vert q_4$, $v_{\mathcal P_p}( \langle \mathscr Q_1, \mathscr Q_2  \rangle) = v_{\mathcal P_p}(q_4^{\{w_1\}})\pm O(\rho_1 n)$. \par

Recall $h_1\equiv \mat{a_{h_1}&0\\0& a_{h_1}^{-1}} (\mod \mathscr Q_2 \mathscr Q_3)$ and $ a_{h_1}-a_{h_1}^{-1}= q_1^{\{w_1\}}  \lambda_{h_1}$. Also write $\mathfrak q_2= q_1^{\{w_1\}} $, which is a uniformizer for $(q_4^{\{w_1\}})$.
Then from Lemma \ref{1947}, 
\begin{align}
\zeta':= \zeta h_1 \zeta^{-1} h_1^{-1} \equiv 1 (\mod \mathscr Q_1\mathscr  Q_2 ) 
\end{align}
and
\begin{align}
\nonumber\zeta' & \equiv 1+  \zeta h_1-h_1\zeta \hspace{2mm} (\mod \mathscr Q_1 \mathscr Q_2\mathcal  \langle \mathscr Q_1, \mathscr Q_2\rangle ) \\
\nonumber & \equiv 1+  (\zeta-1) h_1-h_1(\zeta-1) \hspace{2mm} (\mod \mathscr Q_1 \mathscr Q_2\mathcal  \langle \mathscr Q_1, \mathscr Q_2\rangle ) \\
\Rightarrow \hspace{0.3cm}& \zeta'   \equiv 1+ \mathfrak q_1 \mathfrak q_2  \mat{0&-\lambda_{h_1}Y\\ \lambda_{h_1}Z& 0 }  (\mod \mathscr Q_1 \mathscr Q_2\mathscr Q_3),
\end{align}

Next, applying Lemma \ref{1947} with $x=\zeta'$, $\mathcal I_1= \mathscr Q_1\mathscr Q_2$, $y= h_2\in H$, $\mathcal I_2=\mathscr Q_2$, we obtain 
$$\zeta' h_2 \zeta'^{-1} h_2^{-1} \equiv 1 (\mod \mathscr Q_1\mathscr  Q_2^2) $$
and 
\begin{align}\label{2032}
\zeta' h_2 \zeta'^{-1} h_2^{-1} \equiv 1+ \zeta' h_2- h_2\zeta' \equiv 1+(\zeta'-1)h_2- h_2(\zeta'-1)  (\mod \mathscr Q_1\mathscr  Q_2^2 \mathscr Q_3),
\end{align}
where we notice that $\mathscr Q_3| \mathscr Q_2$. \par
As $\zeta'-1\equiv 0(\mod \mathscr Q_1\mathscr Q_2)$ and $h_2\equiv \mat{a_{h_2} & 0\\ 0& a_{h_2}^{-1} }(\mod\mathscr Q_2 \mathscr Q_3)$, we have 
\begin{align}\label{2033}
(\zeta'-1)h_2- h_2(\zeta'-1)&  \equiv (\zeta'-1)\mat{a_{h_2} & 0\\ 0& a_{h_2}^{-1} }-\mat{a_{h_2} & 0\\ 0& a_{h_2}^{-1} } (\zeta'-1)(\mod \mathscr Q_1\mathscr Q_2^2\mathscr Q_3)   \\
&\equiv \mathfrak q_1\mathfrak q_2^2 \mat{0& \lambda_{h_1}\lambda_{h_2} Y\\ \lambda_{h_1}\lambda_{h_2}Z & 0} (\mod \mathscr Q_1 \mathscr Q_2^2 \mathscr Q_3),
\end{align}
\eqref{2032} and \eqref{2033} then gives 
$$\zeta' h_2 \zeta'^{-1} h_2^{-1} \equiv 1+  \mathfrak q_1\mathfrak q_2^2 \mat{0& \lambda_{h_1}\lambda_{h_2} Y\\ \lambda_{h_1}\lambda_{h_2}Z & 0} (\mod \mathscr Q_1 \mathscr Q_2^2 \mathscr Q_3).$$

Applying Lemma \ref{1947} iteratively, we have for any $h_1, h_2, \cdots, h_{C_3}\in \mathscr H$, 

\begin{align} 
&[\cdots [[\zeta, h_1],h_2],\cdots h_{C_3}] \equiv 1 (\mod \mathscr Q_1 \mathscr Q_2^{C_3}),\\
& [\cdots [[\zeta, h_1],h_2],\cdots h_{C_3}] \equiv 1 + \frak q_1 \frak q_2^{C_3}  \mat{0&(\prod_{i=1}^{C_3} \lambda_{h_i}) (-1)^{C_3} \cdot Y\\ (\prod_{i=1}^{C_3} \lambda_{h_i})Z & 0} (\mod \mathscr Q_1 \mathscr Q_2^{C_3}\mathscr Q_3)
\end{align}

The sum structure is clear.  For any $h_{1,1},\cdots,h_{C_2,C_3}\in \mathscr H$, write $$u_j=[\cdots [[\zeta, h_{j,1}],h_{j,2}],\cdots h_{j, C_3}], \hspace{0.5cm} 1\leq j\leq C_2.$$

We have 
\begin{align}\label{1133}  u_1u_2\cdots u_{C_2}  \equiv 1 +\frak q_1\mathfrak q_2 ^{C_3}\sum_{i=1}^{C_2}\left(\prod_{j=1}^{C_3} h_{\sigma_{i,j}}\right) \mat{0& (-1)^{C_3}Y\\ Z& 0} (\mod \mathscr Q_1 \mathscr Q_2^{C_3}\mathscr Q_3)
\end{align}

From \eqref{0835} and \eqref{1133}, we conclude there is a constant $K_1$ depending only on $S$ such that 
\begin{align} \label{1839}
1+ \mathfrak q_1\mathfrak q_2^{C_3} \mathbb Z\xi \mat{0&(-1)^{C_3}Y\\Z&0 } \subset M^{-1}A_0^{K_1}M (\mod \mathscr Q_1\mathscr  Q_2^{C_3}  \mathcal Q')
\end{align}
where $\mathcal Q'=\langle  \mathcal Q, \mathscr Q_3  \rangle$ with $\mathcal Q$ given in \eqref{0835}.  \par
Now for each $p^n\Vert q_4$, we measure the thickness of the above progression localized at $\mathcal P_p$ in terms of the valuation $v_{\mathcal P_p}$.  We have
\begin{align}
v_{\mathcal P_p}(\mathscr Q_1\mathscr Q_2^{C_3} \mathcal Q')- v_{\mathcal P_p}(\mathfrak q_1\mathfrak q_2^{C_3}\xi Z ) \stackrel{\eqref{1108}}{=} v_{\mathcal P_p}(\mathcal Q')- v_{\mathcal P_p}(\xi)-O(\rho_1 n)=  v_{\mathcal P_p}(\mathcal Q) - v_{\mathcal P_p}(\xi)\pm O(\rho_1 n), \nonumber
\end{align}
where for the second equality we have used that $\mathcal Q$ divides $ (q_4^{\{w_2\}})$ and $\mathscr Q_3=(q_4^{\{w_2-O(\rho_1)\}})$. From the definition of $q_2$ \eqref{2228}, we have 
$$v_{\mathcal P_p}(\mathcal Q)-v_{\mathcal P_p}(\xi)>\frac{nw_2\varepsilon_0}{4},$$
and so 
\begin{align} \label{2250}
v_{\mathcal P_p}(\mathscr Q_1\mathscr Q_2^{C_3} \mathcal Q')- v_{\mathcal P_p}(\mathfrak q_1\mathfrak q_2^{C_3}\xi Z ) > \frac{nw_2\varepsilon_0}{5}
\end{align}
by taking $L_2$ large.  \par
Collecting \eqref{1839} and \eqref{2250}, and conjugating \eqref{1839} back by $M$, we obtain 

\begin{lemma}\label{1244} There is a constant $K_1$ depending only on $S$ and $\varepsilon$, an exact divisor $q_4\Vert q_2$, $q_4\geq q^{\Omega (1)}$, an ideal $\mathcal U \subset \mathcal O$ dividing $(q_4)$, an element $\frak u \in \mathcal O$, and a primitive traceless matrix $ W\in \normalfont\text{ Mat}_2(\mathcal O)$, such that  
\begin{align}
1+ \mathbb Z\mathfrak u W \subset A_0^{K_1} (\mod \mathcal U),
\end{align}
where for each $p^n\Vert q_4$, we have 
\begin{align}
&\label{1431} v_{\mathcal P_p} (\mathcal U)=\Theta (\rho_2 n), \\
&\label{1432} \frac{v_{\mathcal P_p}(\mathcal U)}{2} \geq  v_{\mathcal P_p} (\mathcal U)- v_{\mathcal P_p} (\mathfrak u)\geq \Omega (\rho_2 n). 
\end{align}

\end{lemma}

\begin{proof}[Proof of Lemma \ref{1236}] Let $r_m\in A_0^{K_1}, m\in\mathbb Z$ such that 
\begin{align} \label{1341}
r_m\equiv 1+m\frak u W (\mod \mathcal U)
\end{align}
as given in Lemma \ref{1244}. 
Write 
$$r_1\equiv 1+Q_1 X (\mod Q_1^2).$$
where $Q_1| q_4$ and $X\in \text{Mat}_2(\mathbb Z)$ is traceless and coprime to $q_4$. The choice for $Q_1$ is unique. \par
For each $p^n\Vert q_4$, we have 
\begin{align}
v_{\mathcal P_p}(Q_1)=v_{\mathcal P_p}(\mathfrak u). 
\end{align}
and so 
\begin{align}\label{1354}
v_{p}(Q_1)=\frac{v_{\mathcal P_p}(\mathfrak u)}{v_{\mathcal P_p}(p)}. 
\end{align}

It follows from \eqref{1341} that 
\begin{align} \label{1346}
\forall m\in\mathbb Z, \hspace{0.5 cm} r_1^mr_m^{-1}\equiv 1(\mod \mathcal U). 
\end{align}
Taking Galois conjugate of \eqref{1346}, we obtain
\begin{align} \label{1347}
\forall m\in\mathbb Z, \hspace{0.5 cm} r_1^mr_m^{-1}\equiv 1(\mod \bar{ \mathcal U}) 
\end{align}

 Let $\mathcal U\cap \bar{\mathcal U}=(Q_2)$ for some $Q_2\in \mathbb N$, $Q_2| q_4$. Then \eqref{1346} and \eqref{1347} gives  
 \begin{align} \label{1348}
\forall m\in\mathbb Z, \hspace{0.5 cm} r_1^mr_m^{-1}\equiv 1(\mod Q_2)
\end{align}
by Chinese Remainder Theorem, which implies 
\begin{align}
r_m\equiv 1+mQ_1 X(\mod Q_2).
\end{align}
We also have $v_{\mathcal P_p}(Q_2)=v_{\mathcal P_p}(\mathcal U)$, which can be seen from the prime factorization of $\mathcal U$, so 
\begin{align}\label{1430}
v_{p}(Q_2)= \frac{v_{\mathcal P_p} (Q_2)}{v_{\mathcal P_p}(p)}= \frac{v_{\mathcal P_p }(\mathcal U)}{v_{\mathcal P_p} (p)}.
\end{align}
\eqref{1354} and \eqref{1430} give
\begin{align}\label{1431}
v_p(Q_2)-v_p(Q_1)=\frac{v_{\mathcal P_p}(\mathcal U)-v_{\mathcal P_p}(\mathfrak u)}{v_{\mathcal P_p}(p)}.
\end{align}
\eqref{1431}, \eqref{1432}, \eqref{1354}, \eqref{1430} and \eqref{1431} then give the desired estimates for valuations in Lemma \ref{1236}. 
 
\end{proof}

\subsection{Constructing a segment}\label{0614}
We want to find $g_1, g_2$ from a bounded product of $\mathbb P_1(A)$ so that $g_1, g_2$ conjugates the arithmetic progressions given in Lemma \eqref{1236} to different directions. Multiplying them together then creates a thick segment. Lemma \ref{segment1} is our goal. \par
Identify $\text{Lie(SL}_2)(\mathbb Z)$ with $\mathbb Z^3$ by $\mat{a&b\\c&-a}\mapsto (a, b, c)^t$.  Under this identification, write $X_0=(x_1, x_2, x_3)$.  \par 
Choose an element $T_1\in \text{SL}_3(\mathbb Z)$ so that $T_1 (x_1, x_2, x_3)^t=(1,0,0)^t$. Define $\mathcal L_1: \mathbb Z^3\rightarrow \mathbb Z$ by $\mathcal L_1(a, b, c)=b$. Then $\mathcal L_1\circ T_1$ is a primitive linear form on $\mathbb Z^3$.

We apply Proposition \ref{decay2} {with $\xi_1=X_0 ,\xi_2=\eta_2=0$, and a proper choice of $\eta_1$ so that $$\text{Tr}(g X_0 g^{-1}\eta_1)=\mathcal  L_1\circ T_1 (g X_0 g^{-1}).$$} \par
Recall $c_2$ be the implied constant from Propositions \ref{decay2}. Then we obtain, for any sufficiently large $Q>q_4^{\rho_1}$,
$$\nu^{(l)}(\{ g\in \text{SL}_2(\mathbb Z)| \mathcal L_1\circ T_1 (g  X_0 g^{-1})\equiv 0(\mod {Q})\})<Q^{-c_2},$$
where $l$ is as given in Proposition \ref{2132}. This implies, arguing in the same way as for the construction of $A_1$ in \S \ref{bg}, that there is $q_5\Vert q_4$, $q_5>(q_4)^{\frac{1}{2}}$, and $g_1\in \mathbb P_1(A)$ such that for any $p^n\Vert q_5$, $$\mathcal L_1\circ T_1(g_1 X g_1^{-1} )\neq 0(p^{[2\rho_1 n]}),$$ by taking $L_1$ at \eqref{0108} sufficiently large.

Let $X_1=g_1X_0 g_1^{-1}$, then we can write $$X_1\equiv a(1, 0, 0)+ b\vec {v} (\mod q_5), $$ where $a, b\in \mathbb Z$ with $b$ satisfying, for any $p^n\Vert q_5$, $p^{[2\rho_1 n]}\nmid b$, and $\vec v$ is a primitive vector in the $\mathbb Z$-span of $(0,1,0)$ and $(0, 0, 1)$. \par 
Next, we choose an element $T_2\in \text{SL}_3(\mathbb Z)$ such that $T_2( X_0)=(1,0,0)$ and $T_2(X_1)=(0,b,0)$.  Let $\mathcal L_2: \mathbb Z^3\rightarrow \mathbb Z$ be the linear form getting the third component.  Then $\mathcal L_2\circ T_2$ is a primitive linear form on $\mathbb Z^3$.
Applying Proposition \ref{decay2} again, for any $Q>q_5^{\rho_1}$, we have 
$$\nu^{(l)}(\{ g\in \text{SL}_2(\mathbb Z)| \mathcal L_2\circ T_2 (g  X_0 g^{-1})\equiv 0(\mod {Q})\})<Q^{-c_2},$$
which then implies there is $q_6\Vert q_5$, $q_6>(q_5)^{\frac{1}{2}}$, and $g_2\in \mathbb P_1 (A)$ such that $$c:=\mathcal L_2\circ T_2(g_2 X g_2^{-1} )\not\equiv 0(p^{[2\rho_1 n]})$$
for any $p^n\Vert q_6$.
Therefore, for any $p^n\Vert q_6$,
\begin{align}\label{1543}
\text{Det}(X_0, X_1, X_2)=\text{Det}\mat{1&0&0\\0&b&0\\ *&*& c}\not\equiv 0(\mod p^{[4\rho_1 n]}). 
\end{align}

 Denoting $V=\text{Lie}(\text{SL}_2)(\mathbb Z)$. Multiplying LHS of \eqref{1236} with its $g_1$ and $g_2$ conjugates, with \eqref{1543} in mind, we obtain 
 
 \begin{lemma}\label{segment1} There is an exact divisor $q_6\Vert q_4$, $q_6\geq q_4^{1/4}$,  $Q_3 |Q_4 | q_6$, and some constant $K_2>0$ depending only on $S$ and $\varepsilon$, such that 
 \begin{align}
 \label{01021}1+ Q_3  V(\mod Q_4)\subset  A_0^{K_2} (\mod Q_4), 
 \end{align}
where for each $p^n \Vert q_6$, 
\begin{align}
&v_p(Q_4)=\Theta (\rho_2 n), \\
&  v_p(Q_4)-v_p(Q_3)=\Theta (\rho_2 n).
\end{align}

\end{lemma}

 \subsection{Proof of Proposition \ref{1631}}\label{0459}

We iteratively apply Lemma \ref{1521} to sets starting with $E\subset A_0^{K_2}$ implied by \eqref{01021}, until we create a set $E_0 $ such that 
$$1+Q_5 V(\mod Q_6) \subset E_0(\mod Q_6),$$
where for each $p^n\Vert q_6$, $v_p({Q_6})=\Theta (\rho_2 n)$ and
\begin{equation}
 v_p{(Q_6)}-v_p{(Q_5)}\geq v_p(Q_3). \label{20331}
 \end{equation}
Then we iteratively taking commutators and obtain
\begin{align*}
& E_1=[E_0, E], 1+Q_3Q_5 V(\mod Q_3Q_6) \subset E_1(\mod Q_3Q_6);\\
& \cdots \\
&E_{k}= [E_{k-1}, E], 1+(Q_3)^kQ_5 V(\mod (Q_3)^kQ_6) \subset E_k(\mod (Q_3)^kQ_6);\\
&\cdots
\end{align*} 
until we reach a set with all elements congruent to identity mod $q_6$, which takes $O(\frac{1}{\rho_2})$ steps.  Since
\begin{equation}
\eqref{20331}\Rightarrow (Q_3)^kQ_5 | (Q_3)^{k-1}Q_6 \text{ for any } k\in \mathbb{N}, 
\end{equation}
we deduce
$$E_0E_1E_2\cdots (\mod q_6)\supset \Gamma(Q_5)/\Gamma(q_6).$$
Proposition \ref{1631} is thus proved with $q_*=q_6$, $c=O(1)$, $\rho=\Theta (\rho_2)$ and $C=O(\frac{1}{\rho_2})$.

\section{A gluing technique for $\text{SL}_2(\mathbb Z)\times \text{SL}_2(\mathbb Z)$ \label{gluing}}

In this section, we continue our progress towards proving Proposition \ref{2132}. The set $S$ and the constant $\varepsilon>0$ in Proposition \ref{2132} are fixed.  We still let $c_1$ and $c_2$ be the implied constants for $c$ in Propositions \ref{decay1} and \ref{decay2} respectively. We let $c_0$ be the implied constant for $c$ in Proposition \ref{1631}.  We free all other notations, such as $\mathbb P_1, \mathbb P_2, q_1, q_2\cdots $ and $Q_1, Q_2\cdots$ newly introduced in Section \ref{bg}. The Landau notations $O, \Omega, \Theta$ describe asymptotic behaviours of functions with implied constants depending only on $S$. \par

Proposition \ref{1631} so far only covers a very large subset for some not too small modulus $q_*$ in one simple factor of the group. For the purpose of proving Proposition \ref{2132}, we need to grow $q_*$ to a very large divisor of $q$, and also cross from one simple factor to another simple factor. Our main goal in this section is to prove the following gluing tool:

\begin{proposition}\label{glue} Fix $0<\theta<\min\{10^{-12}, (c_0\varepsilon)^{10}\}$. Suppose $A$ satisfies all the assumptions of Proposition \ref{2132} but fails \eqref{2320} for some $\delta\leq c_0\varepsilon\theta$. \par
Write $$q=\prod_{i\in I}p_i^{n_i}=q_s q_l,$$ where 
$$q_s=\prod_{\substack{p_i|q\\ n_i\leq \frac{3}{\delta}}}p_i^{n_i}, \hspace{1cm} q_l=\prod_{\substack{p_i|q\\ n_i> \frac{3}{\delta}}}p_i^{n_i}.$$

Let $q_1, q_2\Vert q $ and $ q_3\Vert q_l$, with $\text{gcd}(q_1, q_3)=1$, and $q_1q_2, q_3> q^{\frac{\varepsilon c_0 }{2}}$. Let $\tilde{q}_3$ be the square free part of $q_3$. Suppose for some set $B\subset \Gamma(\tilde{q}_3) \times \Gamma \subset \Lambda$, we have 
\begin{align}\label{1515}& |\pi_{q_1, q_2}(B)|>(q_1q_2)^{3-\theta}, \\
\label{1516} &|\pi_{q_3, 1}(B)|>q_3^{3-\theta}.
\end{align}
Then there exists $q_3^*\Vert q_3$, $q_3^*>q_3^{\frac{1}{4}10^{-4}}$, such that  
\begin{align} \label{1259}
|\pi_{q_1q_3^*, q_2}(B\cup B^{-1}\cup A)^{O((\log \frac{1}{\theta})^2)}|>(q_1q_2q_3^*)^{3-O(\theta^{\frac{1}{4}})}.
\end{align}
\end{proposition}
\begin{remark} The point of Proposition \ref{glue} is that we have an increase of modulus from $(q_1, q_2)$ to $(q_1q_3^*, q_2)$, over which the projection of $B$ is large, and the increasing speed, described by $\frac{\log q_3^*}{\log q}$, is lower bounded by ${\frac{10^{-4}}{8}\varepsilon c_0}$, which only depends on $S$ and $\varepsilon$. To see the necessity of the set $A$, for instance, we can take $q_1=1$, $q_2=q_3$ and $B=\{(\gamma, \gamma)\in \Lambda: \lambda\in \Gamma\}$. Then $B\cdot B=B$ and there is no hope to expand only from $B$. 

\end{remark}
\begin{remark}
We require $q_3\Vert q_l$, so that the loss of size from $\Gamma_{q_3}$ to $\Gamma(\tilde{q_3})/\Gamma(q_3)$ is not significant. The reason we work with $\Gamma(\tilde{q_3})/\Gamma(q_3)$ instead of $\Gamma_{q_3}$ is that the former is a product of $p$-groups, from which we can deduce some simple properties for a general homomorphism $f: \Gamma_{q_1}\times \Gamma_{q_2}\rightarrow \Gamma(\tilde{q_3})/\Gamma(q_3)$. 
\end{remark}

\begin{proof}[Proof of Proposition \ref{glue}] We have a natural isomorphism 
\begin{align}\label{1548}
\Gamma_{q_1q_3} \times \Gamma_{q_2}\cong \Gamma_{q_1}\times \Gamma_{q_2} \times \Gamma_{q_3}.
\end{align}
We denote the projections of RHS of \eqref{1548} to the first, the second, and the third factor by $\mathbb P_1$, $\mathbb P_2$ and $\mathbb P_3$. We also denote the projection to $\Gamma_{q_1}\times \Gamma_{q_2}$ by $\mathbb P_{1,2}$. (Caution: to avoid confusion, these notations have different meanings from the $\mathbb P_i$'s in the statement of Proposition \ref{1631}) \par
 For every $U\subset I$, write $q^U=\prod_{j\in U}p_j^{n_j}$. In this way, we identify an exact divisor of $q$ with a subset of $I$. Write $q_1=q^{I_1}, q_2=q^{I_2},q_3=q^{I_3}$ for some $I_1, I_2, I_3\subset I$.

%
Since the set $B$ satisfies \eqref{1515} and \eqref{1516}, by Proposition \ref{0643} and Proposition \ref{1433}, there exists $Q_1|q_1, Q_2|q_2, Q_3|q_3$, such that 
\begin{align}\label{1226}
Q_1Q_2<(q_1q_2)^{80\theta}, Q_3<(q_3)^{10\theta}, 
\end{align}
and 
\begin{align}\label{2357}&\Gamma(Q_1)/\Gamma(q_1) \times \Gamma(Q_2)/\Gamma(q_2) \subset \mathbb P_{1,2}(B^{5760}). \\
\label{2038}&\Gamma (Q_3)/\Gamma (q_3)\subset \mathbb P_3 (B^{1440}). 
\end{align}

Let $$I_3' := \{j\in  I_3: v_{p_j}(Q_3)\leq \theta^{\frac{1}{2}}v_{p_j}(q_3) \}$$

The density condition \eqref{1226} implies 
\begin{align*}
 &q_3':= q^{I_3'} \geq q_3^{1-10\theta^{\frac{1}{2}}}.
 \end{align*}
 Indeed, suppose not, then $q^{I_3-I_3'}>q^{10\theta^{\frac{1}{2}}}$, and 
 $$Q_3\geq \prod_{I_3-I_3'} p_j^{v_{p_j}(Q_3) } \geq \left( \prod_{j\in I_3-I_3'} p_j^{v_{p_j}(q_3)} \right)^{\theta^{\frac{1}{2}}}= (q^{I_3-I_3'})^{\theta^{\frac{1}{2}}}>q^{10\theta},$$
 which contradicts \eqref{1226}.

%

Let $B_1\subset B^{5760}$ be a set of representatives of 
\begin{align}\label{2330}
G:=\Gamma(Q_1)/\Gamma( q_1)\times \Gamma(Q_2)/\Gamma( q_2)
\end{align}
 implied by \eqref{2357} and $B_2\subset B^{1440}$ be a set of representatives of $\Gamma( Q_3)/\Gamma( q_3)$ implied by \eqref{2038}.  \par
Let $\psi: G\rightarrow B_1$ be the inverse the map of $\mathbb P_{1,2}: B_1\rightarrow G$.  \par

For each $j\in I_3'$, we consider the map $$\psi_j: G\rightarrow  \Gamma(p_j)/\Gamma({p_j^{[{n_j \theta^{\frac{1}{4}}}] }}),  \hspace{0.5cm} \psi_j (x)= \pi_{p_j^{[n_j\theta^{\frac{1}{4}}]}}\circ \mathbb P_3( \psi (x)).$$ \par

According to Proposition \ref{BZq_coro6.9}, there are two scenarios: \par

{\bf{Event }1}: Define
$$\mathcal G_j=\{(x,y)\in G\times G|  \psi_j(xy)\neq \psi_j(x)\psi_j(y)\}$$
We have, \begin{align}\label{1746}
|\mathcal G_j|> 10^{-4}|G|^2.
\end{align}

{\bf{Event }2}:
There exists a subset $S_j\subset G$, $|S_j|\geq \frac{99}{100}|G|$ and a homomorphism from $h_j$ from $G$ to $\Gamma(p_j)/\Gamma(p_j^{[n_j\theta^{\frac{1}{4}}]})$, such that $\psi_j\equiv h_j$ over ${S_j}$.    \par
In case Event 2 happens, We consider further whether $h_j$ is trivial at the half level: \par
{\bf{Event }2.1}:  $$h_j \equiv 1 (\mod {p_j^{[\frac{1}{2} n_j\theta^{\frac{1}{4}} ]}}).$$

{\bf{Event }2.2}:  $$h_j \not\equiv 1 (\mod  {p_j^{[\frac{1}{2} n_j\theta^{\frac{1}{4}} ]}}).$$

Let $I_3'=J_1\sqcup J_2$ where $J_1$ is the collection of indices in $I_3'$ where Event 1 holds, and $J_2$ is the complement of $J_1$ in $I_3'$. In particular, for any $j\in J_2$, Event 2 holds.  We further split $J_2=J_{21}\sqcup J_{22}$, where $J_{21}$ is the collection of indices $j$ such that Event 2.1 occurs, and $J_{22}$ is the collection of indices $j$ such that Event 2.2 occurs. \par

In  \S \ref{0223}, \S \ref{0224} and \S \ref{0225}, we divide our analysis into three cases, one of which must occur. We will establish Proposition \ref{glue} for each case.

 \subsection{\label{0223}The case $q^{J_1}\geq (q_3')^{\frac{1}{2}}$} 
 
By the definition of $J_1$, we have 

\begin{align}\label{1601}
\sum_{j \in J_1}(\log p_j^{n_j})|\mathcal G_j| > {10^{-4}} \log \left(q^{J_1}\right)|G|^2.
\end{align}

The left hand side of \eqref{1601} is equal to 
\begin{align}\label{959}
\sum_{\substack {U\subset J_1 \\ U\neq \emptyset}}\log (q^U)|\cap_{j\in U}\mathcal G_j\bigcap \cap_{j\in J_1-U} \cG_j^{c}|,
\end{align}
where $\cG_j^{c}$ is the complement of $\cG_j$ in $G\times G$. \par
Split the $U$-sum on LHS of \eqref{959} into two sums $\Sigma_1+\Sigma_2$ according to whether $\log (q^U)>\frac{1}{2}\cdot 10^{-4}\log (q^{J_1})$ or not. Since $\left\{ \cap_{j\in U}\mathcal G_j\bigcap \cap_{j\in J_1-U} \mathcal G_j^{c}: {U\subset J_1} \right\}$ is a family of mutually disjoint subsets of $G\times G$ indexed by $U$, we have, 
\begin{align}\label{10071}
\sum_2=\sum_{\substack{ U\subset J_1\\ \log (q^U)\leq \frac{1}{2}\cdot 10^{-4} \log (q^{J_1}) }} \log (q^U)|\cap_{j\in U}\mathcal G_j\bigcap \cap_{j\in J_1-U} \mathcal G_j^{c}| \leq  \frac{1}{2}\cdot 10^{-4}   \log \left(q^{J_1}\right)|G|^2,
\end{align}
where we have trivially bound $\log (q^U)$ by $\frac{1}{2}\cdot 10^{-4} \log (q^{J_1})$ for each $U$. 

We then obtain from \eqref{1601} and \eqref{10071} that
\begin{align*}
\sum_{\substack{ U\subset J_1\\ \log q^U>\frac{1}{2}\cdot 10^{-4} \log (q^{J_1}) }} \log (q^U)|\cap_{j\in U}\mathcal G_j| \geq \sum_1 >  \frac{1}{2}\cdot 10^{-4}   \log \left(q^{J_1}\right)|G|^2,
\end{align*}
which implies 
\begin{align}\label{2026}
\sum_{\substack{ U\subset J_1\\ \log q^U>\frac{1}{2}\cdot 10^{-4} \log (q^{J_1}) }} |\cap_{j\in U}\mathcal G_j|  >  \frac{1}{2}\cdot 10^{-4} |G|^2.
\end{align}

Since the number of subsets of $J_1$ is $<q^{0+}$, \eqref{2026} implies there exists $U_0\subset J_1$ such that $$q^{U_0} > (q^{J_1})^{\frac{1}{2}\cdot 10^{-4}},$$ and that 
$$|\cap_{i\in U_0}\mathcal G_i|>q^{0-}|G|^2. $$

Take any $(g_1, g_2) \in \cap _{i\in U_0}\mathcal G_i$, and consider $w =\psi(g_1)\psi(g_2)\psi(g_1g_2)^{-1}$.  Then $w$ satisfies, 
\begin{align*}
& \mathbb P_{1,2} (w)=1, \\
& \pi_{p_j^{[n_j\theta^{\frac{1}{4}}]}}\circ \mathbb P_3(w)\neq 1, \hspace{0.5cm}\forall j\in U_0 . 
\end{align*}
%
%

Now we claim Lemma \ref{2247} implies there is a set $B_3\subset (B_2\cup\{w\})^{O(\log \frac{1}{\theta})}$, and an integer $Q_3'  |q_3'$, such that 
 \begin{align}\label{2259}
\nonumber & \mathbb P_{1,2}(B_3)=1, \\
& \mathbb P_3(B_3) (\mod q^{U_0})\supset \Gamma (Q_3')/\Gamma ({q}^{U_0}),
 \end{align}
 where $v_p(Q_3')= O(\theta^{\frac{1}{4}} n)$ for each $p^n\Vert q^{U_0}$. Then $$\pi_{{q}_1q^{U_0}, q_2}(B_1 B_3)= |\mathbb P_{1,2}(B_1)|\cdot |\mathbb P_3(B_3)|> ({q}_1q^{U_0}{q}_2)^{3-O(\theta^{\frac{1}{4}})},$$
 giving the proof of Proposition \ref{glue} in this case with $q_3^*= q^{U_0}$.  \par

Indeed, let $b_i= 1.5^{i}\theta^{\frac{1}{2}}, i\in \mathbb N$, and let ${i_0}$ be the smallest integer such that $b_{i_0}\geq \frac{1}{2}$. We have $i_0=O(\log\frac{1}{\theta})$. For each $j\in U_0$, we apply Lemma \ref{2247} iteratively with $\gamma_0= w$, $H= B_2$, $a=v_{p_j} (\mathbb P_3(w)-1)=O(n_j\theta^{\frac{1}{4}})$ fixed, and $b= b_0n_j, b_1 n_j, \cdots, b_{i_0} n_j$.\par Multiplying the implied sets together, we obtain
 \begin{align}\label{305}
\nonumber& \mathbb P_{1,2}([[w, B_2],B_2]^{O(\log \frac{1}{\theta})})=1 \\
& \pi_{p_j^{n_j}}\circ\mathbb P_3([[w, B_2],B_2]^{O(\log \frac{1}{\theta})})\supset \Gamma (p_j^{O(n_j\theta ^{\frac{1}{4}})})/\Gamma (p_j^{n_j}), \hspace{0.5cm} \forall j\in U_0
 \end{align}
Using multiplicativity of $\mathbb P_3 (B_2)$, we can derive \eqref{2259} from \eqref{305}.

\subsection{\label{0224}The case $q^{J_2}> (q_3')^{\frac{1}{2}}, q^{J_{21}}\geq (q^{J_2})^{\frac{1}{2}}$.}

Recall for each $j\in J_2$, we obtain a set $S_j\subset G$, $|S_j|\geq 0.99 |G|$ such that $\psi_j$ agrees with a homomorphism $h_j$ on $S_j$.  We claim there is a set $U_1\subset J_{21}$, $W_1:= \cap_{i\in U_1}S_i$, such that 
\begin{align}\label{309}
\nonumber &q^{U_1}\geq (q^{J_{21}})^{\frac{99}{200}},\\
&|W_1|> q^{0-}|G|>(q_1q_2)^{1-80\theta-},
\end{align}
and 
\begin{align}\label{2226}
\psi_j (W_1)\equiv h_j(W_1) \equiv 1 (\mod p_j^{[\frac{1}{2}n_j\theta^{\frac{1}{4}}]}), \hspace{0.5cm} \forall j\in U_1.
\end{align}  
Indeed, by the definition of $J_{21}$, we have 
\begin{align}\label{16012}
\sum_{j \in J_{21}}(\log p_j^{n_j})|S_j| > {\frac{99}{100}} \log \left(q^{J_1}\right)|G|.
\end{align}

The left hand side of \eqref{16012} is equal to 
\begin{align}\label{9592}
\sum_{\substack {U\subset J_{21} \\ U\neq \emptyset}}\log (q^U)|\cap_{j\in U}S_j\bigcap \cap_{j\in J_1-U} S_j^{c}|.
\end{align}

Split the $U$-sum on LHS of \eqref{9592} into two sums $\Sigma_1+\Sigma_2$ according to whether $\log (q^U)>\frac{1}{2}\cdot \frac{99}{100}\log (q^{J_1})$ or not. Since $\left\{ \cap_{j\in U}S_j\bigcap \cap_{j\in J_{21}-U} S_j^{c}: {U\subset J_{21}} \right\}$ is a family of mutually disjoint subsets of $G$, we have, 
\begin{align}\label{10072}
\sum_2=\sum_{\substack{ U\subset J_{21}\\ \log (q^U)\leq \frac{99}{200}\log (q^{J_{21}}) }} \log (q^U)|\cap_{j\in U}\mathcal G_j\bigcap \cap_{j\in J_{21}-U} \mathcal G_j^{c}| \leq  \frac{99}{200}   \log \left(q^{J_{21}}\right)|G|.
\end{align}

We then obtain from \eqref{16012} and \eqref{10072} that
\begin{align*}
\sum_{\substack{ U\subset J_{21}\\ \log (q^U)>\frac{99}{200} \log (q^{J_{21}}) }} \log (q^U)|\cap_{j\in U}\mathcal G_j| \geq \sum_1 >  \frac{99}{200}   \log \left(q^{J_{21}}\right)|G|,
\end{align*}
which implies 
\begin{align}\label{20262}
\sum_{\substack{ U\subset J_{21}\\ \log q^U>\frac{99}{200} \log (q^{J_{21}}) }} |\cap_{j\in U}\mathcal G_j|  >  \frac{99}{200} |G|.
\end{align}

Since the number of subsets of $J_{21}$ is $<q^{0+}$, \eqref{20262} implies there exists $U_1\subset J_{21}$ such that $$q^{U_1} > (q^{J_{21}})^{\frac{99}{200}},$$ and that 
$$|\cap_{i\in U_1}\mathcal G_i|>q^{0-}|G|$$
as desired.

From \eqref{309} and Proposition \ref{0643}, we have  
\begin{align}\label{1113}
W_1^{5760}\supset \Gamma(Q_1')/\Gamma(q_1)\times \Gamma(Q_2')/\Gamma(q_2)
\end{align}
 for some $Q_1' | q_1$, $Q_2' | q_2$, $Q_1'Q_2'<(q_1q_2)^{O(\theta)}$.  
 
 Applying Lemma \ref{1407} to $W_1^{5760}$, we obtain a set $W_2= [W_1^{5760}, W_1^{5760}]^{O(\log \frac{1}{\theta})} $, such that 
 $$ W_2 \supset \Gamma( Q_1'')/\Gamma(q_1)\times \Gamma(Q_2'')/\Gamma(q_2),$$
 where $Q_1'' =\text{gcd}(2(Q_1')^2, q_1)$ and $Q_2''=\text{gcd}(2(Q_2')^2,  q_2)$,  and 
 \begin{align}\label{1418}
  \pi_{p_j^{2[\frac{1}{2}n_j\theta^{\frac{1}{4}}]}} (\psi_j(W_2) )=1, \hspace{0.5cm}\forall j\in U_1.
  \end{align}
  
  Note that we have doubled the exponent of $p_j$ on LHS of \eqref{1418} compared to \eqref{2226}. Apply Lemma \ref{1407} again to $W_2$ and keep iterating for a total of $[\log_2({\theta^{-\frac{1}{4}}})]+1$ times, then we obtain a set $W^*\subset W_1^{O( (\log  \frac{1}{\theta})^2)}$ such that 
  
 \begin{align}
& \mathbb P_{1,2} (\psi(W^*))=W^* = \Gamma(Q_1''')/\Gamma (q_1) \times \Gamma(Q_2''')/\Gamma ( q_2), \\
&\pi_{q^{U_1}}\circ \mathbb P_3 (\psi (W^*)) =1,
 \end{align} 
  where $Q_1''' Q_2''' \leq (Q_1'Q_2')^{[\theta^{-\frac{1}{4}}]} =(q_1 q_2)^{O{(\theta)\cdot \theta^{-\frac{1}{4}}}}=(q_1 q_2)^{O(\theta^{\frac{3}{4}})}$.
  
Multiplying $W^*$ with $B_2$ then gives Proposition \ref{glue}, with $q_3^*=q^{U_1}$, where 
$$|\pi_{q_1q^{U_1}, q_2}(B_2\cdot \psi(W^*) )|=|\mathbb P_{1,2} (\psi (W^*))|\cdot | \pi_{q^{U_1}}(\mathbb P_3 (B_2))|\geq (q_1q_2)^{1-O(\theta^{\frac{3}{4}})} (q^{U_1})^{1-\theta^{\frac{1}{2}}}.$$

%
%
%
%
%
%
%
%

 \subsection{\label{0225} The case $q^{J_2}> (q_3')^{\frac{1}{2}}, q^{J_{22}}>(q^{J_2})^{\frac{1}{2}}$.} Recall in \S \ref{0223} that for each $j\in J_2$, we have a set $S_j\subset G$, $|S_j|\geq 0.99 |G|$ such that $\psi_j$ agrees with a homomorphism $h_j$ on $S_j$.  Following the reasoning for  $W_1$ in \S \ref{0224}, there is a set $J_3\subset J_{22}$, $W:= \cap_{i\in J_3}S_i$, such that 
\begin{align*}
&q^{J_3}> (q^{J_{22}})^{\frac{99}{200}},\\
&|W|> q^{0-}|G|>(q_1q_2)^{1-80\theta-},
\end{align*}
and $\psi_j (W)\equiv h_j(W) \not \equiv 1 (\mod p_j^{[\frac{1}{2}n_j\theta^{\frac{1}{4}}]})$ for any $j\in J_3$.  
The homomorphisms $\{h_j: j\in J_3\}$ uniquely determine a homomorphism 
\begin{align}\label{136}
h: G \rightarrow \Gamma_{{(q^{J_3})}^{\{  \theta^{\frac{1}{4}} \}}}
\end{align}
 such that $\pi_{p_j^{[n_j\theta^{\frac{1}{4}}]} } \circ h = h_j, \forall j\in J_3$.

Since $|W|> (q_1q_2)^{1-O(\theta)}$, by Proposition \ref{0643}, 
\begin{align}\label{405}
 G':=\Gamma( \bar{Q}_1)/\Gamma( q_1)\times  \Gamma(\bar{Q}_2)/\Gamma (q_2) \subset W^{5760}\subset G^{5760}
 \end{align}
  for some $\bar Q_1|q_1, \bar Q_2| q_2$, such that 
\begin{align}\label{11071}
\bar Q_1 \bar Q_2=(q_1q_2)^{O(\theta)}.
\end{align}

Recall $B_1$ is a set of representatives of $G$ and $\psi: G\rightarrow B_1$ is a section map. Let $B_4\subset B_1^{5760}$ be a set of representatives of $G'$. Thus  we can define a section map $\bar{\psi}:G' \rightarrow B_4$ such that $\mathbb P_{1,2}\circ \bar \psi$ is the identity map. We claim that 
\begin{align}\label{1047}
\pi_{(q^{J_3})^{\{\theta^{\frac{1}{4}}\}}}\circ \mathbb P_3 \circ \bar \psi: G'\rightarrow \Gamma_{{(q^{J_3})}^{\{\theta^{\frac{1}{4}}\}}}
\end{align}
agrees with the homomorphism $h$ at \eqref{136}.  Indeed, for any $g\in G'$, write $\bar \psi (g) = b_{1}b_2\cdots b_{5760}$, where $b_{1}, b_{2},\cdots, b_{5760} \in B_1$, and for each $1\leq i\leq 5760$. Write $b_i=\psi(g_i)$, where $g_i\in G$. So we have $g=g_1\cdots g_{5760}$. Then 
\begin{align*}
&\pi_{(q^{J_3})^{\{\theta^{\frac{1}{4}}\}}}\circ \mathbb P_3 \circ \bar \psi (g) =\pi_{(q^{J_3})^{\{\theta^{\frac{1}{4}}\}}}\circ \mathbb P_3 (b_1b_2\cdots b_{5760}) \\
=&\prod_{j=1}^{5760} \pi_{(q^{J_3})^{\{\theta^{\frac{1}{4}}\}}}\circ \mathbb P_3 \circ \psi (g_{j}) 
= \prod_{j=1}^{5760} h(g_j)=h(g).
\end{align*}

Write $$G'=\Gamma( \bar{Q}_1)/\Gamma( q_1)\times  \Gamma(\bar{Q}_2)/\Gamma (q_2)\cong \prod_{j\in I_1}\Gamma(p_i^{m_i^{(1)}})/\Gamma(p^{n_i}) \times \prod_{j\in I_2}\Gamma(p_i^{m_i^{(2)}})/\Gamma(p^{n_i}).$$
The homomorphism \eqref{1047} (which is equal to $h$) is completely factorizable, i.e. is a product of homomorphsims between local factors: 
\begin{align*}
h_{i,j}^{(1)}: \Gamma(p_i^{m_i^{(1)}})/\Gamma(p_i^{n_i}) \rightarrow \Gamma(p_j)/\Gamma(p_j^{[n_j\theta^{\frac{1}{4}}]}), i\in I_{1}, j\in J_3, \\
h_{i,j}^{(2)}: \Gamma(p_i^{m_i^{(2)}})/\Gamma(p_i^{n_i}) \rightarrow \Gamma(p_j)/\Gamma(p_j^{[n_j\theta^{\frac{1}{4}}]}), i\in I_{2}, j\in J_3.
\end{align*}
 Since $\pi_{p_j^{[\frac{1}{2}n_j\theta^{\frac{1}{4}}]}}\circ h$ is nontrivial for each $j\in J_3$ and $(q_1, q_3)=1$, by Lemma \ref{2144}, this can happen only if $j\in I_2$, so $q^{J_3}| q_2$. \par
 Moreover, the density condition \eqref{11071} implies the set 
 $$J_4=:\{i\in J_3: m_i^{(2)}\leq \theta^{\frac{1}{2}}n_i\} $$ satisfies
 \begin{align}
 q^{J_4}> (q^{J_3})^{1-O(\theta^{\frac{1}{4}})}.
 \end{align}
 Indeed, suppose not. Then 
 $$q^{O(\theta)} \stackrel{\eqref{11071}}{\geq}(q^{J_3-J_4})^{\theta^{\frac{1}{2}} } \geq   {(q^{J_3})}^{\Omega(\theta^{\frac{3}{4}})}\geq {q}^{\Omega(c_0\varepsilon \cdot \theta^{\frac{3}{4}})},$$
 a contradiction since $\theta <(c_0\varepsilon)^{10}$. \par
 
 Now we take an element $g_0\in G'$, such that 
 \begin{align*}
 & \mathbb P_1(g_0)=1,\\
 & \mathbb \pi_{p_i^{n_i}}\circ \mathbb P_2(g_0)=1, \forall i\in I_2-J_4, \\
 & \mathbb   \pi_{ p_i^{[\frac{1}{2}n_i\theta^{\frac{1}{4}}]} } \circ h(g_0)\neq 1,  \forall i\in J_4.
 \end{align*}
  Then from Lemma \ref{2144}, we have 
  \begin{align}\label{153}
  v_{p_i}(g_0-1)=O(\theta^{\frac{1}{4}}n_i), \forall i\in J_4.
  \end{align} \par 
  The element $g_0$ is almost satisfactory for our purpose, but we want to turn the big $O$ condition in \eqref{153} into a big $\Theta$ condition. For this, we take a proper power of $g_0$ to control the valuations. Let  $$g=\left(g_0^{[(q^{J_4})^{\{\frac{1}{4}\theta^{\frac{1}{4}}\}}]}\right).$$
  Then $g$ satisfies, 
  \begin{align}
 & \mathbb P_1(\bar{\psi} (g))=1,\\
  &\mathbb  \pi_{p_i^{n_i}}\circ \mathbb P_2(\bar{\psi}(g))=1, \forall i\in I_2-J_4 \\
  \label{207}&v_{ p_i}(\mathbb P_2(\bar\psi (g)-1))=\Theta(\theta^{\frac{1}{4}}n_i),  \forall i\in J_4. \\
 \label{208} & \frac{1}{4}n_i\theta^{\frac{1}{4}}\leq v_{p_i}(\mathbb P_3(\bar\psi(g)-1))\leq \frac{3}{4}n_i\theta^{\frac{1}{4}},  \forall i\in J_4.
  \end{align}
  
Write $$\bar\psi(g)\equiv ( 1+Q_5 X_2, 1+Q_4X_1,) (\mod Q_5^2, Q_4^2),$$
  with $X_1, X_2\in \text{Mat}_2{(\mathbb Z)}$ primitive and traceless, $Q_4| q^{J_4}$, $Q_5| q^{J_4}$. From \eqref{207} and \eqref{208}, we have 
  \begin{align*}
&  v_{p_j} (Q_4)=\Theta (n_i \theta^{\frac{1}{4}}) \\
   & \frac{1}{4}n_i\theta^{\frac{1}{4}}\leq v_{p_i}(Q_5)\leq \frac{3}{4}n_i\theta^{\frac{1}{4}}, \hspace{0.5cm} \forall i\in J_4.  
  \end{align*}
  Let $Q_6= \text{gcd} (Q_5^2, {({q^{J_4}})^{\{\theta^{\frac{1}{4}}\}}})$, recalling that $\mathbb P_3\circ \bar\psi$ is a homomorphism up to reduction by ${({q^{J_4}})^{\{\theta^{\frac{1}{4}}\}}}$. 
  
  We consider the family 
  \begin{align} \label{1429}
  \Xi=\{\bar{\psi} (g^n): n\in\mathbb Z \} \subset B_4.
  \end{align}
  Each $\bar{\psi} (g^n)$ satisfies 
  \begin{align*}
 & \mathbb P_1(\bar{\psi} (g^n))=1\\
 &\mathbb  \pi_{p_i^{n_i}}\circ \mathbb P_2(\bar{\psi}(g^n))=1,\hspace{0.5cm} \forall i\in I_2-J_4 \\
&\bar{\psi} (g^n)\equiv ( 1+nQ_5 X_2, 1+ nQ_4 X_1) (\mod Q_6, Q_4^2)
  \end{align*}
with $$v_{p_j}(Q_4), v_{p_j}(Q_5)=\Theta(n_j\theta^{\frac{1}{4}})$$ and $$\frac{1}{3}v_{p_j}(Q_5) \leq v_{p_j} (Q_6)- v_{p_j} (Q_5)\leq v_{p_j}(Q_5), \hspace{0.5cm}\forall j\in J_4.$$

With $\Xi$ at hand, which plays the same role as the one-parameter group given by Lemma \ref{1236}, we can run the same arguments in \S \ref{0614} and \S \ref{0459}. Recall $\delta<\theta$. We start by applying Proposition \ref{decay2} to find $g_1, g_2, g_3, g_4, g_5 \in A$ to conjugate $X=(X_2, X_1)$ so that 
$$\text{gcd}(\text{Det}(X, g_1Xg_1^{-1},g_2X g_2^{-1}, g_3Xg_3^{-1},g_4X g_4^{-1},g_5Xg_5^{-1} ), q^{J_4})= (q^{J_4})^{O(\theta)},$$
which implies a set $J_5 \subset J_4$, such that 
\begin{align*}
&q^{J_5}>(q^{J_4})^{1-O(\theta^{\frac{1}{2}})}\\
&v_{p_j} (\text{Det}(X, g_1Xg_1^{-1},g_2X g_2^{-1}, g_3Xg_3^{-1},g_4X g_4^{-1},g_5Xg_5^{-1} ))<n_j\theta^{\frac{1}{2}}, \forall j\in J_5.
\end{align*}

 Then we apply the argument \S \ref{0459} to create a set $B_5\subset \{B_4 \cup A\}^{O(\log \frac{1}{\theta})}$ such that 
 $$|\pi_{q^{J_5}, q^{J_5}} (B_5) | >(q^{J_5})^{6-O(\theta^{\frac{1}{4}})},$$
 recalling $\pi_{q^{J_5}, q^{J_5}}$ is the reduction map from $\Gamma_{q_1}\times \Gamma_{q_2}$ to $\Gamma_{q^{J_5}}\times \Gamma_{q^{J_5}}$.
 Therefore, $B_5$ has the following property: 
  \begin{align*} 
 & \mathbb P_1(B_5)=1\\
 &\mathbb  \pi_{p_i^{n_i}}\circ \mathbb P_2(B_5)=1, \forall i\in I_2-J_4 \\
& |\pi_{q^{J_4}, q^{J_4}} (B_5) |\geq |\pi_{q^{J_5}, q^{J_5}} (B_5) | >(q^{J_5})^{6-O(\theta^{\frac{1}{4}})}  = (q^{J_4})^{6-O(\theta^{\frac{1}{4}})}.
  \end{align*}
  So there must be some element $x_0\in \Gamma_{q^{J_4}}$ such that $B_{x_0}:= \{ g\in B_5: \pi_{q^{J_4}}\circ \mathbb P_2(g)=x_0 \}$ satisfies  
 \begin{align*}
 & |\mathbb P_{1, 2}(B_{x_0})|=1 \\
 &|\pi_{q^{J_4}}\circ\mathbb P_3 (B_{x_0})|=(q^{J_4})^{3-O(\theta^{\frac{1}{4}})}. 
  \end{align*}
Then $|\pi_{q_1q^{J_4}, q_2}(B_1\cdot B_{x_0})|= |\mathbb P_{1,2} (B_1) | |\pi_{q^{J_4}}\circ\mathbb P_3 (B_{x_0})|> (q_1q_2q^{J_4})^{3-O(\theta^{\frac{1}{4}})}$. Proposition \ref{glue} is thus proved in this case as well with $q_3^*=q^{J_4}$.
\end{proof}

\section{\label{semisimple}Proof of Proposition \ref{2132} for $\Lambda=\text{SL}_2(\mathbb Z)\times \text{SL}_2(\mathbb Z)$}
From \S\ref{bg} and \S\ref{gluing}, we retain only the notations from the statements of Proposition \ref{1631} and Proposition \ref{glue}, and we do not keep any other newly introduced notations in their proofs (such as $q_1, q_2, I_1, J_1, U_1$, etc.). \par

We first prove Proposition \ref{2132} for all $q=\prod_{i\in I} p_i^{n_i}$ with $n_i>L, \forall i\in I$, where $L>0$ depends only on $S$ and $\varepsilon$ and is determined at \eqref{1144}. 
Suppose $A\subset \Lambda$ satisfies \eqref{1131}, but fails \eqref{2320}, that is, $A$ satisfies 
\begin{align}\label{19551}
|\pi_q(A\cdot A\cdot A)|\leq |\pi_q(A)|^{1+\delta}.
\end{align}
 We will arrive at a contradiction when $\delta$ is sufficiently small. \par 
We first observe if the first two conditions of \eqref{1131} and \eqref{19551} hold, then they also hold with $q$ replaced by any exact divisor $q'$ of $q$ such that $q'>q^{\frac{\varepsilon}{2}}$, and $\delta$ replaced by any 
\begin{align}\label{519}
\delta_0\geq\frac{48\delta}{\varepsilon c_1}.
\end{align}
 Indeed, \eqref{19551} implies a similar product bound with modulus $q'$. Divide $\pi_{q}(A)$ into congruence classes mod $q'$, then at least one congruence class has cardinality $\geq \frac{|\pi_q(A)|}{|\pi_{q'}(A)|}$. Recall by Proposition \ref{decay1}, $|\pi_{q'}(A)|>(q')^{\frac{c_1}{2}}$ if 
\begin{align}\label{520}
\delta\leq \frac{c_1}{2}.  
\end{align}
Therefore, 
\begin{align*}
&|\pi_{q'}(A)\cdot\pi_{q'}(A)\cdot \pi_{q'}(A)| \cdot  \frac{|\pi_q(A)|}{|\pi_{q'}(A)|} \leq |\pi_{q}(A)\cdot\pi_{q}(A)\cdot \pi_{q}(A) \cdot \pi_{q}(A)|\leq  |\pi_{q}(A)|^{1+2\delta} \\
\Rightarrow  & \frac{|\pi_{q'}(A)\cdot\pi_{q'}(A)\cdot \pi_{q'}(A)|}{|\pi_{q'}(A)|} \leq  |\pi_q(A)|^{2\delta}< q^{12\delta}<(q')^{\frac{24\delta}{\varepsilon}}<|\pi_{q'}(A)|^{\frac{48\delta}{\varepsilon c_1}}\leq |\pi_{q'}(A)|^{\delta_0}.
\end{align*}

We take $\delta_0$ small enough so it satisfies the ``sufficiently small'' requirement for $\delta$ at Proposition \ref{1631}.  A proper choice for $\delta_0$ is made at \eqref{1016}. Set 
\begin{align}\label{1144}
L=\frac{3}{\delta_0},
\end{align}
 so that Proposition \ref{1631} is applicable with $q$ replaced by all exact divisors $q'>q^{\frac{\varepsilon}{2}}$ and $\delta$ replaced by $\delta_0$. We denote the constants $\rho_0=\rho(\delta_0), C_0=C(\delta_0), c_0=c$ implied by Proposition \ref{1631}.  \par
We first apply Proposition \ref{1631} to obtain $q_1\Vert q, q_1\geq q^{c_0}$ such that  
$$\Gamma(q_1^{\{\rho_0\}})/\Gamma(q_1)\subset  \mathbb P_1(A)^{C_0} (\mod q_1).$$
If $q_1> q^{1-\frac{\varepsilon}{2}}$, we are done. Otherwise, apply Proposition \ref{1631} again, with $q$ replaced by $\frac{q}{q_1}$, we obtain $q_2'\Vert \frac{q}{q_1}$, $q_2'\geq(\frac{q}{q_1})^{c_0}\geq q^{\frac{\varepsilon c_0}{2}}$, such that 
$$\Gamma((q_2')^{\{\rho_0\}})/\Gamma(q_2')\subset  \mathbb P_1(A)^{C_0} (\mod q_2').$$

We will take 
\begin{align}\label{521}
\delta\leq c_0\epsilon\rho_0,
\end{align}
 so we can apply Proposition \ref{glue} to obtain $q_2''\Vert q_2$, $q_2''>(q_2')^{\frac{1}{4}10^{-4}}$, such that

$$ \left\vert\pi_{q_1q_2''}  \circ \mathbb P_1 (A^{C_0K_1[\log \frac{1}{\rho_0}]^2}) \right\vert > (q_1q_2'')^{3-K_2 \rho_0^{\frac{1}{4}}} $$
for some two absolute constants $K_1, K_2>0$ implied by the two big O notations in \eqref{1259}.

Write $q_2=q_1q_2''$. If $q_2>q^{1-\frac{\varepsilon}{2}}$, we are done. Otherwise, we apply Proposition \ref{1631} and Proposition \ref{glue} again to find $q_2\Vert q_3\Vert q$, with $\frac{q_3}{q_2}>q^{\frac{10^{-4}\varepsilon c_0}{8}}$, such that 

$$ \left\vert\pi_{q_3}  \circ \mathbb P_1 (A^{C_0K_1^2[\log \frac{1}{\rho_0}]^4}) \right\vert > q_3^{3-K_2^{\frac{5}{4}} \rho_0^{\frac{1}{16}}}, $$
and we keep going, until we find an integer $q_T\Vert q$, $q_T>q^{1-\frac{\varepsilon}{2}}$, with 
 \begin{align}\label{0902}
  \left\vert\pi_{q_T}  \circ \mathbb P_1 (A^{C_0K_1^T[\log \frac{1}{\rho_0}]^{2T}}) \right\vert =|\pi_{q_T, 1}(\cdots)| > q_T^{3-K_2^{2} \rho_0^{(\frac{1}{4})^T}}.
  \end{align}
The number of iterations $T$ is bounded by $\frac{8\cdot 10^4}{\varepsilon c_0}$. \par
Now we apply Proposition \ref{1631} and Proposition \ref{glue} to the $A$-power at \ref{0902} to expand on the second modulus. Iterate for another $T'\leq \frac{8\cdot 10^4}{\varepsilon c_0}$ times, we obtain an integer $q_{T'}'\Vert q$, $q_{T'}'>q^{1-\frac{\varepsilon}{2}}$, such that 
 \begin{align}\label{0912}
  \left\vert\pi_{q_T, q_{T'}'}   (A^{C_0K_1^{T+T'}[\log \frac{1}{\rho_0}]^{2(T+T')}}) \right\vert > (q_Tq_{T'}')^{3-K_2^{2} \rho_0^{(\frac{1}{4})^{T+T'}}}.
  \end{align}

The exponent for $A$ at \eqref{0912} is upper bounded by 

$$T_0:= C_0 K_1^{\frac{16\cdot 10^4}{\varepsilon c_0}} [\log \frac{1}{\rho_0}]^{\frac{16\cdot 10^4}{\varepsilon c_0} },$$
and 
the exponent for $q_Tq_{T'}'$ is lower bounded by
$$ 3- K_2^2 \rho_0^{[(\frac{1}{4})^{\frac{16\cdot 10^4}{ \varepsilon c_0}}]}.$$
We take $\delta_0$ small enough so that $\delta_0$ satisfies the ``sufficiently small'' requirement for $\delta$ at Proposition \ref{1631} and  the implied constant $\rho_0=\rho_0(\delta_1)=O(\delta_1)$ satisfies 
\begin{align}\label{1016}
K_2^2 \rho_0^{[(\frac{1}{4})^{\frac{16\cdot 10^4}{ \varepsilon c_0}}]}<\frac{\varepsilon}{2}.
\end{align}
This determines $T_0$ and $L$ subsequently. \par
\eqref{0912} then implies 

\begin{align}\label{10111}
\nonumber&|\pi_q(A)|^{T_0}> q^{6-5\varepsilon}> q^{\varepsilon}|\pi_q(A)|>|\pi_q(A)|^{1+\frac{\varepsilon}{6}} \\
\stackrel{\eqref{0929}}{\Rightarrow} & |\pi_q(A)^3|> |\pi_q(A)|^{1+\frac{\varepsilon}{6T_0}}.
\end{align}

Thus recalling \eqref{519}, \eqref{520} and \eqref{521}, if we set 
$$\delta=\min \{\frac{\varepsilon c_1 \delta_0}{48} ,  \varepsilon c_0\rho_0, \frac{c_1}{2}, \frac{\varepsilon}{6T_0}\},$$
then \eqref{10111} contradicts \eqref{19551}. So \eqref{2320} has to hold for the above choice of $\delta$. We have thus proved Proposition \ref{2132} for all sufficiently large $q$ with exponents of all prime divisors $\geq L$.  \par
For a general $q$, let $q=q_sq_l$, where $$q_s=\prod_{i\in I, n_i\leq L} p_i^{n_i}, \hspace{1cm} q_l=\prod_{i\in I: n_i>L} p_i^{n_i}.$$  \par

If $q_s\leq q^{\frac{\varepsilon}{2}}$, then the modulus $q_s$ can be ignored. We can work with $q_l$ and run the previous argument. \par 

If $q_s>q^{\frac{\varepsilon}{2}}$, take
\begin{align}\label{1459}
\delta < \frac{c_s}{4}
\end{align}
for $c_s=c_s(L)$ the implied spectral gap from Theorem \ref{Golsefidy} for $\Lambda$ and for all integer moduli with exponents of prime divisors $\leq L$. Then 
$$|\pi_{q_s}^*(\chi_S^{(l)})(x)-\frac{1}{|\Gamma_{q_s}|}|\leq \frac{1}{2|\Gamma_{q_s}|}$$
for any $x\in \Gamma_{q_s}$. Since $\pi_{q_s}^*(\chi_S^{(l)})(A)>q^{-\delta}$, we have 

\begin{align}\label{15361}
|\pi_{q_s}(A)|> \frac{|\Gamma_{q_s}|}{2}q^{-\delta}>q_s^{6-\frac{2\delta}{\varepsilon}-}>q_s^{6-2\rho_0}
\end{align}
if we take 
\begin{align} \label{1944}
\delta < \min\{\frac{c_s}{4}, \frac{{\varepsilon \rho_0}}{2} \}.
\end{align}
where $\rho_0$ is given after \eqref{1144}. \par

In case $q_s>q^{1-\frac{\varepsilon}{2}}$,  \eqref{15361} implies the third assumption in \eqref{1131} is void and so Proposition \ref{2132} automatically holds. 

In case $q_s\leq q^{1-\frac{\varepsilon}{2}}$, we can run the argument in the large exponent case to grow the modulus $(q_s, q_s)$ to $(q_1^*, q_2^*)$, where $q_1^*, q_2^*> q^{1-\frac{\varepsilon}{2}}$ and the projection of a product set of $A$ to $(q_1^*, q_2^*)$ is very large.

\section{Proof of Proposition \ref{2132} for $\Lambda=\text{SL}_2(\mathbb Z)\ltimes \mathbb Z^2$ \label{ns}}
The proof of Proposition \ref{2132} for the case $\Lambda=\text{SL}_2(\mathbb Z)\ltimes \mathbb Z^2$ is similar to the case $\Lambda=\text{SL}_2(\mathbb Z)\times  \text{SL}_2(\mathbb Z)$. We give a sketch.  \par 
Let $\mathbb P_0: \Lambda\rightarrow \Gamma$ be the projection to the semisimple part and $\mathbb P_u: \Lambda\rightarrow \mathbb Z^2$ the projection to the unipotent part. \par
We assume $q=\prod_{i\in I} p_i^{n_i}$ with each $n_i$ large. Assume \eqref{1131} holds but \eqref{2320} fails. Then by taking $\delta$ sufficiently small, we can apply Proposition \ref{1631} and Proposition \ref{glue} iteratively to obtain
\begin{align} \label{448}
 G:=\Gamma(q_1')/\Gamma(q_1) \subset \mathbb P_0(A^{O((\log \frac{1}{\rho})^2)} ),
\end{align}
where $\mathbb P_0$ is the projection to the semisimple part, $\rho$ is a small quantity, $q_1=q^{I_1}$ for some $I_1\subset I$ with $q_1>q^{1-\frac{\varepsilon}{2}}$, and for any $i\in I_1$, $v_{p_i}(q_1')<\rho v_{p_i}(q_1)$.  \par

Let $B\subset A^{O((\log \frac{1}{\rho})^2)} $ be a set of representatives of $G$ implied by \eqref{448}. Let
$\psi: G\rightarrow B$ such that $\mathbb P_0\circ \psi : G\rightarrow G$ is the identity map. \par
For each $i\in I_1$, consider $\psi_i=\pi_{p_i^{\{\theta n_i\}}}\circ\mathbb P_u\circ \psi$ for another small quantity $\theta$ which is larger than $\rho$. 

According to Proposition \ref{BZq_coro6.9}, there are two scenarios: \par

{\bf{Event }1}: Define
$$\mathcal G_j=\{(x,y)\in G\times G|  \psi_i(xy)\neq \psi_i(x)\psi_i(y)\}$$
We have, \begin{align}\label{1746}
|\mathcal G_j|> 10^{-4}|G|^2.
\end{align}

{\bf{Event }2}:
There is a subset $S_i\in G$, $|S_i|\geq \frac{99}{100}|G|$ such that $\psi_i\equiv h_i$ over ${S_i}$, where $h_i$ is  a homomorphism from $G$ to $\Gamma/\Gamma(p_i^{[n_i\theta]})$.    \par

We split $I_1=I_2\sqcup I_3$, where $I_2$ consists of all $i$ such that Event 1 happens, and $I_3$ is the complement of $I_2$. We divide our analysis into two cases according to whether $q^{I_2}$ is large or not.

\noindent {\bf Case 1:}  $q^{I_2}\geq(q^{I_1})^{\frac{1}{2}}$.  We can find $(1, w)\in B^2B^{-1}$ such that $w\in\mathbb Z^2$ satisfies $v_{p_i}(w)<\theta n_i$ for all $i\in I_2'\subset I_2$, where $q^{I_2'}>(q^{I_2})^{\frac{1}{2}\cdot 10^{-4}}$. Conjugate $(1,w)$ by $B$ using the formula $$(g, x)\cdot (1, w)\cdot  (g, x)^{-1}=(1, g(w)),$$ we obtain a set $B_1$ such that $\mathbb P_0(B_1)=1$ and $\mathbb P_u(B_1) (\mod q^{I_2'})$ is a very large subset of $(\mathbb Z/q^{I_2'}\mathbb Z)^2$. Therefore, $\pi_{q, q^{I_2'}}(B\cdot B_1) $ is a very large subset of $\Gamma_{q}\ltimes (\mathbb Z/q^{I_2'}\mathbb Z)^2$.  \par

\noindent {\bf Case 2:} $q^{I_3}\geq(q^{I_1})^{\frac{1}{2}}$.  In this case $\pi_{q, q^{I_3}}\circ \psi$ agrees with a homomorphism $h: G\rightarrow \Gamma_q\ltimes (\mathbb Z/(q^{I_3'})^{\{\theta\}}\mathbb Z)^2$ on a very large subset $T$ of $G$, where $q^{I_3'}>(q^{I_3})^{\frac{99}{200}}$, and $T$ boundedly generates a large subgroup $G'$ of $G$. According to Proposition \ref{1744}, 
$$\pi_{(q^{I_3'})^{\{\frac{\theta}{4}\}}, (q^{I_3'})^{\{\frac{\theta}{4}\}}}(h(G))\subset \pi_{(q^{I_3'})^{\{\frac{\theta}{4}\}}, (q^{I_3'})^{\{\frac{\theta}{4}\}}} (H_{u,v})$$
for some $u,v\in\mathbb Q$. \par
Next, we apply Proposition \ref{decay3} to find $(g_0,w_0)$ from $A\cdot A$ such that $g_0\in G'$ and $$\pi_{p_i^{[\frac{1}{4}\theta n_i]}}(g_0,w_0) \not\in H_{u,v}(\mod p_i^{[\frac{1}{4}\theta n_i]}), \hspace{0.5cm}\forall i\in I_3'',$$
where $I_3''\subset I_3'$ and $q^{I_3''}\geq (q^{I_3'})^{\frac{1}{2}}$. \par
On the other hand, since $T$ boundedly generates $G'$, take $(g_0, w_1)\in (\psi(T))^{O(1)}\subset B^{O(1)}$. Let $(1,w')=(g_0, w_0)(g_0,w_1)^{-1}$. Then $$w'\not\equiv 0 (\mod p_j^{[\frac{1}{4}\theta n_j ]}), \forall j\in I_3''.$$ 
With $(1,w')$ playing the same role as $(1, w)$ in the previous case, we can argue in the same way to construct a large subset of $\Gamma_q\ltimes (\mathbb Z/q^{I_3''}\mathbb Z)$. \par
In either case, we manage to construct a subset from a bounded product of $\{A\cup B\}$ whose projection to $\Gamma_q\ltimes (\mathbb Z/q'\mathbb Z)^2$ is very large for some large $q'\Vert q$, starting with a set $B$ only known to have large projection to $\Gamma_q\ltimes \{\vec{0}\}$. \par
We can keep iterating the above argument to grow the modulus of the unipotent factor, and thus prove Proposition \ref{2132}.

\vspace{1cm}

\bibliographystyle{alpha}

\bibliography{XZ}

\begin{thebibliography}{BGGT15}

\bibitem[AM85]{AM85}
Noga Alon and Vitali~D. Milman.
\newblock $\lambda$1, isoperimetric inequalities for graphs, and
  superconcentrators.
\newblock {\em Journal of Combinatorial Theory, Series B}, 38(1):73--88, 1985.

\bibitem[BFLM11]{BFLM11}
Jean Bourgain, Alex Furman, Elon Lindenstrauss, and Shahar Mozes.
\newblock Stationary measures and equidistribution for orbits of nonabelian
  semigroups on the torus.
\newblock {\em Journal of the American Mathematical Society}, 24(1):231--280,
  2011.

\bibitem[BG08a]{BG08a}
Jean Bourgain and Alex Gamburd.
\newblock Expansion and random walks in {$\text{SL}_d({\mathbb Z}/p^n{\mathbb
  Z})$}: I.
\newblock {\em Journal of the European Mathematical Society}, 10(4):987--1011,
  2008.

\bibitem[BG08b]{BG08}
Jean Bourgain and Alex Gamburd.
\newblock Uniform expansion bounds for {C}ayley graphs of {${\rm SL}_2(\Bbb
  F_p)$}.
\newblock {\em Ann. of Math. (2)}, 167(2):625--642, 2008.

\bibitem[BG09]{BG09}
Jean Bourgain and Alex Gamburd.
\newblock Expansion and random walks in {$\text{SL}_d({\mathbb Z}/p^n{\mathbb
  Z})$}, {II}.
\newblock {\em J. Eur. Math. Soc.(JEMS)}, 11(5):1057--1103, 2009.

\bibitem[BGGT15]{BGGT15}
Emmanuel Breuillard, Ben~J. Green, Robert~M. Guralnick, and Terence Tao.
\newblock Expansion in finite simple groups of lie type.
\newblock {\em Journal of the European Mathematical Society}, 17(6):1367--1434,
  2015.

\bibitem[BGS10]{BGS10}
Jean Bourgain, Alex Gamburd, and Peter Sarnak.
\newblock Affine linear sieve, expanders, and sum-product.
\newblock {\em Invent. Math.}, 179(3):559--644, 2010.

\bibitem[Bou08]{Bou08}
Jean Bourgain.
\newblock The sum-product theorem in {${\mathbb Z}/q{\mathbb Z}$} with $q$
  arbitrary.
\newblock {\em Journal d'Analyse Math{\'e}matique}, 106(1):1, 2008.

\bibitem[BS91]{BurgerSarnak1991}
M.~Burger and P.~Sarnak.
\newblock Ramanujan duals {I}{I}.
\newblock {\em Invent. Math}, 106:1--11, 1991.

\bibitem[BV12]{BV12}
Jean Bourgain and P{{\'e}}ter~P. Varj{{\'u}}.
\newblock Expansion in {$\text{SL}_d({\mathbb Z}/q{\mathbb Z}),\,q$} arbitrary.
\newblock {\em Invent. Math.}, 188(1):151--173, 2012.

\bibitem[BY91]{BY91}
Carlos~A. Berenstein and Alain Yger.
\newblock Effective {B}\'ezout identities in ${Q}[z_1,..., z_n]$.
\newblock {\em Acta Math.}, 166(3):69--120, 1991.

\bibitem[Clo03]{Clozel2003}
Laurent Clozel.
\newblock D\'emonstration de la conjecture {$\tau$}.
\newblock {\em Invent. Math.}, 151(2):297--328, 2003.

\bibitem[Gow08]{Go08}
William~T. Gowers.
\newblock Quasirandom groups.
\newblock {\em Combinatorics, Probability and Computing}, 17(3):363--387, 2008.

\bibitem[GS24]{GS24}
Alireza~Salehi Golsefidy and Srivatsa Srinivas.
\newblock Random walks on direct products of groups.
\newblock {\em Journal of the European Mathematical Society}, 2024.

\bibitem[GV12]{GV12}
A.~Salehi Golsefidy and P{{\'e}}ter~P. Varj{{\'u}}.
\newblock Expansion in perfect groups.
\newblock {\em Geom. Funct. Anal.}, 22(6):1832--1891, 2012.

\bibitem[HdS19]{He19}
Weikun He and Nicolas de~Saxc\'e.
\newblock Linear random walks on the torus.
\newblock {\em arXiv preprint arXiv: 1910.13421}, 2019.

\bibitem[HdS21]{HS21}
Weikun He and Nicolas de~Saxc{\'e}.
\newblock Trou spectral dans les groupes simples.
\newblock {\em arXiv preprint arXiv:2103.06679}, 2021.

\bibitem[Hel08]{He08}
H.~A. Helfgott.
\newblock Growth and generation in {${\rm SL}_2(\Bbb Z/p\Bbb Z)$}.
\newblock {\em Ann. of Math. (2)}, 167(2):601--623, 2008.

\bibitem[Mar73]{Ma73}
Grigorii~Aleksandrovich Margulis.
\newblock Explicit constructions of concentrators.
\newblock {\em Problemy Peredachi Informatsii}, 9(4):71--80, 1973.

\bibitem[PS16]{PS16}
L{\'a}szl{\'o} Pyber and Endre Szab{\'o}.
\newblock Growth in finite simple groups of lie type.
\newblock {\em Journal of the American Mathematical Society}, 29(1):95--146,
  2016.

\bibitem[Sel65]{Sel65}
Atle Selberg.
\newblock On the estimation of fourier coefficients of modular forms.
\newblock In {\em Proceedings of Symposia in Pure Mathematics}, pages 1--15.
  American Mathematical Society, 1965.

\bibitem[SG17]{SG17}
Alireza Salehi~Golsefidy.
\newblock Super-approximation, i:-adic semisimple case.
\newblock {\em International Mathematics Research Notices},
  2017(23):7190--7263, 2017.

\bibitem[SG19]{SG19}
Alireza Salehi~Golsefidy.
\newblock Super-approximation, {II}: the $ p $-adic case and the case of
  bounded powers of square-free integers.
\newblock {\em Journal of the European Mathematical Society}, 21(7):2163--2232,
  2019.

\bibitem[SG20]{SG20}
Alireza Salehi~Golsefidy.
\newblock {S}um-product phenomena: $p$-adic case.
\newblock {\em Journal d'Analyse Math{\'e}matique}, 142(2):349--419, 2020.

\bibitem[Sha09]{Sh09}
Aner Shalev.
\newblock Word maps, conjugacy classes, and a noncommutative waring-type
  theorem.
\newblock {\em Annals of Mathematics}, pages 1383--1416, 2009.

\bibitem[Tit72]{Tits1972}
Jacques Tits.
\newblock Free subgroups in linear groups.
\newblock {\em Journal of algebra}, 20(2):250--270, 1972.

\bibitem[TZ23]{TZ23a}
Jincheng Tang and Xin Zhang.
\newblock Sum-product in quotients of rings of algebraic integer.
\newblock {\em arXiv preprint arXiv: 2308.08867}, 2023.

\bibitem[Var12]{VA12}
P{\'e}ter~P. Varj{\'u}.
\newblock Expansion in {$SL_d(O_K/I)$, $I$} square-free.
\newblock {\em J. Eur. Math. Soc.(JEMS)}, 14(1):273--305, 2012.

\end{thebibliography}

\end{document}